\pdfoutput=1
\documentclass[11pt]{article}

\usepackage{preamble}

\begin{document}

\title{Clones of Borel Boolean Functions}
\author{Ruiyuan Chen and Ilir Ziba}
\date{}
\maketitle

\begin{abstract}
We study the lattice of all Borel clones on $2 = \{0,1\}$: classes of Borel functions $f : 2^n -> 2$, $n \le \omega$, which are closed under composition and include all projections.
This is a natural extension to countable arities of Post's 1941 classification of all clones of finitary Boolean functions.
Every Borel clone restricts to a finitary clone, yielding a ``projection'' from the lattice of all Borel clones to Post's lattice.
It is well-known that each finitary clone of affine mod $2$ functions admits a unique extension to a Borel clone.
We show that over each finitary clone containing either both $\wedge, \vee$, or the 2-out-of-3 median operation, there lie at least 2 but only finitely many Borel clones.
Over the remaining clones in Post's lattice, we give only a partial classification of the Borel extensions, and present some evidence that the full structure may be quite complicated.
\let\thefootnote=\relax
\end{abstract}

\tableofcontents

% trickery to get figure to appear before Section 1
\addtocounter{section}{1}
\begin{figure}[p]
\centering
\hbox{\hspace{-1ex}\includegraphics{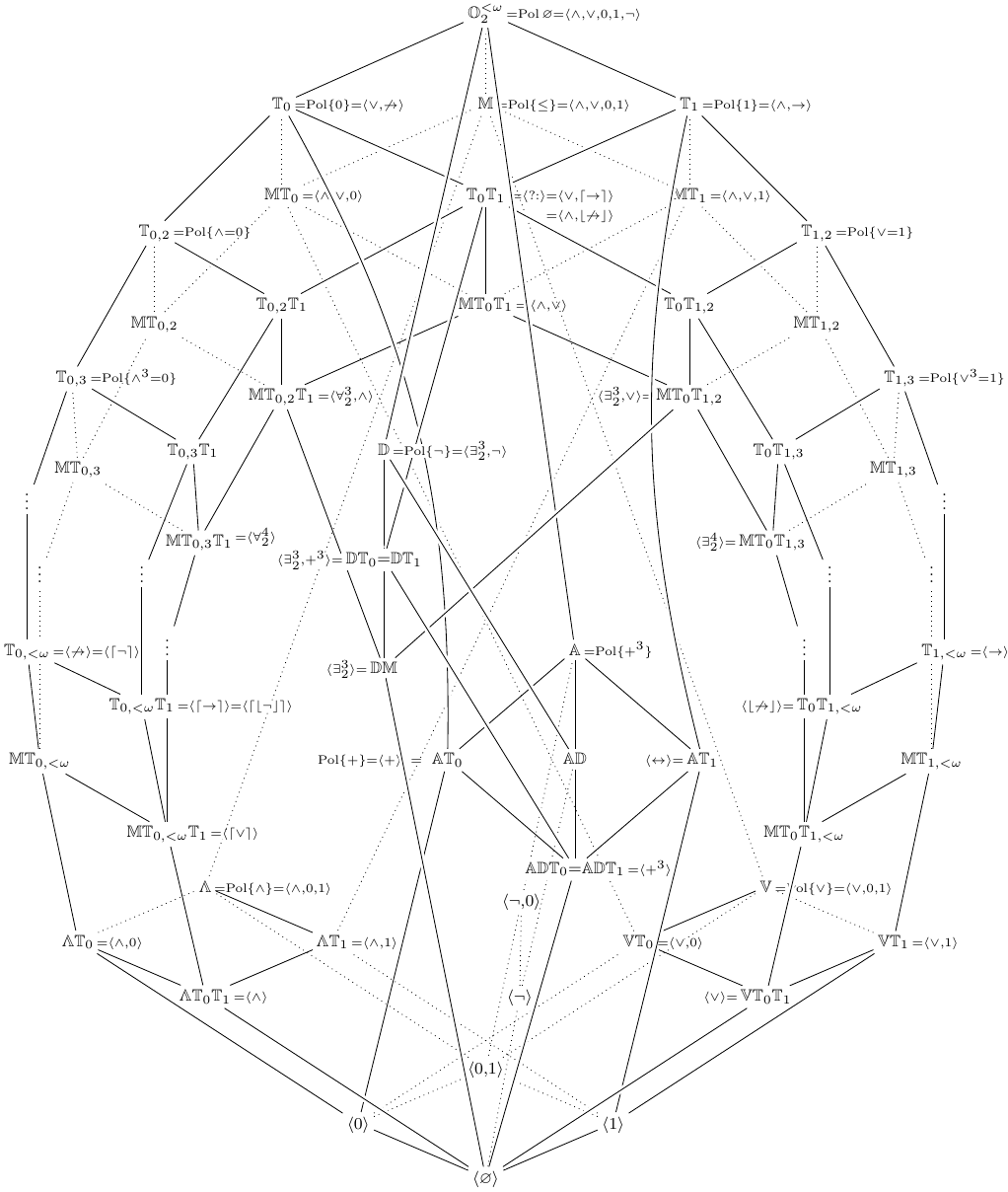}}
\caption{Post's lattice $\Clo^{<\omega}{2}$ of all clones of finitary Boolean functions $2^n -> 2$.
See \cref{sec:post} for the definitions of these clones, functions, and relations.
(Here, all clones are implicitly restricted to finitary.)}
\label{fig:post}
\end{figure}
\addtocounter{section}{-1}

\section{Introduction}
\label{sec:intro}
\addtocounter{equation}{1}

It is a standard exercise in propositional logic that the logical connectives $\wedge, \vee, \neg$ can express all Boolean functions $2^n -> 2$ of finite arity; in fact $\wedge, \neg$ already suffice.
On the other hand, $\wedge, \vee$ do not suffice to generate all Boolean functions, for they are both ``positive'' connectives; and in fact together with the constant truth values $1, 0$, they generate precisely all monotone Boolean functions.
Leaving out the constants $1, 0$, we obtain only the monotone functions fixing $1, 0$, etc.
In 1941, Post \cite{Post} gave a complete classification of all possible classes of finitary Boolean functions $2^n -> 2$ generated by a class of functions under composition, known as \defn{clones} on $2$; these may be thought of as all ``sublogics'' of classical finitary $2$-valued propositional logic.
The resulting countable lattice of all clones on $2$, known as \defn{Post's lattice}, is depicted in \cref{fig:post}.
For background on Post's lattice and clone theory, see \cite{FMMT}, \cite{Lau}, \cite{Szendrei}.

In this paper, we consider the analogous problem of classifying all clones of \emph{countable} arity Boolean functions $2^n -> 2$, $n \le \omega$; these may be thought of as ``sublogics'' of countably infinitary propositional logic.
A Boolean function $2^n -> 2$ may be obtained as a composition of the countable connectives $\bigwedge, \bigvee$ and $\neg$ iff it is a \defn{Borel} function, i.e., the indicator function of a Borel set $A \subseteq 2^n$.
Let $\Op2^\Borel \subseteq \bigsqcup_{n \le \omega} 2^{2^n}$ denote the clone of all Borel functions.
We will mostly restrict attention to subclones of $\Op2^\Borel$, thereby ruling out pathological ``connectives'' such as nonprincipal ultrafilters which cannot be explicitly defined.
Note that definability constraints such as being Borel can interact with the algebraic structure on $2$ in intricate ways.
For example, a Borel function $2^n -> 2$ of countable arity which is affine over $\#Z/2\#Z$ can in fact depend only on finitely many variables; but this is no longer true for any condition weaker (in a precise sense; see below) than affinity.

Even restricted to the Borel clones on $2$, we do not obtain a complete classification in this paper.
However, we obtain a classification of a large ``region'' of the lattice of all Borel clones, as well as a partial classification of the remaining ``regions'' along with some indications that they may be difficult to fully classify.
In order to state our results more precisely, we now give an overview of our approach, which is based on Post's classification of the finitary clones.

\subsection{The bundle of Borel clones}

Let $\Op2^{<\omega} := \bigsqcup_{n < \omega} 2^{2^n}$ denote the clone of all finitary functions.
Thus Post's lattice \ref{fig:post} consists of all finitary subclones of $\Op2^{<\omega}$, while we seek to classify all countable-arity subclones of $\Op2^\Borel$.
Given such a subclone $F \subseteq \Op2^\Borel$, we may restrict to the finitary functions $F \cap \Op2^{<\omega}$ within it.
We get a map
\begin{align*}
\yesnumber
\label{eq:intro-bundle}
\Clo^\Borel{2} := \{\text{Borel clones on $2$}\} &-->> \{\text{finitary clones on $2$}\} =: \Clo^{<\omega}{2} \\
F &|--> F \cap \Op2^{<\omega},
\end{align*}
which turns out to be a complete lattice homomorphism (see \cref{thm:pancake} and \cref{eq:pancake-borel}).
Thus, we may regard the lattice $\Clo^\Borel{2}$ of Borel clones on $2$ as a ``fiber bundle'' over Post's lattice $\Clo^{<\omega}{2}$, depicted in \cref{fig:post-borel}; and the classification problem for Borel clones decomposes into, for each finitary clone $G$ in Post's lattice, classifying the ``fiber'' over $G$, of all Borel clones $F$ with finitary restriction $G$.
When $G$ is defined as all the finitary functions preserving certain finitary relations, then such $F$ are precisely the Borel clones of functions preserving the same relations, and containing all the finitary such functions $G$.

To illustrate this, consider the following clones, which are the maximal nodes in Post's lattice (defined in more detail in \cref{sec:post}):
\begin{itemize}
\item  $\Mono :=$ all monotone (i.e., $\le$-preserving) functions $2^n -> 2$.
\item  $\Cons{c}{1} :=$ all functions $2^n -> 2$ preserving the constant $c \in \{0,1\}$.
\item  $\Dual :=$ all functions $2^n -> 2$ equal to their own de~Morgan dual.
\item  $\Aff :=$ all functions $2^n -> 2$ affine over $\#Z/2\#Z$.
\end{itemize}
We denote the finitary, respectively Borel, versions of these clones by a superscript $^{<\omega}$, resp., $^\Borel$.
Thus, for example, $\Mono^{<\omega}$ is the clone of finitary monotone functions, generated by $\wedge, \vee, 0, 1$.
Then the Borel clones $F$ in the fiber of the bundle \cref{eq:intro-bundle} over $\Mono^{<\omega}$ are those such that $\wedge, \vee, 0, 1 \in F \subseteq \Mono^\Borel$.
There is a greatest such $F$, namely $\Mono^\Borel$, as well as a least, namely the Borel clone generated by $\wedge, \vee, 0, 1$; these are distinct since the former contains the countable disjunction (join/supremum) $\bigvee$.
In contrast, in the fiber over $\Aff^{<\omega}$, there lies only a single Borel clone, by the aforementioned fact that Borel affine maps can depend on only finitely many variables.
It follows that the same holds over each subclone of $\Aff^{<\omega}$ in Post's lattice, e.g., $\Aff\Dual^{<\omega} := \Aff^{<\omega} \cap \Dual^{<\omega}$.

\subsection{Main results}

We are able to completely classify the Borel clones lying over a node near the top of Post's lattice:

\begin{theorem}[see \cref{sec:borel-topcube}]
\label{thm:intro-topcube-dual}
\leavevmode
\begin{enumerate}[label=(\alph*)]
\item \label{thm:intro-topcube-dual:topcube}
Over each of the 8 finitary clones in the ``cube'' at the top of Post's lattice \ref{fig:post}, consisting of all intersections of combinations of $\Mono, \Cons01, \Cons11$, there lie at least 2 but only finitely many Borel clones restricting to that finitary clone, namely:
2 Borel clones over $\Op2^{<\omega}$;
3 over $\Cons01^{<\omega}, \Cons11^{<\omega}$ each;
5 over $\Cons01\Cons11^{<\omega}$;
4 over $\Mono^{<\omega}$;
6 over $\Mono\Cons01^{<\omega}, \Mono\Cons11^{<\omega}$ each; and
9 over $\Mono\Cons01\Cons11^{<\omega}$.
\item \label{thm:intro-topcube-dual:dual}
Over the finitary clones $\Dual^{<\omega}, \Dual\Cons01^{<\omega}, \Dual\Mono^{<\omega}$, there lie 2, 3, 3 Borel clones respectively.
\item \label{thm:intro-topcube-dual:aff}
Over each finitary clone $F$ in Post's lattice \ref{fig:post}, there lies only 1 Borel clone iff $F \subseteq \Aff$.
\end{enumerate}
\end{theorem}

\Cref{fig:post-borel} (shaded regions) depicts those fibers of the bundle $\Clo^\Borel{2} ->> \Clo^{<\omega}{2}$ from \cref{eq:intro-bundle} in which we get a complete classification of the Borel clones.
The proofs are spread out over several results in \cref{sec:borel-topcube}; see there for more detailed pictures of each of the fibers individually.

Over the remaining finitary clones in Post's lattice \ref{fig:post}, we do not get a complete classification of the Borel clones.
Instead, we exhibit a wide variety of complex behaviors among the Borel clones.
In order to state these, we recall some more clones from Post's lattice:
\begin{itemize}
\item  $\Meet :=$ all indicator functions $f : 2^n -> 2$ of filters $f^{-1}(1) \subseteq 2^n$, or $f = 0$.
\item  $\Cons0k :=$ all indicator functions $f : 2^n -> 2$ of sets $f^{-1}(1) \subseteq 2^n$ with the $k$-ary intersection property (any $k$ bit strings in $f^{-1}(1)$ have bitwise conjunction $\ne \vec{0}$); thus $\Cons01 = \Cons0{{1}}$.
\item  $\Cons0{<\omega} := \bigcap_{k < \omega} \Cons0k =$ all indicator functions of sets with the finite intersection property.
\item  $\Join, \Cons1k, \Cons1{<\omega}$ are the de~Morgan duals of these, concerning ideals/disjunctions.
\end{itemize}
As before, e.g., $\Cons0k^{<\omega}$ denotes the clone of all finitary such functions, while $\Cons0k^\Borel$ denotes the corresponding Borel clone.
Note that $\Cons0{<\omega}^{<\omega}$ thus consists of all finitary functions which are bounded above by some particular variable (by considering a conjunction of strings with a single $0$); whereas $\Cons0{<\omega}^\Borel \supseteq \Meet^\Borel$, which includes all Borel filters $\subseteq 2^\omega$, is much more complicated.

\begin{theorem}[see \crefrange{sec:borel-T0inf}{sec:borel-T0k}, especially \cref{thm:post-borel-T0k,thm:post-borel-Meet,ex:summable}]
\label{thm:intro-T0k}
\leavevmode
\begin{enumerate}[label=(\alph*)]
\item \label{thm:intro-T0k:T0k}
Over each of finitary clones $\Cons0k, \Cons0k\Cons11, \Mono\Cons0k, \Mono\Cons0k\Cons11$ in the left ``side tube'' of Post's lattice \ref{fig:post}, for $k < \omega$, there lie at least $2^k+1, 2^k+k+2, 2^k+4, 2^k+k+6$ Borel clones, respectively.
\item \label{thm:intro-T0k:T0inf}
Over each of the finitary clones $\Cons0{<\omega}$ and its intersections with $\Mono, \Cons11$ (at the ``base of the side tube''), there lie at least countably infinitely many Borel clones.
\item \label{thm:intro-T0k:meet}
Over each of the finitary clones $\Meet, \Meet\Cons11, \Meet\Cons01, \Meet\Cons01\Cons11$ and its intersections with $\Cons01, \Cons11$ (below the ``side tube''), there lie precisely 3, 3, 4, 4 Borel clones $F \subseteq 2^{2^{\le\omega}}$ respectively which are ``countably closed'' in $2^{2^\omega}$, i.e., contain all functions which agree at any countably many input strings $\vec{x}_0, \vec{x}_1, \vec{x}_2, \dotsc \in 2^\omega$ with some function in $F$.
However, there also exist other ``non-countably closed'' clones of Borel filters.
\end{enumerate}
Similarly for the de~Morgan duals of these clones.
\end{theorem}

\Cref{fig:post-borel-T0k} (shaded regions) depicts the fibers of the bundle \cref{eq:intro-bundle} mentioned above, in which we get only a partial classification of the Borel clones (these clones are defined in \cref{sec:borel-T0k}); see also \cref{fig:post-borel-T0inf,fig:post-borel-Meet,fig:post-borel-L0kt}.
As these pictures indicate, the lower bounds on the numbers of Borel clones in \cref{thm:intro-T0k}\cref{thm:intro-T0k:T0k,thm:intro-T0k:T0inf} are merely the numerical counts of detailed order-theoretic structures on these fibers.
For instance, the constant terms in these lower bounds count the fully classified lower (solid-shaded in \ref{fig:post-borel-T0k}) portions of each fiber, with $2, 3, 4, 6$ elements respectively, which are isomorphic to four of the fibers in \cref{thm:intro-topcube-dual}\cref{thm:intro-topcube-dual:topcube}.
On the other hand, the upper, hatch-shaded ``cobweb'' portions of \ref{fig:post-borel-T0k} contribute $2^k$ distinct Borel clones in each fiber, with a recursively generated order-structure; but we cannot rule out the existence of other Borel clones between these.

\Cref{thm:intro-T0k}\cref{thm:intro-T0k:meet} shows that in a precise sense, it is ``difficult'' to fully classify all Borel clones over $\Meet^{<\omega}$ and its variants.
Recall that by standard clone theory, every clone $F \subseteq \bigsqcup_n 2^{2^n}$ can be defined as all Boolean functions which preserve some given set of $k$-ary relations $R \subseteq 2^k$ for various $k$, called \defn{polymorphisms} of those relations;
this includes $R = {\le} \subseteq 2^2$ (yielding the monotone functions $\Mono$), ${\ne} \subseteq 2^2$ (yielding the self-dual functions $\Dual$), etc.
The \emph{countably closed} clones $F$ as in \ref{thm:intro-T0k}\cref{thm:intro-T0k:meet} are precisely those which can be defined as the polymorphisms of \emph{countable-arity} relations; for example, polymorphisms of the $\omega$-ary convergence relation ``$\lim_{i -> \infty} x_i = 0$'' yield the functions $f : 2^n -> 2$ which fix and are continuous at $\vec{0}$, denoted $\Limm011$ near the tops of \cref{fig:post-borel,fig:post-borel-T0k}.
Thus, \cref{thm:intro-T0k}\cref{thm:intro-T0k:meet} says that on the one hand, there are only a few Borel clones of filters which are definable by countable-arity relations; on the other hand, such relations cannot distinguish all distinct Borel clones.

\begin{figure}[p]
\centering
\includegraphics{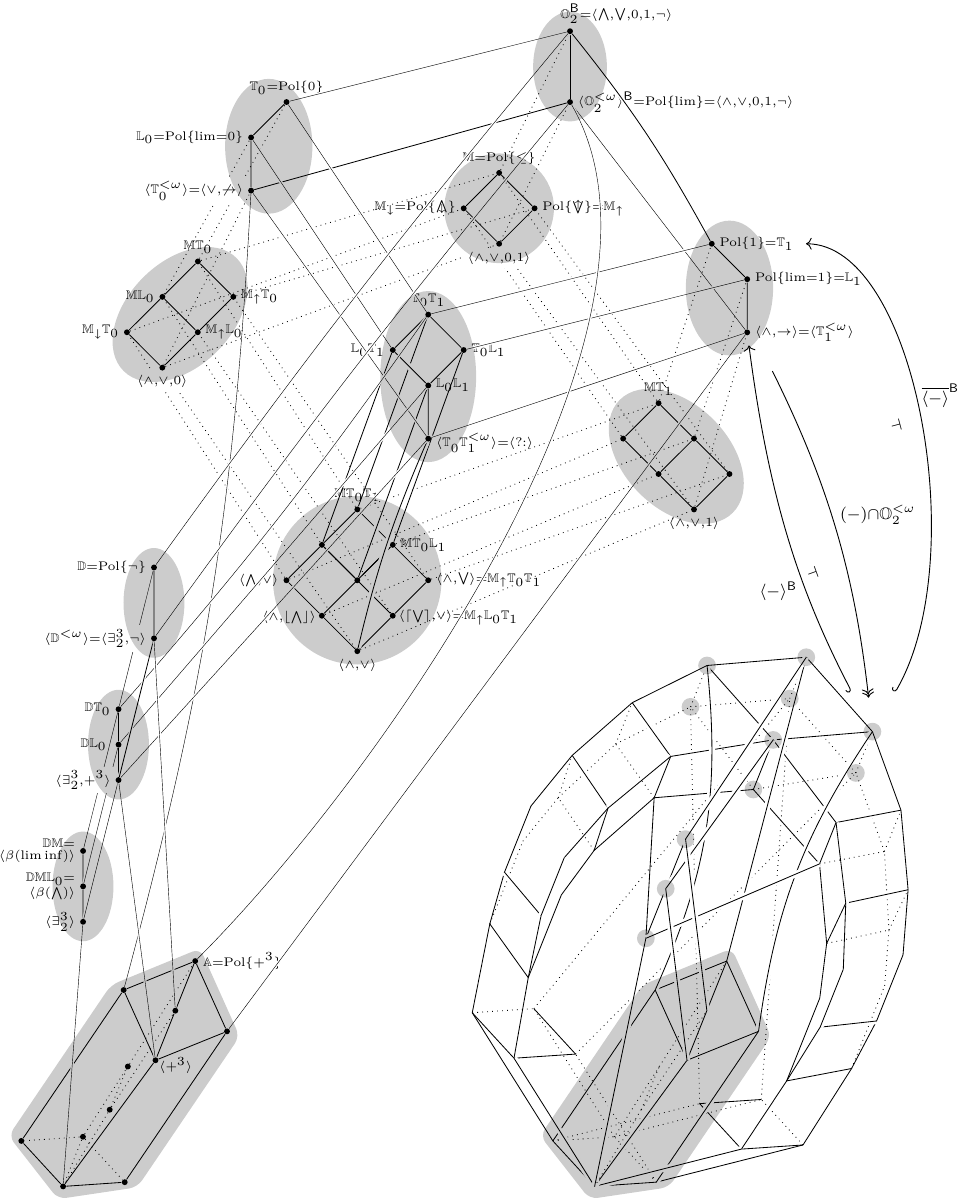}
\caption{Bundle of Borel clones $\Clo^\Borel{2} ->> \Clo^{<\omega}{2}$ projecting to Post's lattice \ref{fig:post} via finitary restriction, with fibers (shaded blobs) in which a complete classification of the Borel clones is known.}
\label{fig:post-borel}
\end{figure}

\begin{figure}[p]
\centering
\includegraphics{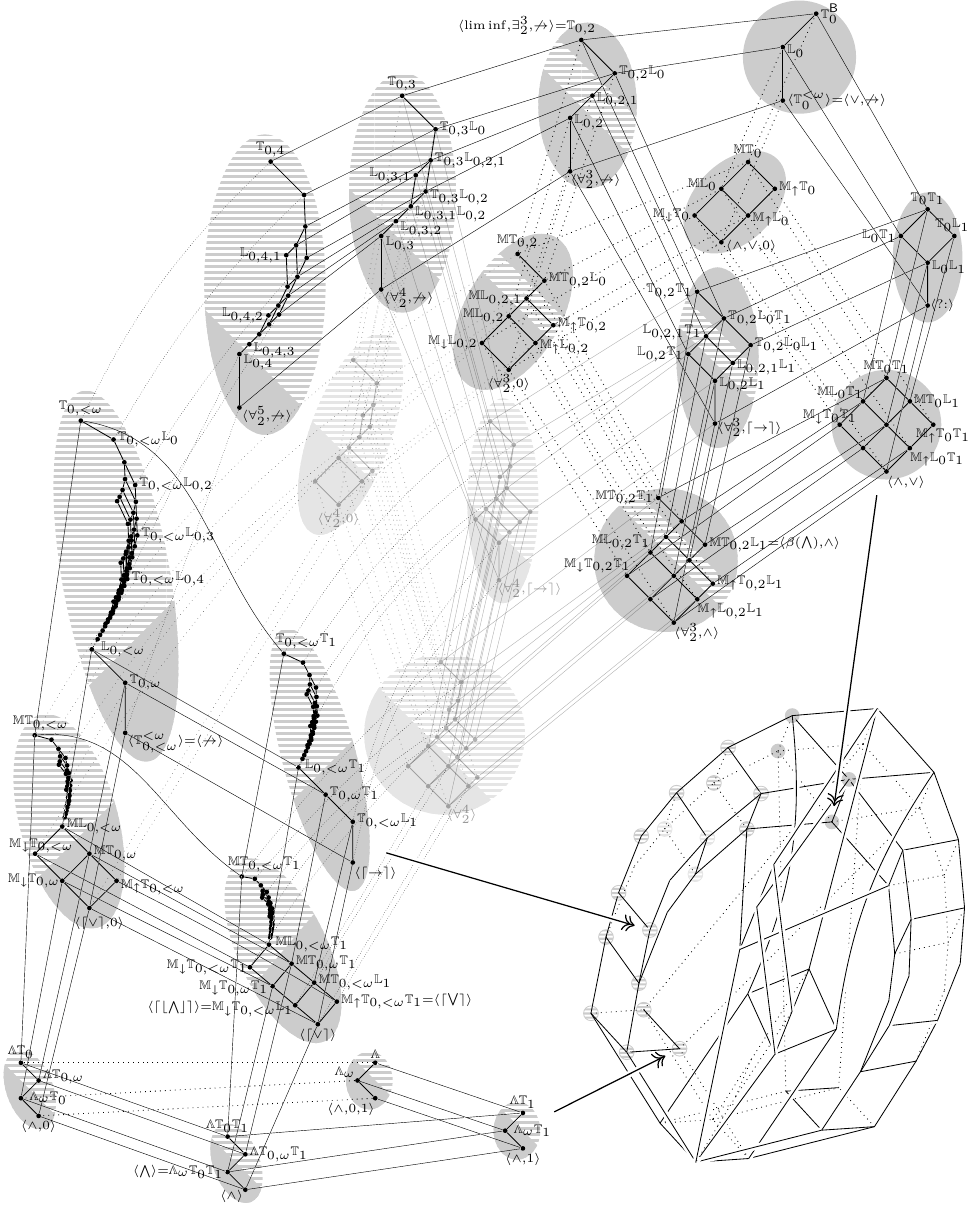}
\caption{Fibers of the bundle \ref{fig:post-borel} in which only a partial classification of the Borel clones is known (solid-shaded = fully classified; hatch-shaded = at least the shown nodes are known to be distinct).}
\label{fig:post-borel-T0k}
\end{figure}

We now briefly outline our proof strategy for the aforementioned positive classification results.
The combinatorial core consists of a handful of simple ``Wadge's lemma''-type dichotomies, showing that every Borel clone $F$ must either be contained in a specific clone, or else contain a ``minimally complex'' function outside of that clone; see \cref{thm:post-wadge,thm:post-forall2,thm:kahane,thm:fkt-gen}.
These dichotomies are used to bootstrap several abstract structural mappings, showing that large regions of the lattice of Borel clones $\Clo^\Borel{2}$ are isomorphic or embed into each other; see \cref{fig:post-mod} and \cref{thm:post-T0inf-mod,thm:post-T0inf-down,thm:post-D-beta-mod,thm:post-T1-mod,thm:post-T0k-lim1-mod}.
By applying these isomorphisms repeatedly, we then transport the core dichotomies across various regions of the lattice $\Clo^\Borel{2}$, in order to resolve it into the pieces shown in \cref{fig:post-borel,fig:post-borel-T0k}.

\subsection{Related work and future directions}

An unusual feature of our classification shown in \cref{fig:post-borel-T0k} is that among the Borel clones over $\Cons02$ and its variants, we are able to fully characterize the bottommost as well as topmost clones, but not the intermediate clones $\Limm021$ and its variants.
This has to do with the special connection between the clone $\Mono\Cons02$ and the self-dual monotone functions $\Dual\Mono$ (the function $\forall^3_2 = \exists^3_2 : 2^3 -> 2$ is the 2-out-of-3 median operation).
It remains to be seen whether similar ideas may be used to characterize upper regions (the ``cobwebs'' in \ref{fig:post-borel-T0k}) of the $\Cons0k$ Borel clones for $k \ge 3$.

The \cref{ex:summable} of a Borel clone not definable by countable-arity relations mentioned in \cref{thm:intro-T0k}\cref{thm:intro-T0k:meet} essentially follows from a result of Kanovei--Reeken~\cite{KanoveiReeken} on Borel equivalence relations and ideals.
This result uses deep set-theoretic tools such as forcing and Hjorth's turbulence theory \cite{Hjorth}, while showing something much stronger (namely, Borel non-reducibility of equivalence relations) than is needed for the application to Borel clones (namely, Rudin--Keisler non-reducibility between Borel filters).
We expect there to be other examples of distinct Borel clones inseparable by countable-arity relations; a better understanding of the connections with the theory of Borel ideals and filters (see \cite{Solecki}, \cite{Hrusak}, \cite{Kanovei}) may help with such a pursuit.

More generally, we are pleasantly surprised that a large part of the lattice of Borel clones seems to be classifiable using only abstract lattice-theoretic reasoning plus a few simple combinatorial dichotomies, as described in the preceding subsection.
In particular, we make no use in this paper of more powerful techniques such as determinacy, forcing, or effective descriptive set theory (see \cite{Mdst}, \cite{Kanovei}).
This is especially surprising given the use of such techniques in a structure theory which is conceptually related to that of Borel clones: the Wadge hierarchy \cite{Wadge}.
Recall that a \defn{Wadge class} (on $2^\omega$) is essentially a family $F$ of functions $2^\omega -> 2$ which is closed under \emph{right-}composition with arbitrary continuous functions.
The lattice (essentially a well-order) of all Wadge classes of Borel functions has been completely described \cite{Wadge}, \cite{Louveau}.

We note that the classification of Borel clones in this paper is formally ``orthogonal'' to the Wadge hierarchy: every Wadge class contains all finitary (i.e., continuous) functions, whereas a Borel clone containing all finitary functions and a single discontinuous function must be the entirety of $\Op2^\Borel$.
Rather, we expect that the machinery we have developed may be useful in future studies of infinitary clones of ``definable'' Boolean functions \emph{beyond} Borel (under suitable determinacy hypotheses), with the Wadge hierarchy as a ``backbone''.
Partly for this reason, we have taken care to state our main lemmas, such as the structural lattice isomorphisms mentioned above, as much as possible for arbitrary infinitary clones on $2$, without assuming Borelness.

Clones of polymorphisms have received much attention in recent years in connection with \defn{constraint satisfation problems (CSPs)}, a large class of computational/combinatorial problems, especially since the proof of the \emph{CSP dichotomy theorem} \cite{Bulatov}, \cite{Zhuk} using universal-algebraic methods.
See \cite{Bodirsky} for a detailed survey of this area.
Recently, clones and related concepts have also seen applications to CSPs in Borel combinatorics \cite{Thornton}, \cite{KatayTothVidnyanszky}.
We hope to investigate potential combinatorial applications of Borel clones on $2$ in future work.

\subsection*{Acknowledgments}

We would like to thank Ilijas Farah for providing the reference \cite{KanoveiReeken} cited in \cref{ex:summable}.
The authors were supported by NSF grant DMS-2224709.

\section{Clone theory}
\label{sec:clone}

We begin by reviewing the concepts from general clone theory which we will need.
All of the ideas here are standard in universal algebra; see \cite{FMMT}, \cite{Lau}, \cite{Szendrei}, \cite{Bodirsky}.
However, we will be needing the infinite-arity (even occasionally uncountable-arity) versions of the usual machinery, which may be less familiar.
We will sketch the less routine proofs for the reader's convenience.

\subsection{Lattices and adjunctions}

Recall that a \defn{lattice} is a poset in which any two elements $a, b$ have a greatest lower bound or \defn{meet} $a \wedge b$ as well as least upper bound or \defn{join} $a \vee b$, while a \defn{complete lattice} is a poset in which arbitrarily many elements $a_i$ (possibly none) have a meet $\bigwedge_i a_i$ and join $\bigvee_i a_i$.
A \defn{meet-irreducible} element in a complete lattice is one which cannot be obtained as a meet of strictly smaller elements, i.e., has a unique immediate predecessor; \defn{join-irreducible} is defined similarly.

Recall that an \defn{adjunction} or \defn{Galois connection} between two posets $A, B$ is a pair of maps $f : A -> B$ (the \defn{left adjoint}) and $g : B -> A$ (the \defn{right adjoint}) satisfying
\begin{equation}
\label{eq:adj}
f(a) \le b  \iff  a \le g(b),
\end{equation}
denoted $f \dashv g$.
A left adjoint preserves all existing joins, while a right adjoint preserves all existing meets; the converses hold assuming $A, B$ are complete lattices.
If $B$ is replaced with its order-reversal $B^\op$, then \cref{eq:adj} becomes
\begin{equation*}
b \le f(a)  \iff  a \le g(b)
\end{equation*}
(we call $f, g$ \emph{dually adjoint on the right}), and each of $f, g$ maps joins to meets.
In either case, we have $g \circ f = \id_A$ iff $f$ is injective iff $g$ is surjective, and dually for $f \circ g = \id_B$.

The \defn{interval} between elements $a, b \in A$ of a poset is
\begin{align*}
[a,b] := \{x \in A \mid a \le x \le b\}.
\end{align*}
Given elements $a, b, c, d \in A$ of a lattice, such that $b, c \in [a,d]$, we have the \defn{modularity adjunction}
\begin{equation}
\label{eq:adj-mod}
\begin{tikzcd}
{}
[a,b] \rar["c \vee (-)", shift left=2] &[2em]
{}
[c,d]. \lar["b \wedge (-)", shift left=2, right adjoint']
\end{tikzcd}
\end{equation}
If these form an isomorphism, then necessarily $a = b \wedge c$ and $d = b \vee c$.
(Recall that the lattice $A$ is called \emph{modular} if we have such a modularity isomorphism for all $b, c \in A$.)
See \cref{fig:adj-mod}.

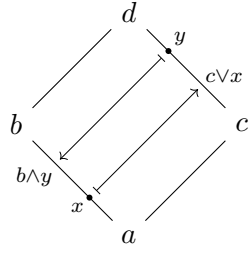
\begin{figure}
\centering
\begin{tikzpicture}[x=1.5cm, y=1.5cm, every label/.append style={inner sep=1pt}]
\node(a) {$a$};
\node(b) at ($(a)+(-1,1)$) {$b$} edge (a);
\node(c) at ($(a)+(1,1)$) {$c$} edge (a);
\node(d) at ($(b)+(c)-(a)$) {$d$} edge (b) edge (c);
\node(x)[dot, label={below left:$\scriptstyle x$}] at ($(a)!.35!(b)$) {};
\node(x')[coordinate, label={above right:$\scriptstyle c \vee x$}] at ($(x)+(c)-(a)$) {};
\node(y)[dot, label={above right:$\scriptstyle y$}] at ($(c)!.65!(d)$) {};
\node(y')[coordinate, label={below left:$\scriptstyle b \wedge y$}] at ($(y)-(c)+(a)$) {};
\draw[|->, very thin, shorten=.5ex] (x) -- (x');
\draw[|->, very thin, shorten=.5ex] (y) -- (y');
\end{tikzpicture}
\caption{The modularity adjunction $[a,b] \rightleftarrows [c,d]$ in a lattice.}
\label{fig:adj-mod}
\end{figure}

For more information on general lattice theory, see \cite{Gratzer}, \cite{GHKLMS}, \cite{Jstone}.

\subsection{Clones}

\begin{convention}
\label{def:arity}
The \defn{arity} of an operation or relation will always be a cardinal, by which we mean a (von Neumann) initial ordinal.
We will use the letters $k, l, m, n$ for cardinals, including infinite ones.
For a class $N$ of cardinals, by \defn{$N$-ary} we mean $n$-ary for some $n \in N$.
We will use the letters $K, L, M, N$ for classes of cardinals.

We abuse notation by writing $n$ instead of $\{n\}$ when convenient.
For a regular cardinal $\kappa$, we write $N = {<}\kappa$ for the set of all cardinals $n < \kappa$.
This includes the case $\kappa = \omega$, where ${<}\omega$ is formally the same as the set $\omega$; however, ``${<}\omega$-ary'' means finitary, while ``$\omega$-ary'' means countably infinitary.
We will use the letters $\kappa, \lambda, \mu, \nu$ for regular cardinals.
We also use ${\le}\kappa$ as an abbreviation for ${<}\kappa^+$.

We will occasionally use $N = {<}\infty$ to mean the class of all cardinals.
\end{convention}

\begin{notation}
\label{def:op}
Let $N$ be a class of cardinals, $X$ be a set.
We let
\begin{equation*}
\Op{X}^N := \bigsqcup_{n \in N} X^{X^n}
\end{equation*}
denote the class of $N$-ary functions on $X$.
We also let $\Op{X} := \Op{X}^{<\infty}$ denote \emph{all} functions on $X$.

For a subclass $F \subseteq \Op{X}^N$ of $N$-ary functions and $n \in N$, we let
\begin{align*}
F^n := F \cap \Op{X}^n
\end{align*}
denote all $n$-ary functions in $F$.
For a subclass of cardinals $M \subseteq N$, we let
\begin{align*}
F^M := F \cap \Op{X}^M = \bigsqcup_{m \in M} F^m
\end{align*}
denote the \defn{$M$-ary restriction} of $F$.
Note that the notation $\Op{X}^N$ can thus be consistently interpreted as the $N$-ary restriction of $\Op{X} = \Op{X}^{<\infty}$.
\end{notation}

\begin{definition}
We say that $F \subseteq \Op{X}^N$ is an \defn{$N$-ary clone} on $X$ if it contains all \defn{projections}
\begin{align*}
\pi_i = \pi^n_i : X^n &--> X \\
\vec{x} &|--> x_i \\
\intertext{for $i < n \in N$, as well as all \defn{compositions}}
g(\vec{f}) = g \circ \vec{f} : X^m &--> X \\
\vec{x} &|--> g(f_0(\vec{x}), f_1(\vec{x}), \dotsc)
\end{align*}
for $m, n \in N$, $g \in F^n$, $\vec{f} = (f_i)_{i < n} \in (F^m)^n$.
Denote the complete lattice of $N$-ary clones on $X$ by
\begin{align*}
\Clo^N{X} \subseteq \@P(\Op{X}^N).
\end{align*}
\end{definition}

\begin{notation}
For an arbitrary class of $N$-ary functions $F \subseteq \Op{X}^N$, we let
\begin{align*}
\ang{F}^N \subseteq \Op{X}^N
\end{align*}
denote the \defn{$N$-ary clone generated by $F$}, i.e., the smallest $N$-ary clone containing $F$.

We also write $\ang{F} := \ang{F}^{<\infty}$.
\end{notation}

\begin{lemma}
\label{thm:clone-gen-left}
$\ang{F}^N \subseteq \Op{X}^N$ is also the smallest class of $N$-ary functions containing all $N$-ary projections and closed under left composition with functions in $F$.
\end{lemma}
\begin{proof}[Proof sketch]
It suffices to verify that said smallest class, call it $G$, is closed under composition.
Indeed, it is easily checked that $\{f \in \Op{X}^N \mid \forall \vec{g} \in G\, (f \circ \vec{g} \in G)\}$ contains all projections and is closed under left composition with functions in $F$, hence contains $G$.
\end{proof}

Note that this description of $\ang{F}^N$ separately characterizes the set of $n$-ary functions in $\ang{F}^N$, for each $n \in N$ (independently of the other cardinals in $N$).
Thus

\begin{corollary}
\label{thm:clone-arity}
For two classes of cardinals $M \subseteq N$ and $F \subseteq \Op{X}^M$, we have $\ang{F}^N \cap \Op{X}^M = \ang{F}^M$.

Thus, if $F$ is an $M$-ary clone, then $\ang{F}^N \cap \Op{X}^M = F$.
In other words, the adjunction
\begin{equation*}
\begin{tikzcd}
\Clo^M{X} \rar[hook, shift left=2, "{\ang{-}^N}"] &[2em]
\Clo^N{X} \lar[two heads, shift left=2, "{(-) \cap \Op{X}^M}", right adjoint']
\end{tikzcd}
\end{equation*}
exhibits the lattice $\Clo^M{X}$ as a retract of $\Clo^N{X}$.
\qed
\end{corollary}

When $M = {<}\mu$ and $N = {<}\nu$ for infinite regular cardinals $\mu \le \nu$, we have a simpler description of $\ang{F}^N$ for an $M$-ary clone $F$:

\begin{lemma}
For a regular cardinal $\mu$ and ${<}\mu$-ary clone $F \subseteq \Op{X}^{<\mu}$, for any cardinal $n$,
\begin{align*}
\ang{F}^n = \{f \circ (\pi^n_{s(i)})_{i < m} \mid m < \mu,\, f \in F^m,\, s : m `-> n\}
\end{align*}
consists of precisely the $n$-ary functions in $F$ which depend on only ${<}\mu$-many variables.
\end{lemma}
\begin{proof}[Proof sketch]
It is easily checked that said class contains all projections and is closed under left composition with functions in $F$ (using regularity of $\mu$ to disjointify variables).
\end{proof}

\begin{corollary}
For regular cardinals $\mu \le \nu$, the image of the map
\begin{align*}
\ang{-}^{<\nu} : \Clo^{<\mu}{X} &--> \Clo^{<\nu}{X}
\end{align*}
consists of precisely the ${<}\nu$-ary subclones of $\ang{\Op{X}^{<\mu}}^{<\nu}$.
\qed
\end{corollary}

\begin{definition}
We call a function $f : X^n -> X$ \defn{essentially ${<}\mu$-ary} if it depends on only ${<}\mu$-many variables, i.e., it is in $\ang{\Op{X}^{<\mu}}$.
We call a clone $F$ \defn{essentially ${<}\mu$-ary} if it consists entirely of essentially ${<}\mu$-ary functions, i.e., $F \subseteq \ang{\Op{X}^{<\mu}}$.
\end{definition}

\begin{remark}
\label{rmk:ktop}
If $X$ is finite, then an essentially finitary function $f : X^n -> X$ is one which is continuous with respect to the discrete topology on $X$ and the product topology on $X^n$.

More generally, for a class of cardinals $M$, let us define a \defn{$M$-topology} to mean a family of sets closed under arbitrary unions and $M$-ary intersections, and the \defn{product $M$-topology} on $X^n$ to mean the smallest $M$-topology containing preimages of arbitrary sets under the projections $\pi_i : X^n -> X$ (so interpolating between the usual product topology when $M = {<}\omega$ and the box topology when $M = {<}\infty$).
Then, for $\abs{X} < \mu$, an essentially ${<}\mu$-ary function $f : X^n -> X$ is one which is \defn{${<}\mu$-continuous}, i.e., the preimage of an arbitrary set belongs to the product ${<}\mu$-topology.
\end{remark}

%\begin{remark}
%\label{rmk:knet}
%Recall that a function $f : X -> Y$ between topological spaces is continuous iff it preserves limits of nets $(x_i)_{i \in I} \in X^I$, meaning $\lim_i x_i = x_\infty \implies \lim_i f(x_i) = f(x_\infty)$, for all directed posets $I$.
%If each $x \in X$ has a neighborhood basis of size $< \kappa$, then it suffices to consider $\abs{I} < \kappa$.
%(For example, if $X$ is first-countable, then it suffices to consider countable sequences.)
%
%More generally, let us say a poset $I$ is \defn{$M$-directed} if any subfamily $(i_j)_{j < m} \in I^m$ of size $m \in M$ has an upper bound in $I$; and an \defn{$M$-net} is a family of elements indexed by an $M$-directed poset $I$.
%Then it is easy to see that a function $f : X -> Y$ between ${<}\mu$-topological spaces is ${<}\mu$-continuous iff it preserves limits of ${<}\mu$-nets, and that the product ${<}\mu$-topology on $X^n$ has limits of ${<}\mu$-nets given pointwise.
%If $\sum_{m < \mu} n^m < \kappa$, then each $\vec{x} \in X^n$ has a neighborhood basis of size $< \kappa$ in the product ${<}\mu$-topology, which is hence defined by limits of ${<}\mu$-nets of size $< \kappa$.
%\end{remark}

\subsection{Polymorphisms}

\begin{notation}
Let $K$ be a class of cardinals, $X$ be a set.
We let
\begin{equation*}
\Rel{X}^K := \bigsqcup_{k \in K} \@P(X^k)
\end{equation*}
denote the class of all $K$-ary relations on $X$.
We will loosely refer to a subclass $\@M \subseteq \Rel{X}^K$ as a (relational) \defn{$K$-ary structure} on $X$.
We use similar notations as for functions (\cref{def:op}) to denote arities: $\@M^k := \@M \cap \Rel{X}^k$ means the $k$-ary relations in $\@M$; $\Rel{X} := \Rel{X}^{<\infty}$; etc.
\end{notation}

\begin{definition}
\label{def:pol-inv}
Let $X$ be a set, $f : X^n -> X$ be an $n$-ary function, $R \subseteq X^k$ be a $k$-ary relation, for some cardinals $n, k$.
The following phrases are synonymous:
\begin{itemize}
\item
$f$ \defn{preserves} $R$, or is a \defn{polymorphism} of $R$.
\item
$f$ is a homomorphism $(X,R)^n -> (X,R)$, where $(X,R)^n := (X^n,R^n)$ is the \defn{product structure} (in the category of structures with a single $n$-ary relation).
\item
$R$ is \defn{$f$-invariant}, or \defn{closed} under $f$, or a \defn{substructure} of the product structure $(X,f)^k$ where $f$ acts coordinatewise.
\item
For any $k \times n$ matrix $(\vec{x}_i)_{i < k} = (x_{i,j})_{i < k, j < n} \in X^{k \times n}$, if each $(x_{i,j})_{i < k} \in R$ for fixed $j < n$, then also $(f(\vec{x}_i))_{i < k} = (f((x_{i,j})_{j < n}))_{i < k} \in R$.
\begin{equation*}
\begin{tikzcd}[
    row sep=0em, column sep=1em,
    every node/.append style={allow upside down},
    execute at end picture={
        \draw[line width=.8pt]
            (\tikzcdmatrixname-1-1.north west) to[square left brace] (\tikzcdmatrixname-1-1.north west |- \tikzcdmatrixname-3-1.south)
            (\tikzcdmatrixname-1-4.north east) to[square right brace] (\tikzcdmatrixname-1-4.north east |- \tikzcdmatrixname-3-4.south)
            (\tikzcdmatrixname-1-1.north -| \tikzcdmatrixname-1-5.west) to[square left brace] (\tikzcdmatrixname-1-5.north west |- \tikzcdmatrixname-3-5.south)
            (\tikzcdmatrixname-1-1.north -| \tikzcdmatrixname-1-5.east) to[square right brace] (\tikzcdmatrixname-1-5.north east |- \tikzcdmatrixname-3-5.south);
        \draw[decoration={brace,mirror}, decorate]
            ([xshift=-1em]\tikzcdmatrixname-1-1.north west) to["$\scriptstyle k$"{left}] ([xshift=-1em]\tikzcdmatrixname-1-1.north west |- \tikzcdmatrixname-3-1.south);
        \draw[decoration={brace}, decorate]
            ([yshift=1.5ex]\tikzcdmatrixname-1-1.north west) to["$\scriptstyle n$"{above}] ([yshift=1.5ex]\tikzcdmatrixname-1-1.north west -| \tikzcdmatrixname-1-4.east);
    },
]
x_{0,0} & x_{0,1} & x_{0,2} & \dotsb &[2em] f(\vec{x}_0) \\
x_{1,0} & x_{1,1} & x_{1,2} & \dotsb \rar[phantom, "\to"] & f(\vec{x}_1) \\
\vdots \dar[phantom, "\in"{sloped}] & \vdots \dar[phantom, "\in"{sloped}] & \vdots \dar[phantom, "\in"{sloped}] & \ddots & \vdots \dar[phantom, "\in"{sloped}] \\[1.5em]
R & R & R & & R
\end{tikzcd}
\end{equation*}
\end{itemize}
If these hold, we write interchangeably
\begin{align*}
f \in \Pol(R)  \iff  R \in \Inv(f)
\end{align*}
where $\Pol(R) \subseteq \Op{X}$ is the class of all polymorphisms of $R$, and $\Inv(f) \subseteq \Rel{X}$ is the class of all $f$-invariant relations.
More generally, for any $\@M \subseteq \Rel{X}$ and $F \subseteq \Op{X}$, we write
\begin{align*}
\Pol(\@M) := \bigcap_{R \in \@M} \Pol(R), &&
\Inv(F) := \bigcap_{f \in F} \Inv(f).
\end{align*}
We let also $\Pol^N(\@M) := \Pol(\@M) \cap \Op{X}^N$ and $\Inv^K(F) := \Inv(F) \cap \Rel{X}^K$ for classes of cardinals $N, K$.
\end{definition}

Thus for any classes of cardinals $N, K$, we have an (order-reversing) Galois connection
\begin{equation}
\label{eq:pol-inv}
\begin{tikzcd}
\@P(\Op{X}^N) \rar[shift left, "\Inv^K"] &[2em]
\@P(\Rel{X}^K) \lar[shift left, "\Pol^N"],
\end{tikzcd}
\end{equation}
i.e., $F \subseteq \Pol(\@M) \iff \@M \subseteq \Inv(F) \iff$ every $f \in F$ preserves every $R \in \@M$.

Every class of functions of the form $\Pol(\@M)$ is a clone.
Thus the fixed classes on the left side of the above Galois connection, i.e., those $F \subseteq \Op{X}^N$ for which $F = \Pol^N(\Inv^K(F))$, must in particular be clones.
We may characterize them more precisely as follows:

\begin{lemma}
\label{thm:pol-inv-kcl}
For any class of functions $F \subseteq \Op{X}$, class of cardinals $K$, and $n$-ary function $g : X^n -> X$, we have $g \in \Pol^n(\Inv^K(F))$ iff for any $K$-ary family of $n$-tuples $(\vec{x}_i)_{i < k} \in (X^n)^k$, where $k \in K$, there is an $f \in \ang{F}^n$ such that $f(\vec{x}_i) = g(\vec{x}_i)$ for each $i < k$.
\end{lemma}

In other words, this says that $\Pol^n(\Inv^K(F))$ is the \defn{$K$-closure} of $\ang{F}^n \subseteq X^{X^n}$, with respect to the product $K$-topology from \cref{rmk:ktop}.
For instance:
\begin{itemize}
\item
When $K = {<}\omega$, we get that $\Pol^n(\Inv^{<\omega}(F)) = \-{\ang{F}^n}$ is the usual closure of $\ang{F}^n \subseteq X^{X^n}$.
\item
When $K = {<}\omega_1$, we get that $\Pol^n(\Inv^{<\omega_1}(F))$ consists of all functions agreeing on any countably many $n$-tuples with a function in $\ang{F}^n$; we call this the \defn{countable closure} of $\ang{F}^n$.
\item
When $K \supseteq {\le}\abs{X}^n$, e.g., $X$ is finite, $n$ countable, and $K = {\le}2^{\aleph_0}$, we get $\Pol^n(\Inv^K(F)) = \ang{F}^n$.
\end{itemize}

\begin{proof}
Any $\Pol^n(R)$ for a $K$-ary relation $R$ is $K$-closed, since for a function to preserve $R$ requires checking its values only on each $K$-ary family of $n$-tuples at a time; thus $\Pol^n(\Inv^K(F))$ contains the $K$-closure of $F$.
Conversely, for $g \in \Pol^n(\Inv^K(F))$ and $(\vec{x}_i)_{i < k} \in (X^n)^k$, the forward $\ang{F}$-orbit $R := \{(f(\vec{x}_i))_{i < k} \mid f \in \ang{F}^n\}$ is $F$-invariant, hence preserved by $g$; we have $(x_{i,j})_{i < k} = (\pi^n_j(\vec{x}_i))_{i < k} \in R$ for each $j$, whence $(g(\vec{x}_i))_{i < k} \in R$, i.e., $(g(\vec{x}_i))_{i < k} = (f(\vec{x}_i))_{i < k}$ for some $f \in \ang{F}^n$.
\end{proof}

On the right side of the adjunction \cref{eq:pol-inv}, the fixed classes are characterized as follows:

\begin{lemma}
\label{thm:inv-pol-npp}
For any class of relations $\@M \subseteq \Rel{X}$, regular cardinal $\nu$, and $k$-ary relation $S \subseteq X^k$, we have $S \in \Inv^k(\Pol^{<\nu}(\@M))$ iff $S$ is a ${<}\nu$-directed union of relations \defn{positive-primitively definable} from $\@M$, i.e., (infinitary) first-order definable from the relations in $\@M$ using $=$, $\exists$ over arbitrarily many variables, and $\bigwedge$ of arbitrary arity.
\end{lemma}

Here by a \emph{${<}\nu$-directed union} we mean a union of a %${<}\nu$-directed (recall \cref{rmk:knet}) family.
family of sets in which any ${<}\nu$-sized subfamily has an upper bound.
Let
\begin{equation*}
\PPStr^K{X} \subseteq \@P(\Rel{X}^K)
\end{equation*}
denote the class of \defn{positive-primitive $K$-ary structures}, by which we mean structures closed under positive-primitive definability (sometimes called \emph{coclones}; see \cite{Lau}).
The above then says that $\Inv^K(\Pol^{<\nu}(\@M)) = \@M$ is the closure under ${<}\nu$-directed union of the positive-primitive $K$-ary structure generated by $\@M$.
Note that for sufficiently large $\nu$, namely $\nu > 2^{\abs{X}^k}$, ${<}\nu$-directed unions in $X^k$ are trivial; thus we simply get the positive-primitively definable relations in that case.

\begin{proof}[Proof]
Any $\Inv^k(f)$ for a ${<}\nu$-ary function $f$ is easily seen to be closed under ${<}\nu$-directed union and positive-primitive definability; thus $\Inv^k(\Pol^{<\nu}(\@M))$ contains all such $R$.
Conversely, let $S \in \Inv^k(\Pol^{<\nu}(\@M))$.
Then any family of tuples $(\vec{x}_j)_{j < n} \in S^n$ where $n < \nu$ generates a smallest $\Pol^{<\nu}(\@M)$-invariant subset $R((\vec{x}_j)_{j < n}) \subseteq X^k$ contained in $S$; and $S$ is the ${<}\nu$-directed union of all of these $R((\vec{x}_j)_{j < n})$.
But each $R((\vec{x}_j)_{j < n})$ is the saturation of $(\vec{x}_j)_{j < n}$ under all $f \in \Pol^n(\@M)$ (as in the proof of \cref{thm:pol-inv-kcl}, with the ``transpose matrix'' of $(\vec{x}_j)_j$); and ``$\exists f \in X^{X^n}\, (f \text{ is a polymorphism and } f((\vec{x}_j)_j) = \vec{y})$'' is a positive-primitive formula (where the existential is over $\abs{X^n}$-many variables).
\end{proof}

The preceding two lemmas are the natural generalizations of classical results of Geiger \cite{Geiger} and Bodnarčuk--Kalužnin--Kotov--Romov \cite{BKKR} to arbitrary arities.

\subsection{Comparing arities}

The $\Pol$--$\Inv$ adjunction \cref{eq:pol-inv} interacts with the change of arity adjunction from \cref{thm:clone-arity} as follows.
For classes of cardinals $M \subseteq N$ and $K$, we have the diagram
\begin{equation}
\label{eq:pol-inv-arity}
\begin{tikzcd}
\@P(\Op{X}^N)
    \dar[two heads, "{(-) \cap \Op{X}^M}"]
    \rar[shift left=4, "\ang{-}^N"]
    &[2em]
\Clo^N{X}
    \lar[hook', right adjoint']
    \dar[two heads, shift left=0, right adjoint', "{(-) \cap \Op{X}^M}"]
    \rar[shift left=4, "\Inv^K"]
    &[2em]
\PPStr^K{X}^\op
    \lar[shift left=0, right adjoint', "\Pol^N"]
    \dlar[shift left=4, right adjoint', "\Pol^M"]
    \rar[hook]
    &
\@P(\Rel{X}^K)^\op
\\[2em]
\@P(\Op{X}^M)
    \uar[hook, shift left=4, left adjoint']
    \rar[shift left=2, "\ang{-}^M"]
    &
\Clo^M{X}
    \lar[hook', right adjoint', shift left=2]
    \uar[hook, shift left=4, "{\ang{-}^N}"]
    \urar["\Inv^K"{pos=.7,inner sep=0pt}]
\end{tikzcd}
\end{equation}
in which the two diagrams of left/right adjoints each commute, and the unmarked arrows are inclusions.
For sufficiently large $K$, $\Inv^K$ is an embedding of both $\Clo^N{X}$ and $\Clo^M{X}$ into $\PPStr^K{X}^\op$ by \cref{thm:pol-inv-kcl}.
Thus for instance, commutativity of the triangle yields

\begin{corollary}
For a set of finitary functions $F \subseteq \Op{X}^{<\omega}$, the countable clone it generates is $\ang{F}^{<\omega_1} = \Pol^{<\omega_1}(\Inv(F))$.
If $X$ is finite, it suffices to take $\Inv^{\le 2^{\aleph_0}}$.
\qed
\end{corollary}

On the other hand, if $M, K$ are such that $\Pol^M, \Inv^K$ form inverse order-isomorphisms $\Clo^M{X} \cong \PPStr^K{X}^\op$ by \cref{thm:pol-inv-kcl,thm:inv-pol-npp}, e.g., $X$ is finite and $M = K = {<}\omega$, then the triangle exhibits this common lattice as a retract of $\Clo^N{X}$ in two different ways.

\begin{lemma}
For a finite $X$, regular cardinal $\nu$, and ${<}\nu$-ary clone $F$, $\Pol^{<\omega}(\Inv^{<\omega}(F)) = F \cap \Op{X}^{<\omega}$.
\end{lemma}
\begin{proof}
For $n < \omega$, $\Pol^n(\Inv^{<\omega}(F))$ is by \cref{thm:pol-inv-kcl} the closure of $F^n$; but $F^n$ is finite.
\end{proof}

\begin{corollary}
\label{thm:pancake}
For a finite $X$ and regular cardinal $\nu$, the adjunction of \cref{thm:clone-arity} admits a further right adjoint:
\begin{equation*}
\begin{tikzcd}
\Clo^{<\nu}{X}
    \dar[two heads, shift right=2, "{(-) \cap \Op{X}^{<\omega}}"]
\\[2em]
\Clo^{<\omega}{X}
    \uar[hook, shift left=6, "{\ang{-}^{<\nu}}", left adjoint']
    \uar[hook, bend right=60, looseness=2, shift right=4, "{\Pol^{<\nu} \circ \Inv^{<\omega} = \-{\ang{-}^{<\nu}}}"{right}, right adjoint]
\end{tikzcd}
\end{equation*}
In other words, the complete lattice $\Clo^{<\nu}{X}$ of ${<}\nu$-ary clones admits the complete lattice $\Clo^{<\omega}{X}$ of finitary clones as a quotient (via restriction to finitary functions).
Thus we may regard the former lattice as a ``bundle'' over the latter.
Each finitary clone $F \in \Clo^{<\omega}{X}$ admits a least extension to a ${<}\nu$-ary clone, namely $\ang{F}^{<\nu} = \{\text{essentially finite $\nu$-ary versions of functions in $F$}\}$, as well as a greatest extension to a ${<}\nu$-ary clone, namely $\-{\ang{F}^{<\nu}} = \{\text{pointwise limits of functions in } \ang{F}^{<\nu}\}$.
\qed
\end{corollary}

\begin{notation}
\label{def:clone-fiber}
In the situation of \cref{thm:pancake}, for a finitary clone $F \in \Clo^{<\omega}{X}$, we write
\begin{align*}
\Clo^{<\nu}_F{X} := [\ang{F}^{<\nu}, \-{\ang{F}^{<\nu}}] \subseteq \Clo^{<\nu}{X}
\end{align*}
for the interval of ${<}\nu$-ary clones with finitary restriction $F$, i.e., the fiber of the bundle \ref{thm:pancake} over $F$.
\end{notation}

As is typical in discussions of Post's lattice, we will also adopt the following

\begin{convention}
\label{def:nullary}
We will henceforth assume that clones do not contain nullary functions (constants).
Thus, when we say e.g., a ``${<}\omega$-ary clone'', we really mean $N$-ary for $N = \{n \mid 0 < n < \omega\}$.
When we refer to a constant function, e.g., $0 \in 2$, we will always mean $0 : 2^n -> 2$ for some $n > 0$.
\end{convention}

This convention is used to eliminate uninteresting duplication when classifying clones on a given set such as $X = 2$.
Without it, for each clone $F$ containing a positive-arity constant function $c = 0$ or $1$, we may arbitrarily choose whether or not to include the nullary versions of all constants in $F$.
Thus, the classification of clones on $2$ possibly containing nullary functions is essentially the same as for clones without, except that clones with constants are duplicated.

\section{Clones on $2$}
\label{sec:post}

In this section, we discuss general aspects of clones of arbitrary arity on $2 = \{0,1\}$.
In \cref{sec:post-funs-rels}, we define the standard operations (logical connectives) and invariant relations used to specify the clones in Post's lattice $\Clo^{<\omega}{2}$, as well as their infinitary generalizations.
We then define the standard clones in Post's lattice and state Post's theorem in \cref{sec:post-clones}.
In \crefrange{sec:post-ceil-floor-xsec}{sec:post-beta}, we introduce some operators on functions that map between different regions of Post's lattice.

\subsection{Basic functions and relations on $2$}
\label{sec:post-funs-rels}

\begin{definition}
\label{def:post-funs}
We define the following functions on $2$ of various arities:
\begin{itemize}
\item
$\neg : 2 -> 2$ is the bit flip (logical negation).
\end{itemize}
Given now any other function $f : 2^n -> 2$, its \defn{de~Morgan dual} is its $\neg$-conjugate
\begin{align*}
\delta(f) := \neg \circ f \circ \neg.
\end{align*}
This is an automorphism of the clone $\Op2$; thus below, anything said about a function or a clone automatically transfers to its de~Morgan dual.
\begin{itemize}
\item
$0^n, 1^n : 2^n -> 2$ are the constant functions.
\item
$\wedge = \wedge^2 : 2^2 -> 2$ is binary conjunction (minimum); ${\vee} = \delta(\wedge)$ is binary disjunction.
\item
$\wedge^n$ or $\bigwedge^n : 2^n -> 2$ is $n$-ary conjunction; $\bigwedge$ by default means $\bigwedge^\omega$.
Similarly for ${\bigvee^n} := \delta(\bigwedge^n)$.
\item
$\forall^n_k : 2^n -> 2$ and its dual $\exists^n_k : 2^n -> 2$, where $0 \le k \le n^+$, are given by:
\begin{align*}
\forall^n_k(\vec{x})
&:= \bigwedge_{\substack{I \subseteq n \\ \abs{I} \ge k}} \bigvee_{i \in I} x_i
= \bigvee_{\substack{J \subseteq n \\ \abs{n \setminus J} < k}} \bigwedge_{j \in J} x_j
= \text{``all but $<k$ inputs are true''}, \\
\exists^n_k(\vec{x})
&:= \bigvee_{\substack{I \subseteq n \\ \abs{I} \ge k}} \bigwedge_{i \in I} x_i
= \bigwedge_{\substack{J \subseteq n \\ \abs{n \setminus J} < k}} \bigvee_{j \in J} x_j
= \text{``at least $\ge k$ inputs are true''}.
\end{align*}
Note that $\forall^n_1 = \wedge^n$, and $\exists^n_1 = \vee^n$.
Note also that for $k, n < \omega$, we have $\forall^n_k = \exists^n_{n-k+1}$.
In particular, $\forall^3_2 = \exists^3_2$ is the 2-out-of-3 median function
\begin{align*}
\forall^3_2(x,y,z)
= (x \vee y) \wedge (x \vee z) \wedge (y \vee z)
= (x \wedge y) \vee (x \wedge z) \vee (y \wedge z)
= \exists^3_2(x,y,z).
\end{align*}
By default, $\forall_k$ means $\forall^\omega_k : 2^\omega -> 2$, and similarly $\exists_k := \exists^\omega_k$.
\item
$\liminf := \forall_\omega : 2^\omega -> 2$ and its dual $\limsup := \exists_\omega$, which may be written more simply as:
\begin{align*}
\liminf \vec{x} &= \bigvee_{i \in \omega} \bigwedge_{j \ge i} x_j, &
\limsup \vec{x} &= \bigwedge_{i \in \omega} \bigvee_{j \ge i} x_j.
\end{align*}
We thus have the following ordering among all aforementioned functions of arity $\omega$:
\begin{align*}
\forall_0 = 0 \le \forall_1 = \inline\bigwedge \le \forall_2 \le \forall_3 \le \dotsb \le \forall_\omega = \liminf
\le \limsup = \exists_\omega \le \dotsb \le \exists_2 \le \exists_1 \le \exists_0.
\end{align*}
%\item
%More generally, for any directed poset $I$ (whose underlying set we may assume to be a cardinal), we have $\liminf^I : 2^I -> 2$ and $\limsup^I : 2^I -> 2$ taking nets as input, defined by the same formulas as above (replacing $\omega$ with $I$).
\item
${+}^n : 2^n -> 2$ is the $n$-ary addition mod $2$ (i.e., ``XOR'').
By default, ${+}$ means $+^2$.
\item
$(?:) : 2^3 -> 2$ is the ternary conditional (``if-then-else'')
\begin{align*}
(x ? y : z) := (x \wedge y) \vee (\neg x \wedge z).
\end{align*}
\item
$-> : 2^2 -> 2$ is the Boolean implication $x -> y = \neg x \vee y$.
Its de Morgan dual is $\delta(->) = </-$, the variable transposition of $x -/> y = x \wedge \neg y$.
\end{itemize}
\end{definition}

\begin{definition}
\label{def:ceil-floor}
For $f : 2^n -> 2$, define its \defn{upper/lower-bounded versions}
\begin{align*}
\begin{aligned}
\ceil{f} : 2^{1+n} &-> 2 \\
(x_0, x_1, \dotsc) &|-> x_0 \wedge f(x_1, \dotsc),
\end{aligned}
&&
\begin{aligned}
\floor{f} : 2^{1+n} &-> 2 \\
(x_0, x_1, \dotsc) &|-> x_0 \vee f(x_1, \dotsc).
\end{aligned}
\end{align*}
Examples include:
\begin{itemize}
\item
$<- = \floor{\neg}$ (where $x <- y := y -> x$), and similarly $-/> = \ceil{\neg}$.
\item
$\ceil{->} : 2^3 -> 2$ is thus given by $\ceil{->}(x, y, z) = \ceil{\floor{\neg}}(x, z, y) = x \wedge (y -> z)$.
Its de~Morgan dual is $\floor{-/>}(x, y, z) = \floor{\ceil{\neg}}(x, y, z) = x \vee (y \wedge \neg z)$.
\item
$\ceil{\vee} : 2^3 -> 2$ is similarly given by $\ceil{\vee}(x, y, z) = x \wedge (y \vee z)$; dually, $\floor{\wedge}(x, y, z) = x \vee (y \wedge z)$.
\item
$\ceil{\bigvee} : 2^\omega -> 2$ is similarly given by
$
\ceil{\bigvee}(x_0, x_1, x_2, \dotsc) = x_0 \wedge (x_1 \vee x_2 \vee \dotsb).
$
\item
$\ceil{\wedge} = \wedge^3$ and $\ceil{\bigwedge} = \bigwedge$.
\end{itemize}
\end{definition}

\begin{notation}
\label{def:fun-rel-<}
For a family of functions $(f^n : 2^n -> 2)_n$ of various arities $n \ge 1$, when we write $f^{<\nu}$ as part of a set of functions, we mean that $f^n$ for each $1 \le n < \nu$ is included.
For instance,
\begin{align*}
\ang{\vee, \bigwedge^{<\omega_1}} = \ang{\vee, \id, \wedge, \wedge^3, \wedge^4, \dotsc, \bigwedge}
\end{align*}
consists of all functions built from binary joins and countable meets (which is also just $\ang{\vee, \bigwedge}$).

%For a family $(f^I)_I$ of ``limit-type'' operations indexed over directed posets $I$, by $f^{<\nu}$ we instead mean that $f^I$ for each $I$ with $\abs{I} < \nu$ is included, while by $f_{<\mu}^{<\nu}$ we mean that we only consider ${<}\mu$-directed $I$ (recall \cref{rmk:knet}).
%For example,
%\begin{align*}
%\brace{\liminf^{<\infty}} &= \set{\liminf^I}{I \text{ directed}}, \\
%\brace{\liminf_{<\omega_1}^{\le 2^{\aleph_0}}} &= \set{\liminf^I}{I \text{ countably directed, } \abs{I} \le 2^{\aleph_0}}.
%\end{align*}

We use similar notation for sets of relations indexed over a cardinal %or a directed set
(see below).
\end{notation}

\begin{definition}
\label{def:post-rels}
We define the following relations on $2$ of various arities:
\begin{itemize}
\item
${\le} \subseteq 2^2$ is the usual linear order where $0 < 1$.
A function preserves $\le$ iff it is monotone.
\end{itemize}
If $f : 2^k -> 2$ is a function of arity $k$, or more generally a partial function, then we may also treat $f$ as the $(1+k)$-ary relation given by its graph.
Examples include:
\begin{itemize}
\item
The constant $0$ may be treated as the unary relation $\{0\}$.
A function $g$ preserves $0$ iff $g(\vec{0}) = 0$.
\item
$\neg : 2^1 -> 2$ is identified with the binary relation $\ne$.
A function $g$ preserves $\neg$ iff it is \defn{self-dual}, i.e., $\delta(g) = g$, i.e., whenever $\vec{x}, \vec{y}$ differ in \emph{every} coordinate (so $\vec{y} = \neg \vec{x}$), then $g(\vec{x}) \ne g(\vec{y})$.
\item
$+ : 2^2 -> 2$ is identified with the ternary relation $\{(x,y,z) \in 2^3 \mid x+y+z = 0\}$.
A function preserves $+$ iff it is a linear transformation (over $\#Z/2\#Z$).
\item
$+^3 : 2^3 -> 2$ is identified with the quaternary relation $\{(x,y,z,w) \in 2^4 \mid x+y+z+w = 0\}$.
A function preserves $+^3$ iff it is an affine transformation.
\item
$\wedge : 2^2 -> 2$ as a ternary relation is preserved by $g : 2^n -> 2$ iff $g$ is a meet-semilattice homomorphism, which means that $g^{-1}(1) \subseteq 2^n$ is empty or a filter.
\item
The partial increasing join operation $\bigveeup : 2^\omega -> 2$, defined by
\begin{align*}
\bigveeup \vec{x} = y  \coloniff  (x_0 \le x_1 \le \dotsb) \wedge \paren[\Big]{\bigvee_i x_i = y},
\end{align*}
as a $1+\omega = \omega$-ary relation is preserved by $g : 2^n -> 2$, $n \le \omega$, iff $g$ is \emph{Scott-continuous}, i.e., monotone and $g^{-1}(0) \subseteq 2^n$ is closed under increasing joins, or equivalently closed in the Cantor topology (see e.g., \cite[III-1.6]{GHKLMS}).
Similarly for decreasing meet $\bigwedgedown: 2^\omega -> 2$.
\item
Similarly, the partial limit operation $\lim : 2^\omega -> 2$ is preserved by $g : 2^\omega -> 2$ iff $g$ is continuous.
%\item
%More generally, for any directed poset $I$, we have a partial limit operation $\lim^I : 2^I -> 2$ that takes limits of $I$-indexed nets.
%Thus, recalling \cref{rmk:knet,def:fun-rel-<},
%\begin{align*}
%\Pol^{<\nu} \brace{\lim^{<\nu}} &= \Pol^{<\nu} \brace{\lim^{<\infty}} = \{g : 2^n -> 2 \mid n < \nu,\, g \text{ continuous}\}, \\
%\Pol^{<\nu} \brace{\lim_{<\mu}^{<\kappa}} &= \Pol^{<\nu} \brace{\lim_{<\mu}^{<\infty}} = \{g : 2^n -> 2 \mid n < \nu,\, g \text{ ${<}\mu$-continuous}\}
%\end{align*}
%for sufficiently large $\kappa$ (namely so that $\sum_{m < \mu} n^m < \kappa$ for all $n < \nu$).
%For instance,
%$\Pol^{\le 2^{\aleph_0}} \brace{\lim_{<\omega_1}^{\le 2^{\aleph_0}}}$
%consists of all countably continuous functions of arity $\le 2^{\aleph_0}$.
\end{itemize}
\end{definition}

\begin{definition}
\label{def:event}
For $f : 2^k -> 2$, we also have two $k$-ary relations, distinct from the graph of $f$, which we denote using the ``probabilist's event notation''
\begin{align*}
(f{=}0) := f^{-1}(0) = \{\vec{x} \in 2^k \mid f(\vec{x}) = 0\}, &&
(f{=}1) := f^{-1}(1) = \{\vec{x} \in 2^k \mid f(\vec{x}) = 1\}.
\end{align*}
Examples include:
\begin{itemize}
\item
$({\bigwedge^k}{=}0) \subseteq 2^k$ is the $k$-ary disjointness relation ($k \ge 1$).
A function $g : 2^n -> 2$ preserves ${\bigwedge^k}{=}0$ iff $g^{-1}(1) \subseteq 2^n$ does not contain $k$ strings with bitwise meet $\vec{0}$.
If $n \le k$, then (by considering the strings with exactly one $1$) this is equivalent to $g \le \pi_i$ for some $i < n$.
%Thus for instance,
%\begin{align*}
%\Pol \brace{{\bigwedge^{<\omega}}{=}0} = \{g : 2^n -> 2 \mid g^{-1}(1) \subseteq 2^n \text{ has the finite intersection property}\}.
%\end{align*}
\item
In particular, ${\bigwedge^1}{=}0$ yields the same unary relation as the constant $0$.
\item
$({\lim}{=}0) \subseteq 2^\omega$ is the set of eventually 0 sequences.
A function $g : 2^n -> 2$ for $n \le \omega$ preserves ${\lim}{=}0$ iff $g$ vanishes on a neighborhood of $\vec{0} \in 2^n$, i.e., $g(\vec{0}) = 0$ and $g$ is continuous at $\vec{0} \in 2^n$.
%\item
%Similarly, for a directed set $I$, ${\lim^I}{=}0$ is the set of eventually $0$ $I$-nets.
%Thus
%\begin{align*}
%\Pol^n \brace{{\lim_{<\mu}^{<\kappa}}{=}0} =
%\Pol^n \brace{{\lim_{<\mu}^{<\infty}}{=}0} =
%\set{g : 2^n -> 2}{g(\vec{0}) = 0 \AND g \text{ is ${<}\mu$-continuous at } \vec{0}}
%\end{align*}
%for sufficiently large $\kappa$.
\end{itemize}
\end{definition}

\subsection{Post's lattice}
\label{sec:post-clones}

\begin{definition}
\label{def:post-clones}
We use the following names for certain clones on $2$:
\begin{itemize}
\item  $\Aff := \Pol \{+^3\}$ consists of the affine functions.
\item  $\Dual := \Pol \{\neg\}$ consists of the self-dual functions.
\item  $\Mono := \Pol \{\le\}$ consists of the monotone functions.
\item  $\Decr := \Pol \{\bigwedgedown\}$ and $\Incr := \Pol \{\bigveeup\}$ (note that these are contained in $\Mono$, by definition of the domains of the partial operations $\bigwedgedown, \bigveeup$: for example, $x \le y \iff \bigwedgedown (y, x, x, \dotsc) = x$).
\item  $\Cons{0}{k} := \Pol \{{\bigwedge^k}{=}0\}$ and $\Cons{1}{k} := \Pol \{{\bigvee^k}{=}1\}$, where $k$ is a positive cardinal.
\item  $\Cons{c}{1} := \Cons{c}{{1}} = \Pol \{c\}$ consists of the functions preserving the constant $c \in 2$.
\item  $\Cons{c}{<\kappa} := \bigcap_{1 \le k < \kappa} \Cons{c}{k}$ for a regular cardinal $\kappa$.
Thus
\begin{align*}
\Cons{c}{1} = \Cons{c}{{1}} \supseteq \Cons{c}{2} \supseteq \Cons{c}{3} \supseteq \dotsb \supseteq \Cons{c}{<\omega} \supseteq \Cons{c}{\omega} = \Cons{c}{<\omega_1} \supseteq \Cons{c}{\omega_1} \supseteq \dotsb.
\end{align*}
\item  $\Limm{c}{1}{1} := \Pol \{{\lim}{=}c\}$.
(Generalizations called $\Limm{c}{k}{k}$ and $\Limm{c}{k}{t}$ will be defined later in \cref{sec:borel-T0k}.)
\item  $\Meet := \Pol \{\wedge\}$ and $\Join := \Pol \{\vee\}$; more generally, $\Meet_k := \Pol \{\bigwedge^k\}$ and $\Join_k := \Pol \{\bigvee^k\}$.
\end{itemize}
For two (or more) clones $\clonename{X}, \clonename{Y}$ in this list, we use the abbreviation
\begin{align*}
\clonename{X}\clonename{Y} := \clonename{X} \cap \clonename{Y}.
\end{align*}
Thus for instance,
$\Aff\Cons01 = \Pol \{+^3, 0\} = \Pol \{+\}$ consists of the linear functions.

Note that we use these symbols for the ${<}\infty$-ary clones (see \cref{def:op}), consisting of \emph{all} functions preserving said relations.
We may restrict to functions of a certain arity by writing e.g., $\Mono^{<\omega}$ for the finitary monotone functions.
\end{definition}

\begin{theorem}[Post \cite{Post}]
\label{thm:post}
The lattice $\Clo^{<\omega}{2}$ of finitary clones on $2$ is countable, and consists of precisely the finitary restrictions of the following clones:
\begin{itemize}
\item
The ``cube'' of 8 clones formed from intersecting all combinations of the 3 clones $\Mono, \Cons01, \Cons11$.
\item
The 8 infinite series, consisting of the $\omega+1$ clones $\Cons0k$, where $k$ is one of $2, 3, \dotsc, {<}\omega$; the intersections of these with $\Mono, \Cons11$, or both; and the de~Morgan duals of these (consisting of $\Cons1k$ and their intersections with $\Mono$ and/or $\Cons01$).
\item
The 8 clones consisting of $\Meet$; its intersections with $\Cons01$ and/or $\Cons11$; and their de~Morgan duals.
\item
The 11 subclones of $\Aff$: $\Aff, \Aff\Dual, \Aff\Cons01, \Aff\Cons11, \Aff\Dual\Cons01 = \Aff\Dual\Cons11$; and the clone $\ang{\neg, 0} = \ang{\neg, 1}$ of essentially unary functions and its subclones $\ang{\neg}, \ang{0, 1}, \ang{0}, \ang{1}$, and the trivial clone $\ang{\emptyset}$.
\item
The 3 subclones of $\Dual$ which are not subclones of $\Aff$: $\Dual$, $\Dual\Cons01 = \Dual\Cons11$, and $\Dual\Mono$.
\end{itemize}
The ordering between these clones, as well as generators for them, are depicted in \cref{fig:post}.
\end{theorem}

It follows from \cref{thm:pancake} that the task of analyzing $\Clo^{<\nu}{2}$ for higher cardinals $\nu > \omega$ largely reduces to analyzing the lattice $\Clo^{<\nu}_F{2}$ of ${<}\nu$-ary clones restricting to $F$ (\cref{def:clone-fiber}), separately for each finitary clone $F \in \Clo^{<\omega}{2}$ given in \cref{thm:post}, i.e., in \cref{fig:post}.

\subsection{The bounding and cross-sectioning operators}
\label{sec:post-ceil-floor-xsec}

In this and the following subsections, we consider several operators $\Op2 -> \Op2$, turning functions on $2$ into new functions.
We will use these to show some global structural relationships between distinct ``regions'' of (infinitary) Post's lattice; see \cref{fig:post-mod}.

First, recall the bounding operators $f |-> \ceil{f}, \floor{f}$ from \cref{def:ceil-floor}.

\begin{remark}
\label{rmk:ceil-floor}
We have the following simple identities:
\begin{align*}
\ceil{f \circ \vec{g}} &= \ceil{f} \circ (\pi_0, \ceil{\vec{g}}), &
\floor{f \circ \vec{g}} &= \floor{f} \circ (\pi_0, \floor{\vec{g}}) \\
\intertext{(where $\ceil{(g_0, g_1, \dotsc)} := (\ceil{g_0}, \ceil{g_1}, \dotsc)$ and similarly for $\floor{\vec{g}}$).
Also, clearly}
\ceil{f} &\in \ang{f, \wedge}, &
\floor{f} &\in \ang{f, \vee}.
\end{align*}
\end{remark}

\begin{lemma}
\label{thm:ceil-floor-gen}
For any regular cardinal $\nu$ and set of functions $F \subseteq \Op2^{<\nu}$, we have
\begin{align*}
\ang{\ceil{\ang{F}^{<\nu}}}^{<\nu} &= \ang{\ceil{F} \cup \{\wedge\}}^{<\nu}, &
\ang{\floor{\ang{F}^{<\nu}}}^{<\nu} &= \ang{\floor{F} \cup \{\vee\}}^{<\nu}
\end{align*}
(where $\ceil{F} := \{\ceil{f} \mid f \in F\}$).
\end{lemma}
\begin{proof}
We only prove the first identity; the second is dual.

$\supseteq$ follows easily from $\wedge = \ceil{\id} \in \ceil{\ang{F}^{<\nu}}$.

$\subseteq$ follows from $\ceil{\ang{F}^{<\nu}} \subseteq \ang{\ceil{F} \cup \{\wedge\}}^{<\nu}$, which follows from the fact that the preimage $\ceil{-}^{-1}(\ang{\ceil{F} \cup \{\wedge\}}^{<\nu})$ is closed under $\circ$ and contains all projections, by the above remark.
\end{proof}

\begin{lemma}
\label{thm:ceil-clone}
For any regular cardinal $\nu$ and ${<}\nu$-ary clone $\wedge \in F \subseteq \Op2^{<\nu}$, we have
\begin{align*}
F \cap \Cons0{<\nu} &= \ang{\ceil{F}}^{<\nu}.
\end{align*}
\end{lemma}
\begin{proof}
$\supseteq$ follows from \cref{rmk:ceil-floor}.
Conversely, for $f \in F \cap \Cons0{<\nu}$, we have $f \le \pi_i$ for some $i$ (recall \cref{def:post-clones}), whence $f(\vec{x}) = \ceil{f}(x_i, \vec{x})$ is in $\ang{\ceil{F}}$.
\end{proof}

We now consider an operator that forms a partial inverse to $\ceil{-}, \floor{-}$.

\begin{definition}
\label{def:xsec}
For $f : 2^{1+n} -> 2$ and a constant $c \in 2$, define the \defn{cross-section} $f_c : 2^n -> 2$ by
\begin{align*}
f_c(\vec{x}) := f(c, \vec{x}).
\end{align*}
These satisfy the identities
\begin{align*}
%\begin{aligned}
c^n &= (\pi^{1+n}_0)_c, &
\pi^n_i &= (\pi^{1+n}_{1+i})_c,
%\end{aligned}
&
%\begin{aligned}
f_c \circ \vec{g}_c &= (f \circ (\pi_0, \vec{g}))_c, &
f &= (f \circ (\pi_{1+i})_i)_c.
%\end{aligned}
\end{align*}
\end{definition}

\begin{lemma}
\label{thm:post-const-xsec}
For a regular cardinal $\nu$, ${<}\nu$-ary clone $F \subseteq \Op2^{<\nu}$, and $c \in 2$,
\begin{align*}
\ang{F \cup \{c\}}^{<\nu} = \{f_c \mid f \in F^{\ge2}\}.
\end{align*}
\end{lemma}
(Here, following \cref{def:nullary}, $c$ on the left denotes a positive-arity constant function.)
\begin{proof}
$\supseteq$ is obvious; $\subseteq$ follows from the above identities
which ensure that the right-hand side is a clone containing $F \cup \{c\}$.
\end{proof}

\begin{remark}
\label{rmk:xsec-ceil-floor}
The above operators are related via
\begin{align*}
\floor{f}_0 &= f, &
\floor{f_0}(x_0, x_1, \dotsc) &= x_0 \vee f(0, x_1, \dotsc) = x_0 \vee f(x_0, x_1, \dotsc), \\
\ceil{f}_1 &= f, &
\ceil{f_1}(x_0, x_1, \dotsc) &= x_0 \wedge f(1, x_1, \dotsc) = x_0 \wedge f(x_0, x_1, \dotsc).
\end{align*}
\end{remark}

\begin{proposition}
\label{thm:post-T0inf-mod}
Let $\nu$ be a regular cardinal.
The composite of the modularity adjunctions \cref{eq:adj-mod} between the \rlap{three intervals}
\begin{equation*}
\begin{tikzcd}
{} [\ang{\wedge}^{<\nu}, \Op2^{<\nu}]
    \rar[shift left=2, two heads, "\qquad\qquad\qquad F \mapsto \ang{F \cup \{1\}}^{<\nu} = \{f_1 \mid f \in F^{\ge2}\}"{yshift=1ex}]
    \dar[shift left=0, two heads, right adjoint', "\substack{G \\ \mapsdown \\ G \cap \Cons0{<\nu} \mathrlap{= \ang{\ceil{G}}^{<\nu}}}"] &[1em]
{} [\ang{\wedge, 1}^{<\nu}, \Op2^{<\nu}]
    \lar[shift left=2, hook, right adjoint']
%    \dlar[shift left=4]
\\[1em]
{} [\ang{\wedge}^{<\nu}, \Cons0{<\nu}^{<\nu}]
    \uar[shift left=4, hook]
%    \urar[]
%    \urar[shift right=2, phantom, "\scriptstyle\cong"{sloped}]
    \urar[bend right=30, <->, "\cong"']
\end{tikzcd}
\end{equation*}
is an isomorphism.
(Here the $`->$ arrows are inclusions.)

Moreover, the composite $F |-> \ang{F \cup \{1\}}^{<\nu} |-> \ang{F \cup \{1\}}^{<\nu} \cap \Cons0{<\nu} : [\ang{\wedge}^{<\nu}, \Op2^{<\nu}] -> [\ang{\wedge}^{<\nu}, \Cons0{<\nu}^{<\nu}]$ is equal to $G |-> G \cap \Cons0{<\nu}$, which thus has a further right adjoint $F |-> \ang{F \cup \{1\}}^{<\nu}$.
\end{proposition}
\begin{proof}
We prove the last statement first.
For a clone $\wedge \in F \subseteq \Op2^{<\nu}$, we have
$\ang{F \cup \{1\}}^{<\nu} \cap \Cons0{<\nu}
= \ang{\ceil{\{f_1 \mid f \in F^{\ge2}\}}}^{<\nu}$
(by \cref{thm:ceil-clone,thm:post-const-xsec}),
which by the preceding remark is contained in $F \cap \Cons0{<\nu}$ (since $\ceil{f_1} \in \ang{f, \wedge} \subseteq F$); the other inclusion is obvious.

Now to finish proving that $[\ang{\wedge}^{<\nu}, \Cons0{<\nu}^{<\nu}] \cong [\ang{\wedge, 1}^{<\nu}, \Op2^{<\nu}]$, it remains to show that for $G$ in the latter interval, we have $G \subseteq \ang{\ceil{G} \cup \{1\}}^{<\nu}$, which again follows from the preceding remark.
\end{proof}

\Cref{thm:post-T0inf-mod} shows that, analogously to \cref{thm:pancake}, the entire interval $[\ang{\wedge}^{<\nu}, \Op2^{<\nu}]$ may be regarded as a ``bundle'' over the intervals $[\ang{\wedge}^{<\nu}, \Cons0{<\nu}^{<\nu}] \cong [\ang{\wedge, 1}^{<\nu}, \Op2^{<\nu}]$, which contain respectively the least and greatest elements in each fiber.
This is depicted in \cref{fig:post-mod-T0inf} for $\nu = \omega$: the bottommost and topmost shaded intervals are isomorphic, and everything between them is between two clones corresponding to each other in the bottom interval and the top interval.

\subsection{The downward-closure of a clone}
\label{sec:post-down}

We now extend \cref{thm:post-T0inf-mod} to show that in fact,  ``most'' of $[\ang{\wedge}^{<\nu}, \Op2^{<\nu}]$ decomposes as a product of two intervals, i.e., ``most'' clones containing $\wedge$ are determined by two ``orthogonal projections'', one onto $[\ang{\wedge}^{<\nu}, \Cons0{<\nu}^{<\nu}] \cong [\ang{\wedge, 1}^{<\nu}, \Op2^{<\nu}]$ as in \ref{thm:post-T0inf-mod}, and the other given as follows.

\begin{definition}
\label{def:down}
For a class of functions $G \subseteq \Op2$, let $\down G$ denote its \defn{pointwise downward-closure}:
\begin{align*}
\down G := \set{f : 2^n -> 2}{\exists g \in G^n\, (f \le g)}.
\end{align*}
For example:
\begin{itemize}
\item  $\down \ang{\emptyset}^{<\nu} = \set{f : 2^n -> 2}{n < \nu,\, \exists i < n\, (f \le \pi_i)} = \Cons0{<\nu}^{<\nu}$.
%\item  If $1^n \in G$, then $\Op2^n \subseteq \down G$.
\item  $\down \ang{\Mono\Cons01\Cons11^{<\omega}}^{<\omega_1} = \Limm011^{<\omega_1}$.
(Recall \cref{def:post-clones}.
Given $f \in \Limm011^\omega$, i.e., $f : 2^\omega -> 2$ which is $0$ on a neighborhood of $\vec{0} \in 2^\omega$, which we may assume to be clopen, downward-closed and $\subsetneq 2^\omega$, the indicator function $g$ of the complement of that neighborhood is in $\ang{\Mono\Cons01\Cons11^{<\omega}}$.)
\end{itemize}
\end{definition}

\begin{lemma}
\label{thm:down-clone}
If $G \subseteq \Mono$ is a clone, then so is $\down G$.
\end{lemma}
\begin{proof}
Let $f \le g \in G^n$, and $f_i \le g_i \in G^m$ for each $i < n$.
Then $f \circ \vec{f} \le g \circ \vec{f} \le g \circ \vec{g} \in G^m$.
\end{proof}

\begin{remark}
Clearly $G \subseteq F \subseteq \down G \implies \down F = \down G$.
Thus if $G \subseteq \Op2$ is such that $\down G$ is a clone, then so is $\down F$ for every $F \in [G, \down G]$.
In particular, if $F \subseteq \down (F \cap \Mono)$, then $\down F$ is a clone.
\end{remark}

The following generalizes \cref{thm:ceil-clone}, which is the case $G = \ang{\emptyset}^{<\nu}$:

\begin{lemma}
\label{thm:ceil-down-clone}
For any regular cardinal $\nu$, $G \subseteq \Op2^{<\nu}$, and ${<}\nu$-ary clone $G \cup \{\wedge\} \subseteq F \subseteq \Op2^{<\nu}$,
\begin{align*}
\ang{F \cap \down G}^{<\nu} = \ang{\ceil{F} \cup G}^{<\nu}.
\end{align*}
\end{lemma}
\begin{proof}
$\supseteq$ follows from \cref{rmk:ceil-floor}.
Conversely, for $f \in F \cap \down G$, we have $f \le g$ for some $g \in G$, whence $f(\vec{x}) = g(\vec{x}) \wedge f(\vec{x}) = \ceil{f}(g(\vec{x}), \vec{x})$ is in $\ang{\ceil{F}}$.
\end{proof}

\begin{corollary}
\label{thm:down-T0inf}
For any regular cardinal $\nu$ and $G \subseteq \Op2^{<\nu}$,
\begin{align*}
\ang{\down G}^{<\nu} = \ang{\Cons0{<\nu}^{<\nu} \cup G}^{<\nu}.
\end{align*}
\end{corollary}
\begin{proof}
Take $F = \Op2^{<\nu}$ above.
\end{proof}

\begin{corollary}
\label{thm:post-lower}
For any regular cardinal $\nu$ and ${<}\nu$-ary clone $\wedge \in G \subseteq \Op2^{<\nu}$,
\begin{align*}
\ang{G \cup \{1\}}^{<\nu} \cap \down G = G.
\end{align*}
\end{corollary}
This says that $G$ is ``downward-closed within its fiber from \cref{thm:post-T0inf-mod}''.
\begin{proof}
Take $F = \ang{G \cup \{1\}}^{<\nu}$ above, noting that $\ang{\ceil{\ang{G \cup \{1\}}^{<\nu}}}^{<\nu} = G \cap \Cons0{<\nu}$ by \ref{thm:post-T0inf-mod}.
\end{proof}

Summarizing the preceding results, we have:

\begin{proposition}
\label{thm:post-T0inf-down}
Let $\nu$ be a regular cardinal.
We have an adjunction
\begin{equation*}
\begin{tikzcd}[column sep=1em]
&& \scriptstyle (\ang{F \cup \{1\}}^{<\nu}, \ang{\down F}^{<\nu})
    \dar[phantom, "\scriptstyle\in"{sloped}]
\\[-1em]
\scriptstyle (F \cap \Cons0{<\nu}, \ang{\down F}^{<\nu})
    \rar[phantom, "\scriptstyle\in"{sloped}]
    \ar[urr, mapsto, bend left=5]
&
{} [\ang{\wedge}^{<\nu}, \Cons0{<\nu}^{<\nu}] \times [\Cons0{<\nu}^{<\nu}, \Op2^{<\nu}] \rar[phantom, "\cong"]
    \dar[shift left=2, right adjoint'] &
|[label={[rotate=-90,anchor=west,inner sep=0pt,label={[name=GH]right:\scriptstyle(G,H)}]below:\scriptstyle\ni}]|
{} [\ang{\wedge, 1}^{<\nu}, \Op2^{<\nu}] \times [\Cons0{<\nu}^{<\nu}, \Op2^{<\nu}]
\\[-1em]
\scriptstyle F \rar[phantom, "\scriptstyle\in"] \ar[u, mapsto] &
{} [\ang{\wedge}^{<\nu}, \Op2^{<\nu}]
    \uar[shift left=2] &
|[xshift=-5em]|
\scriptstyle G \cap H
    \lar[phantom, "\scriptstyle\ni"]
    \ar[to=GH, mapsfrom]
\end{tikzcd}
\end{equation*}
which restricts to an order-embedding on those $F \in [\ang{\wedge}^{<\nu}, \Op2^{<\nu}]$ such that $\down F$ is a clone.
This includes all monotone $F$, and more generally, all $F \in [G, \down G]$ for some subclone $G \subseteq \Mono^{<\nu}$.
\qed
\end{proposition}

In other words, every such clone $F$ is determined by its two ``projections'' or ``coordinates'' $F \cap \Cons0{<\nu}$ and $\down F$, in the intervals $[\ang{\wedge}^{<\nu}, \Cons0{<\nu}^{<\nu}]$ and $[\Cons0{<\nu}^{<\nu}, \Op2^{<\nu}]$.
This gives a detailed analysis of the ``tube'' on the left side of Post's lattice \ref{fig:post} and its ${<}\nu$-ary version; see \cref{fig:post-mod-down}.

Some particular consequences of \cref{thm:post-T0inf-down} are:

\begin{corollary}
\label{thm:post-down-emb}
For any regular cardinal $\nu$ and ${<}\nu$-ary clones $\wedge \in G \subseteq \Mono\Cons0{<\nu}^{<\nu}$ and $G \subseteq H \subseteq \Cons0{<\nu}^{<\nu}$, the modularity adjunction
\begin{align*}
[G, \ang{G \cup \{1\}}^{<\nu}] &\rightleftarrows [H, \ang{H \cup \{1\}}^{<\nu}] \\
F &|-> \ang{F \cup H}^{<\nu}
\end{align*}
exhibits the left interval as a retract of the right.
\qed
\end{corollary}

This says that the ``vertical fibers'' of the image of the embedding in \ref{thm:post-T0inf-down} (i.e., the fibers of the bundle \ref{thm:post-T0inf-mod}) are ``increasing as the $x$-coordinate (in $[\ang{\wedge}^{<\nu}, \Cons0{<\nu}^{<\nu}]$) increases''.

\begin{corollary}
\label{thm:post-down-mod}
Let $\nu$ be a regular cardinal, $\wedge \in H \subseteq \Mono^{<\nu}$ be a ${<}\nu$-ary clone.
Then we have modularity isomorphisms between the intervals
\begin{equation*}
\begin{alignedat}[b]{2}
[H \cap \Cons0{<\nu}, \Cons0{<\nu}^{<\nu}]
&\cong [H, \down H]
&&\cong [\ang{H \cup \{1\}}^{<\nu}, \Op2^{<\nu}] \\
F &|-> \ang{F \cup H}^{<\nu} &&|-> \ang{F \cup H \cup \{1\}}^{<\nu} = \ang{F \cup \{1\}}^{<\nu} \\
G \cap \Cons0{<\nu} = \ang{\ceil{G}}^{<\nu} &<-| && \mathllap{G \cap \down H} <-| G.
\end{alignedat}
\qed
\end{equation*}
\end{corollary}

This says that the ``horizontal slices'' of \ref{thm:post-T0inf-down}, with fixed ``$y$-coordinate $\down H \in [\Cons0{<\nu}^{<\nu}, \Op2^{<\nu}]$'', are isomorphic to the top and bottom slices between which they are sandwiched.
For example, taking $H = \ang{\Mono\Cons01\Cons11^{<\omega}}^{<\omega_1} = \ang{\wedge,\vee}^{\omega_1}$, so that $\down H = \Limm011^{<\omega_1}$ (see \cref{def:down}), we get
\begin{alignat*}{2}
\yesnumber
\label{eq:post-down-MT01T11-mod}
[\ang{\ceil{\vee}}^{<\omega_1}, \Cons0\omega^{<\omega_1}]
&\cong [\ang{\wedge,\vee}^{<\omega_1}, \Limm011^{<\omega_1}]
&&\cong [\ang{\wedge,\vee,1}^{<\omega_1}, \Op2^{<\omega_1}] \\
F &|-> \ang{F \cup \{\vee\}}^{<\omega_1} &&|-> \ang{F \cup \{\vee,1\}}^{<\omega_1} = \ang{F \cup \{1\}}^{<\omega_1} \\
G \cap \Cons0\omega = \ang{\ceil{G}}^{<\omega_1} &<-| && \mathllap{G \cap \Limm011} <-| G.
\end{alignat*}
When instead $H = \Mono\Cons02\Cons11^{<\omega}$, this yields the slice depicted in \cref{fig:post-mod}\subref{fig:post-mod-T0inf} and \subref{fig:post-mod-down}.

\begin{remark}
When $\nu = \omega$, by an inspection of finitary Post's lattice \ref{fig:post}, we see that \cref{thm:post-T0inf-down} in fact applies to all finitary clones $F \ni \wedge$, i.e., $\down F = \ang{F \cup \Cons0{<\omega}^{<\omega}}^{<\omega}$ is always a clone.
It does not seem possible to prove this in full generality for $\nu > \omega$.
However, in the context of ${<}\omega_1$-ary Borel clones, it turns out that \cref{thm:post-T0inf-down} does always apply; see \cref{thm:post-borel-T0inf-down}.
\end{remark}

In determining for which clones $F \ni \wedge$ is $\down F$ also a clone, the following notion can be useful:

\begin{definition}
\label{def:up}
For a function $f : 2^n -> 2$, let $\up f : 2^n -> 2$ denote the indicator function of the \defn{upward-closure} $\up f^{-1}(1)$ of $f^{-1}(1) \subseteq 2^n$, namely
\begin{align*}
(\up f)(\vec{x}) := \bigvee_{\vec{y} \le \vec{x}} f(\vec{y}).
\end{align*}
For a class of functions $F \subseteq \Op2$, let $\up[F] := \{\up f \mid f \in F\}$.
(Note that this is \emph{not} dual to the pointwise downward-closure $\down G$ of a set of functions from \cref{def:down}, which refers to the pointwise ordering on functions $2^n -> 2$, whereas this refers to the bitwise ordering on $2^n$.)
\end{definition}

The upward-closure $\up f$ is the least monotone function $\ge f$.
It follows that for any $F \subseteq \Op2$,
\begin{align}
\label{eq:up}
F \cap \Mono \subseteq \up[F] \subseteq \Mono, &&
F \cap \Mono = F \cap \up[F], &&
F \subseteq \down \up[F].
\end{align}
Thus if $\up[F] \subseteq F$, then $G := \up[F] \subseteq \Mono$ obeys $G \subseteq F \subseteq \down G$, and so $\down F = \down G$ is a clone.
(It is not always true that $\up[F] \subseteq F$ for a clone $F$; see \cref{rmk:post-up-analytic}.)

\begin{figure}
\centering
\subcaptionbox{%
    \cref{thm:post-T0inf-mod}%
    \label{fig:post-mod-T0inf}%
}{\includegraphics[valign=t]{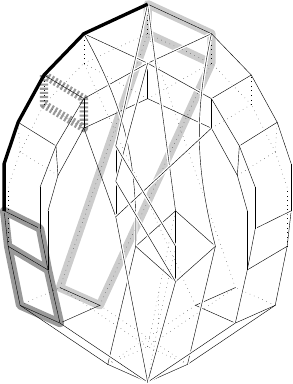}}
\hfill
\subcaptionbox{%
    \cref{thm:post-T0inf-down}%
    \label{fig:post-mod-down}%
}{\includegraphics[valign=t]{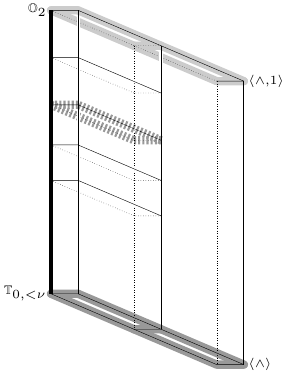}}
\hfill
\subcaptionbox{%
    \cref{thm:post-D-beta-mod}%
    \label{fig:post-mod-D}%
}{\includegraphics[valign=t]{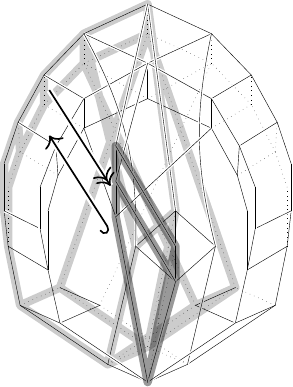}}
\caption{Modularity isomorphisms/retractions between various subintervals of Post's lattice.}
\label{fig:post-mod}
\end{figure}

\subsection{The self-dualizing operator}
\label{sec:post-beta}

Next, we consider a different region of Post's lattice on which the cross-sectioning operators yield an isomorphism: the self-dual functions.
(This is the same approach used to classify the self-dual clones in the original finitary Post's lattice in e.g., \cite[\S3.2.3]{Lau}.)

\begin{definition}
\label{def:beta}
For $f : 2^n -> 2$, define $\beta(f) : 2^{1+n} -> 2$ by
\begin{align*}
\beta(f)(x_0, \vec{x}) := (x_0 ? \delta(f)(\vec{x}) : f(\vec{x})).
\end{align*}
(Recall from \cref{def:post-funs} that $?:$ denotes the ternary conditional.)
\end{definition}

\begin{remark}
\label{rmk:beta}
It is easily seen that for any cardinal $n$, including $n=0$, $\beta$ is a bijection
\begin{equation*}
\begin{tikzcd}[column sep=1em]
2^{2^n} \rar[phantom, "\smash{\overset{\beta}\cong}\mathstrut"] &
\#D^{1+n} \rar[phantom, "\subseteq"] &
2^{2^{1+n}}
    \ar[ll, two heads, bend left=30, "f_0 \mapsfrom f"]
\end{tikzcd}
\end{equation*}
with retraction $(-)_0 : 2^{2^{1+n}} ->> 2^{2^n}$ (the cross-section from \cref{def:xsec}), i.e., for any $f : 2^{1+n} -> 2$,
\begin{align*}
f \in \#D^{1+n}
\iff  f \in \im(\beta)
\iff  f = \beta(f_0).
\end{align*}
\end{remark}

\begin{proposition}
\label{thm:post-D-beta-mod}
For any regular cardinal $\nu$, we have an order-isomorphism
\begin{align*}
[\ang{\emptyset}^{<\nu}, \Dual^{<\nu}] &\cong \{G \in \Clo^{<\nu}{2} \mid 0 \in G^1,\, \beta(G) \subseteq G\} \subseteq [\ang{0}^{<\nu}, \Op2^{<\nu}] \\
F &|-> \ang{F \cup \{0\}}^{<\nu} = \{f_0 \mid f \in F^{\ge2}\} \\
G \cap \Dual = \ang{\beta(G)}^{<\nu} &<-| G.
\end{align*}
\end{proposition}
\begin{proof}
By \cref{rmk:beta}, $F |-> F_0 := \{f_0 \mid f \in F\}$ is a bijection between all subsets of $\Dual$ and arbitrary sets of functions on $2$, including nullary functions $f_0$ which occur when $f \in F$ is unary.
But for a clone $F$ and unary $f \in F$, the unary constant function $x |-> f(0) = f_0$ also appears as $(f \circ \pi^2_0)_0$; thus $F_0$ is completely determined by the positive-arity functions $F_0^{\ge1} \subseteq F_0$, namely a nullary function is in $F_0$ iff the corresponding unary constant function is in $F_0^{\ge1}$.
Hence, $F |-> \ang{F \cup \{0\}}^{<\nu} = F_0^{\ge1}$ (by \cref{thm:post-const-xsec}) remains injective on clones $F$, with inverse (by \cref{rmk:beta}) taking $G = F_0^{\ge1}$ to $\beta(F_0)$, which can be obtained from $\beta(G) = \beta(F_0^{\ge1})$ by adding a unary function $f$ whenever the corresponding binary function $f'(x,y) := f(x)$ is in $\beta(G)$; since $f(x) = f'(x,x)$, this means $\beta(F_0) = \ang{\beta(G)}^{<\nu}$.
This proves that $F |-> \ang{F \cup \{0\}}^{<\nu}$ is injective, with inverse on its image given by $G |-> \ang{\beta(G)}^{<\nu}$, which must hence equal the right adjoint $G |-> G \cap \Dual$ (recall again \ref{eq:adj-mod}).

For a clone of the form $G = \ang{F \cup \{0\}}^{<\nu}$, we clearly have $0 \in G^1$ and $\beta(G) \subseteq \ang{\beta(G)}^{<\nu} = F \subseteq G$.
Conversely, if $G \in \Clo^{<\nu}{2}$ satisfies these two conditions, then clearly $\ang{\ang{\beta(G)}^{<\nu} \cup \{0\}}^{<\nu} \subseteq G$, and again by \cref{rmk:beta,thm:post-const-xsec}, $G = G^{\ge1} = \beta(G)_0^{\ge1} \subseteq (\ang{\beta(G)}^{<\nu})_0^{\ge1} = \ang{\ang{\beta(G)}^{<\nu} \cup \{0\}}^{<\nu}$.
This shows that the image of the isomorphism is as claimed.
\end{proof}

Thus, the classification of self-dual clones reduces to the classification of clones containing $0$ and closed under the $\beta$ operator.
See \cref{fig:post-mod-D}, which should probably be viewed only impressionistically as it is rather difficult to tell at a glance which clones are closed under $\beta$.
In the finitary Post's lattice \ref{fig:post}, the three non-affine clones containing $0$ and closed under $\beta$ are
\begin{align}
\label{eq:post-D-beta-clones}
\ang{\Dual^{<\omega} \cup \{0\}}^{<\omega} &= \Op2^{<\omega}, &
\ang{\Dual\Cons01^{<\omega} \cup \{0\}}^{<\omega} &= \Cons01^{<\omega}, &
\ang{\Dual\Mono^{<\omega} \cup \{0\}}^{<\omega} &= \Mono\Cons02^{<\omega}.
\end{align}
The following lemma can simplify checking closure of a clone under $\beta$ in general:

\begin{lemma}
\label{thm:post-D-beta-gen}
For any set of functions $G \subseteq \Op2^{<\nu}$, we have
\begin{equation*}
\beta(\ang{G}^{<\nu}) \subseteq \ang{\beta(G)}^{<\nu}.
\end{equation*}
Thus, $\beta(G) \subseteq \ang{G}^{<\nu}$ iff $\beta(\ang{G}^{<\nu}) \subseteq \ang{G}^{<\nu}$.
\end{lemma}
\begin{proof}
By \cref{rmk:beta}, $\beta^{-1}(\ang{\beta(G)}^{<\nu}) = (\ang{\beta(G)}^{<\nu})_0 \supseteq \beta(G)_0 = G$, whence by \cref{thm:post-const-xsec}, $\beta^{-1}(\ang{\beta(G)}^{<\nu})^{\ge1} = (\ang{\beta(G)}^{<\nu})_0^{\ge1} = \ang{\ang{\beta(G)}^{<\nu} \cup \{0\}}^{<\nu}$ is a clone containing $G$.
\end{proof}

\section{Borel clones on $2$}
\label{sec:borel}

\begin{definition}
Recall that a \defn{Borel set} $A \subseteq 2^n$, where $n \le \omega$, is a set obtained as a countable Boolean combination of the subbasic clopen sets $\pi_i^{-1}(1)$ for each $i < n$.
A \defn{Borel function} $f : 2^n -> 2$ is one such that $f^{-1}(0), f^{-1}(1)$ are Borel sets, i.e., $f$ is the indicator function of a Borel set, i.e., $f$ is obtained by applying $\bigwedge^\omega, \bigvee^\omega, \neg$ pointwise to the $\pi_i$.
Thus
\begin{align*}
\Op2^\Borel := \ang{\bigwedge, \bigvee, \neg}^{<\omega_1} \subseteq \Op2^{<\omega_1}
\end{align*}
is the ${<}\omega_1$-ary clone of all Borel functions of countable arity on $2$.
\end{definition}

We are interested in classifying subclones of $\Op2^\Borel$, which we call \defn{Borel clones} on $2$.
As such, we adopt the following

\begin{notation}
\label{def:clone-borel}
We will treat the superscript ${}^\Borel$ for ``Borel'' as if it were a class of arities, intermediate between finite (${<}\omega$) and countable (${<}\omega_1$).
Thus, in analogy with \cref{def:op}, for a class of functions $F \subseteq \Op2$, we write
\begin{equation*}
F^\Borel := F \cap \Op2^\Borel
\end{equation*}
for the Borel functions in $F$.
The following are some example uses of this notation:
\begin{itemize}
\item
$\Mono^\Borel \subseteq \Op2^\Borel$ denotes the clone of monotone Borel functions $2^n -> 2$, $n \le \omega$ (recall \cref{def:post-clones}).
\item
If $\@M \subseteq \Rel2$ is a class of relations on $2$, then $\Pol^\Borel(\@M) = \Pol(\@M) \cap \Op2^\Borel = \Pol^{<\omega_1}(\@M) \cap \Op2^\Borel$ consists of all Borel polymorphisms of $\@M$.
\item
If $F \subseteq \Op2^\Borel$ is a class of Borel functions, then $\ang{F}^\Borel = \ang{F} \cap \Op2^\Borel = \ang{F}^{<\omega_1}$, while $\-{\ang{F}}^\Borel = \-{\ang{F}} \cap \Op2^\Borel$ consists of the Borel functions which are pointwise limits of functions in $\ang{F}^\Borel$.
Note that if $F = \Pol^{<\omega}(\@M)$ consists of all finitary polymorphisms of a class of relations, then $\-{\ang{F}}^\Borel = \Pol^\Borel(\Inv^{<\omega}(\Pol^{<\omega}(\@M))) = \Pol^\Borel(\@M)$ (by \cref{thm:pancake}); e.g., $\-{\ang{\Mono^{<\omega}}}^\Borel = \Mono^\Borel$.
\end{itemize}
We also write
\begin{equation*}
\Clo^\Borel{2} := [\ang{\emptyset}^\Borel, \Op2^\Borel] \subseteq \Clo^{<\omega_1}{2}
\end{equation*}
for the sublattice of all clones of Borel functions.
Thus, the ``bundle'' of \cref{thm:pancake} specializes to
\begin{equation}
\label{eq:pancake-borel}
\begin{tikzcd}
\Clo^\Borel{2}
    \dar[two heads, shift right=2, "{(-) \cap \Op2^{<\omega}}"]
\\[2em]
\Clo^{<\omega}{2}
    \uar[hook, shift left=6, "{\ang{-}^\Borel}", left adjoint']
    \uar[hook, bend right=60, looseness=2, shift right=4, "{\Pol^\Borel \circ \Inv^{<\omega} = \-{\ang{-}}^\Borel}"{right}, right adjoint]
\end{tikzcd}
\end{equation}
As in \cref{def:clone-fiber}, put
\begin{equation*}
\Clo^\Borel_F{2} := [\ang{F}^\Borel, \-{\ang{F}}^\Borel] \subseteq \Clo^\Borel{2}
\end{equation*}
for the fiber of this bundle over each finitary clone $F \in \Clo^{<\omega}{2}$ in Post's lattice.
\end{notation}

In the rest of this section, we will describe the structure of $\Clo^\Borel_F{2}$ for $F$ in various ``regions'' of Post's lattice $\Clo^{<\omega}{2}$ (see \cref{fig:post-borel}).

%\subsection{Affine functions}
%\label{sec:borel-affine}
We begin by dispensing with the simplest case: the affine (over $\#Z/2\#Z$) functions.
The following classical result is a special case of general ``automatic continuity'' results for well-behaved topological groups, and can be proved using either Haar measure (as the Steinhaus--Weil theorem) or Baire category (as Pettis's theorem).
See e.g., \cite[9.9]{Kcdst}, \cite{Racts}.

\begin{theorem}
\label{thm:pettis}
Let $A \subseteq 2^\omega$ be a Borel subgroup under addition mod $2$.
If $A$ has countable index, then $A$ is clopen.
\end{theorem}

\begin{corollary}
$\Aff^\Borel = \ang{\Aff^{<\omega}}^\Borel$, i.e., every affine Borel map $f : 2^n -> 2$ for $n \le \omega$ is continuous.
\end{corollary}
\begin{proof}
Either $f^{-1}(0)$ or $f^{-1}(1)$ is an index $\le2$ Borel subgroup of $2^n$.
\end{proof}

\begin{corollary}
For every finitary clone $F \subseteq \Aff^{<\omega}$, there is a unique Borel clone restricting to $F$, namely $\ang{F}^\Borel = \-{\ang{F}}^\Borel \in \Clo^\Borel_F{2}$.
\qed
\end{corollary}

This is illustrated in the large shaded region near the bottom of \cref{fig:post-borel}.

\subsection{The top cube}
\label{sec:borel-topcube}

We now turn to the Borel clones lying over one of the 8 finitary cones in the ``top cube'' of Post's lattice \ref{fig:post}, between $\Mono\Cons01\Cons11^{<\omega}$ and $\Op2^{<\omega}$.
We will give a complete classification of the corresponding Borel clones, of which there are only finitely many (see \cref{fig:post-borel}).
Our main tool will be the modularity isomorphisms from \cref{thm:post-down-mod}.

The following lemma consists of variations of the well-known fact that countable supremum is the ``simplest'' discontinuous function (the base case of Wadge's lemma; see e.g., \cite[21.16]{Kcdst}).

\begin{lemma}
\label{thm:post-wadge}
Let $f : 2^\omega -> 2$.
\begin{enumerate}[label=(\alph*)]
\item \label{thm:post-wadge:O2}
If $f$ is discontinuous, then $\bigvee \in \ang{\{f\} \cup \Cons1{<\omega}^{<\omega}}$.
\item \label{thm:post-wadge:T0}
If $f \in \Cons01 \setminus \Limm011$, then $\bigvee \in \ang{\{f\} \cup \Cons01\Cons1{<\omega}^{<\omega}}$.
\item \label{thm:post-wadge:M}
If $f \in \Mono \setminus \Decr$, then $\bigvee \in \ang{\{f\} \cup \Mono\Cons1{<\omega}^{<\omega}}$.
\item \label{thm:post-wadge:MT0}
If $f \in \Mono\Cons01 \setminus \Limm011$, then $\bigvee \in \ang{\{f\} \cup \Mono\Cons01\Cons1{<\omega}^{<\omega}}$.
\end{enumerate}
\end{lemma}
\begin{proof}
\cref{thm:post-wadge:O2}
Suppose $f$ is discontinuous at $\vec{x}_\infty \in 2^\omega$.
Then there is a sequence of strings $\vec{x}_0, \vec{x}_1, \dotsc \in 2^\omega$ converging to $\vec{x}_\infty$ such that $f(\vec{x}_0) = f(\vec{x}_1) = \dotsb \ne f(\vec{x}_\infty)$.
Define
\begin{align*}
g : 2^\omega &--> 2^\omega \\
1\dotsm &|--> \vec{1} \\
01\dotsm &|--> \vec{x}_0 \\
001\dotsm &|--> \vec{x}_1 \\
&\vdotswithin{|-->} \\
\vec{0} &|--> \vec{x}_\infty
\end{align*}
where each tail ``$\dotsm$'' may be an arbitrary string.
This is evidently a continuous function, such that each coordinate $g_i := \pi_i \circ g \ge \pi_0 : 2^\omega -> 2$, hence is in $\ang{\Cons1{<\omega}^{<\omega}}$; thus $f \circ g = f \circ (g_i)_i \in \ang{\{f\} \cup \Cons1{<\omega}^{<\omega}}$.
If $f(\vec{x}_\infty) = 0$, then $f \circ g$ agrees with $\bigvee : 2^\omega -> 2$ on all strings of the form $0\dotsm$, whence ${\bigvee} = \pi_0 \vee (f \circ g) \in \ang{\{f\} \cup \Cons1{<\omega}^{<\omega}}$.
Otherwise, ${\bigvee} = \pi_0 \vee \neg (f \circ g) = (f \circ g) -> \pi_0 \in \ang{\{f\} \cup \Cons1{<\omega}^{<\omega}}$.

\cref{thm:post-wadge:T0}
Let $\vec{x}_0, \vec{x}_1, \dotsc \in f^{-1}(1)$ converge to $\vec{x}_\infty := \vec{0} \in f^{-1}(0)$.
Then the function $g$ defined above preserves $0$ coordinatewise, and we know $f(\vec{x}_\infty) = 0$, whence ${\bigvee} = \pi_0 \vee (f \circ g) \in \ang{\{f\} \cup \Cons01\Cons1{<\omega}^{<\omega}}$.

\cref{thm:post-wadge:M}
Let $\vec{x}_0 \ge \vec{x}_1 \ge \dotsb$ converge to $\vec{x}_\infty \in 2^\omega$ with $f(\vec{x}_0) = f(\vec{x}_1) = \dotsb > f(\vec{x}_\infty)$.
Then $g$ defined above is monotone, and $f(\vec{x}_\infty) = 0$, whence ${\bigvee} = \pi_0 \vee (f \circ g) \in \ang{\{f\} \cup \Mono\Cons1{<\omega}^{<\omega}}$.

\cref{thm:post-wadge:MT0}
Let $\vec{x}_0, \vec{x}_1, \dotsc \in f^{-1}(1)$ converge to $\vec{x}_\infty := \vec{0} \in f^{-1}(0)$.
Then $\vec{x}'_i := \bigvee_{j \ge i} \vec{x}_j$ converge monotonically to $\vec{0}$, and $f(\vec{x}'_i) \ge f(\vec{x}_i) = 1$.
Combine the arguments of \cref{thm:post-wadge:T0,thm:post-wadge:M}.
\end{proof}

\begin{corollary}
\label{thm:post-borel-O2}
There are precisely 2 Borel clones on $2$ restricting to $\Op2^{<\omega}$ (see \cref{fig:post-borel-O2-T0}):
\begin{itemize}
\item  $\Op2^\Borel = \ang{\bigvee, \neg}^\Borel = \ang{\bigwedge, \bigvee, 0, 1, \neg}^\Borel$.
\item  $\ang{\Op2^{<\omega}}^\Borel = \Pol^\Borel\brace{\lim} = \ang{\vee, \neg}^\Borel = \ang{\wedge, \vee, 0, 1, \neg}^\Borel$.
\end{itemize}
\end{corollary}
\begin{proof}
If $\ang{\Op2^{<\omega}}^\Borel \ne F \in \Clo^\Borel_{\Op2^{<\omega}}{2}$, then $F$ contains a discontinuous function, hence contains $\bigvee$ by \cref{thm:post-wadge}\cref{thm:post-wadge:O2}, hence contains all Borel functions since it contains $\neg \in \Op2^{<\omega}$.
\end{proof}

\begin{figure}
\centering
\includegraphics{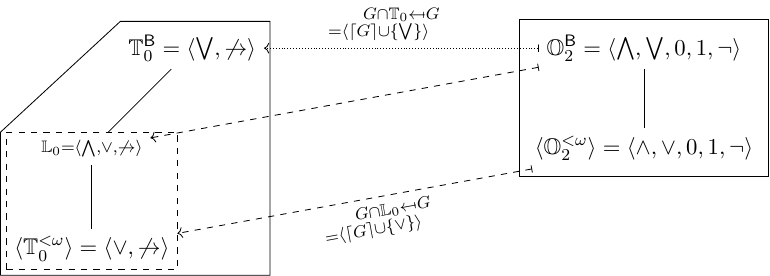}
\caption{Borel clones with finitary restrictions $\Cons01^{<\omega}$ and $\Op2^{<\omega}$, respectively, along with modularity isomorphisms from \cref{thm:post-down-mod} used to deduce the former from the latter.
See also \cref{fig:post-borel}.
(Here, and in the below figures, all clones are ${<}\omega_1$-ary and restricted to Borel functions, unless otherwise indicated.)}
\label{fig:post-borel-O2-T0}
\end{figure}

\begin{theorem}
\label{thm:post-borel-T0}
There are precisely 3 Borel clones on $2$ restricting to $\Cons01^{<\omega}$ (see \cref{fig:post-borel-O2-T0}):
\begin{itemize}
\item  $\Cons01^\Borel = \Pol^\Borel\brace{0} = \ang{\bigvee, -/>}^\Borel = \ang{\bigwedge, \bigvee, 0, -/>}^\Borel$.
\item  $\Limm011^\Borel = \Pol^\Borel\brace{{\lim}{=}0} = \ang{\bigwedge, \vee, -/>}^\Borel = \ang{\ceil{\bigvee}, \vee, -/>}^\Borel = \ang{\bigwedge, \ceil{\bigvee}, \vee, 0, -/>}^\Borel$.
\item  $\ang{\Cons01^{<\omega}}^\Borel = \Pol^\Borel\brace{{\lim}, 0} = \ang{\vee, -/>}^\Borel = \ang{\wedge, \vee, 0, -/>}^\Borel$.
\end{itemize}
\end{theorem}
\begin{proof}
By \cref{thm:post-wadge}\cref{thm:post-wadge:T0}, the greatest clone in $\smash{\Clo^\Borel_{\Cons01^{<\omega}}{2}}$ below the maximum $\Cons01^\Borel$ is $\Limm011^\Borel$ (which is indeed below the maximum since $\bigvee : 2^\omega -> 2$ preserves $0$ but is discontinuous at $\vec{0}$).

By \cref{thm:post-down-mod} with
$H = \ang{\Mono\Cons01\Cons11^{<\omega}}^{<\omega_1} = \ang{\wedge,\vee}^\Borel$
and
$\down H = \Limm011^{<\omega_1}$
as in \cref{eq:post-down-MT01T11-mod}, the clones between $\ang{\Cons01^{<\omega}}^\Borel$ and $\Limm011^\Borel$ are in order-preserving bijection, via $F |-> \ang{F \cup \{1\}}^\Borel$, with the clones in $\Clo^\Borel_{\Op2^{<\omega}}{2}$
(since $\ang{\Cons01^{<\omega} \cup \{1\}}^\Borel = \ang{\vee, -/>, 1}^\Borel = \ang{\Op2^{<\omega}}^\Borel$
and $\ang{\Limm011^\Borel \cup \{1\}}^\Borel \supseteq \ang{\bigwedge, -/>, 1}^\Borel = \Op2^\Borel$).
The inverse of this bijection takes $G \in \Clo^\Borel_{\Op2^{<\omega}}{2}$ to $G \cap \Limm011$, which can also be obtained (see again \ref{eq:post-down-MT01T11-mod}) by first restricting all the way to $G \cap \Cons0\omega = \ang{\ceil{G}}^\Borel \ni \wedge$, then adding $\vee$ to yield $\ang{\ceil{G} \cup \{\vee\}}^\Borel$.
Applying this to the two clones in \cref{thm:post-borel-O2} yields
\begin{align*}
\Op2^\Borel \cap \Limm011
&= \Pol^\Borel \brace{{\lim}{=}0}
= \ang{\ceil{\bigvee}, \ceil{\neg}, \vee}^\Borel
= \ang{\ceil{\bigvee}, -/>, \vee}^\Borel
    &&\text{(by \cref{thm:ceil-floor-gen})} \\%(since $\Op2^\Borel = \ang{\bigvee, \neg}^\Borel$)} \\
&\hphantom{{}= \Pol^\Borel \brace{{\lim}{=}0}}
= \ang{\ceil{\bigwedge}, \ceil{\neg}, \vee}^\Borel
= \ang{\bigwedge, -/>, \vee}^\Borel
    &&\text{(by \cref{thm:ceil-floor-gen})}, \\%(since $\Op2^\Borel = \ang{\bigwedge, \neg}^\Borel$)}, \\
\ang{\Op2^{<\omega}}^\Borel \cap \Limm011
&= \Pol^\Borel \brace{\lim, 0}
= \ang{\Cons01^{<\omega}}^\Borel
    &&\text{(by \cref{rmk:ktop})}.
\end{align*}
Thus, these are the two clones in $\Clo^\Borel_{\Cons01^{<\omega}}{2}$ below the maximum.

Finally, to see that the maximum $\Cons01^\Borel$ is generated by $\bigvee, -/>$, we can again apply \cref{thm:post-down-mod} but with $H = \ang{\wedge,\bigvee}^\Borel$, which clearly bounds all $\Cons01^{<\omega_1}$ functions, to get that the clones between $\ang{\bigvee, -/>}^\Borel$ and $\Cons01^\Borel$ correspond to the clones above $\ang{\bigvee, -/>, 1}^\Borel = \Op2^\Borel$, whence $\ang{\bigvee, -/>}^\Borel = \Cons01^\Borel$.
\end{proof}

This proof technique is illustrated in \cref{fig:post-borel-O2-T0}: we partition the lattice $\Clo^\Borel_{\Cons01^{<\omega}}{2}$ into two pieces, each of which is isomorphic by \cref{thm:post-down-mod} to a part of the previously known lattice $\Clo^\Borel_{\Op2^{<\omega}}{2}$.
We now repeatedly apply this technique to the remaining clones in the top cube.

\begin{theorem}
\label{thm:post-borel-T0T1}
There are precisely 5 Borel clones on $2$ restricting to $\Cons01\Cons11^{<\omega}$ (see \cref{fig:post-borel-T0T1}):
\begin{itemize}
\item  $\Cons01\Cons11^\Borel = \Pol^\Borel \brace{0, 1} = \ang{\bigwedge, \bigvee, ?:}^\Borel$.
\item  $\Cons01\Limm111^\Borel = \Pol^\Borel \brace{0, {\lim}{=}1} = \ang{\bigvee, ?:}^\Borel$.
\item  $\Limm011\Cons11^\Borel = \Pol^\Borel \brace{{\lim}{=}0, 1} = \ang{\bigwedge, ?:}^\Borel$.
\item  $\Limm011\Limm111^\Borel = \Pol^\Borel \brace{{\lim}{=}0, {\lim}{=}1} = \ang{\floor{\bigwedge}, ?:}^\Borel = \ang{\ceil{\bigvee}, ?:}^\Borel$.
\item  $\ang{\Cons01\Cons11^{<\omega}}^\Borel = \Pol^\Borel \brace{\lim, 0, 1} = \ang{?:}^\Borel$.
\end{itemize}
\end{theorem}
\begin{proof}
By \cref{thm:post-wadge}\cref{thm:post-wadge:T0}, each clone in $\smash{\Clo^\Borel_{\Cons01\Cons11^{<\omega}}{2}}$ either is contained in $\Limm011\Cons11$ or else contains $\ang{\{\bigvee\} \cup \Cons01\Cons11^{<\omega}}^\Borel = \ang{\bigvee, ?:}^\Borel$ (and these are mutually exclusive, again since $\bigvee$ is discontinuous at $\vec{0}$).

As in the proof of \cref{thm:post-borel-T0}, by \cref{thm:post-down-mod}, the clones below $\Limm011\Cons11$ may be obtained by applying $G |-> G \cap \Limm011 = \ang{\ceil{G} \cup \{\vee\}}^\Borel$ to each of the clones in $\smash{\Clo^\Borel_{\Cons11^{<\omega}}{2}}$, which are the de~Morgan duals of the clones in \cref{thm:post-borel-T0}; these are easily seen to map to the last 3 of the 5 clones listed (see \cref{fig:post-borel-T0T1}; recall from Post's lattice \ref{fig:post} that $\ang{?:} = \ang{\vee, \ceil{->}} = \ang{\wedge, \floor{-/>}}$).

Again by \cref{thm:post-down-mod} but with $H = \ang{\wedge, \bigvee}^\Borel$, clones above $\ang{\bigvee, ?:}^\Borel$ and below $\Cons01\Cons11^\Borel$ (which are all bounded by $\bigvee$) are obtained by applying $G |-> G \cap \Cons01 = \ang{\ceil{G} \cup \{\bigvee\}}^\Borel$ to each of the clones in $\smash{\Clo^\Borel_{\Cons11^{<\omega}}{2}}$ containing $\bigvee$; this yields the first 2 listed clones.
\end{proof}

\begin{figure}
\centering
\includegraphics{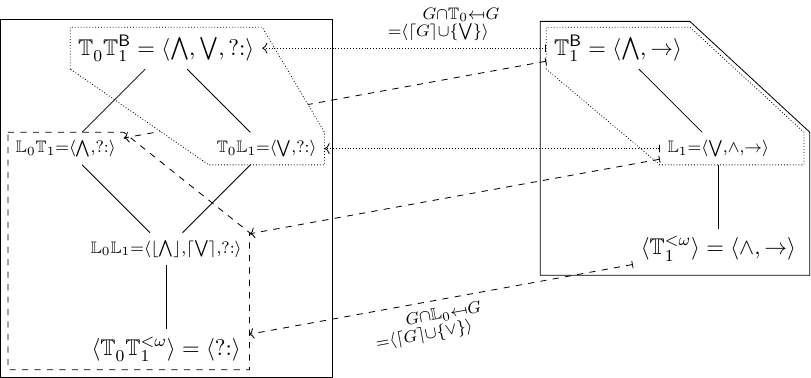}
\caption{Borel clones restricting to $\Cons01\Cons11^{<\omega}$ (left), obtained by applying modularity isomorphisms from \ref{thm:post-down-mod} to the Borel clones restricting to $\Cons11^{<\omega}$ (right), which are isomorphic to $\Clo^\Borel_{\Cons01^{<\omega}}{2}$ from \cref{fig:post-borel-O2-T0}.}
\label{fig:post-borel-T0T1}
\end{figure}

In order to apply the same technique to subclones of $\smash{\Mono^\Borel}$, we need the following standard facts.
The next lemma is well-known in topological lattice theory; see e.g., \cite[VII~1.7]{Jstone}.

\begin{lemma}
\label{thm:mono-lim}
A monotone function $f : 2^\omega -> 2$ is continuous iff it preserves both increasing limits (=joins) and decreasing limits (=meets).
In other words, $\ang{\Mono^{<\omega}}^{<\omega_1} = \Pol^{<\omega_1}\{\le, \lim\} = \Decr\Incr^{<\omega_1}$.
\end{lemma}
\begin{proof}
For a sequence of bits $x_0, x_1, \dotsc, x_\infty \in 2$, we have
\begin{align*}
\lim_{n -> \infty} x_n = x_\infty  \iff  \exists (y_n)_{n \in \#N}, (z_n)_{n \in \#N} \in 2^\#N\, \paren[\Big]{\bigwedge_n (y_n \le x_n \le z_n) \wedge \paren[\Big]{\bigveeup_n y_n = x_\infty = \bigwedgedown_n z_n}}
\end{align*}
which is a positive-primitive definition (recall \cref{thm:inv-pol-npp}) of $\lim$ from $\le, \bigveeup, \bigwedgedown$.
\end{proof}

Recall from \cref{def:post-rels} that $g : 2^n -> 2$, where $n \le \omega$, preserves $\bigveeup$ iff it is Scott-continuous, i.e., $g$ is the indicator function of a Scott-open set $g^{-1}(1) \subseteq 2^n$, meaning $g^{-1}(1)$ is upward-closed and its complement is closed under directed joins.
The following is again well-known in the theory of directed-complete posets; see e.g., \cite[VII~4.8]{Jstone}.

\begin{lemma}
\label{thm:scott-prod}
$\Incr^{<\omega_1} = \ang{\wedge, \bigvee, 0, 1}^{<\omega_1}$.
\end{lemma}
\begin{proof}
Let $n \le \omega$.
We are to show that if $U \subseteq 2^n$ is Scott-open, then its indicator function is a join of finite meets of projections, i.e., it is open in the product topology on $2^n$ where $2$ has the \emph{Sierpinski topology} $\{\emptyset, \{1\}, 2\}$.
We identify $2^n$ with the powerset $\@P(n)$.
For every $a \in U$, since $a$ is a directed union of finite subsets, there is a finite $b \subseteq a$ such that $b \in U$, whence $a \in [b,n] := \{a \in \@P(n) \mid b \subseteq a\} \subseteq U$.
Thus $U$ is a union of such $[b,n]$, each of which is a finite intersection of subbasic opens $[b,n] = \bigcap_{i \in b} \{a \in \@P(n) \mid i \in a\}$ in the product topology.
\end{proof}

\begin{corollary}
\label{thm:post-borel-M}
There are precisely 4 Borel clones on $2$ restricting to $\Mono^{<\omega}$ (see \cref{fig:post-borel-M-MT0}):
\begin{itemize}
\item  $\Mono^\Borel = \Pol^\Borel\brace{\le} = \ang{\bigwedge, \bigvee, 0, 1}^\Borel$.
\item  $\Incr^\Borel = \Pol^\Borel\brace{\bigveeup} = \ang{\wedge, \bigvee, 0, 1}^\Borel$.
\item  $\Decr^\Borel = \Pol^\Borel\brace{\bigwedgedown} = \ang{\bigwedge, \vee, 0, 1}^\Borel$.
\item  $\ang{\Mono^{<\omega}}^\Borel = \Pol^\Borel\brace{\lim, \le} = \ang{\wedge, \vee, 0, 1}^\Borel$.
\end{itemize}
\end{corollary}
\begin{proof}
$\Mono^\Borel = \ang{\bigwedge, \bigvee, 0, 1}^\Borel$ by Dyck's monotone version of the Lusin separation theorem for analytic sets; see \cite[28.12]{Kcdst}, \cite[\S5]{Dougherty}.
The generators for $\Incr^\Borel, \Decr^\Borel$ are by \cref{thm:scott-prod} and its dual; and $\ang{\Mono^{<\omega}}^\Borel = \ang{\wedge, \vee, 0, 1}^\Borel$ from Post's lattice \ref{fig:post}.
By \cref{thm:post-wadge}\cref{thm:post-wadge:M}, each clone in $\Clo^\Borel_{\Mono^{<\omega}}{2}$ either is contained in $\Decr^\Borel$ or contains $\bigvee$, hence $\ang{\wedge, \bigvee, 0, 1}^\Borel = \Incr^\Borel$; dually, it either is contained in $\Incr^\Borel$ or contains $\Decr^\Borel$.
If it contains both $\Incr^\Borel, \Decr^\Borel$, then it contains $\ang{\bigwedge, \bigvee, 0, 1}^\Borel = \Mono^\Borel$.
If it is contained in both, then it is contained in $\ang{\Mono^{<\omega}}^\Borel$ by \cref{thm:mono-lim}.
\end{proof}

\begin{figure}
\centering
\includegraphics{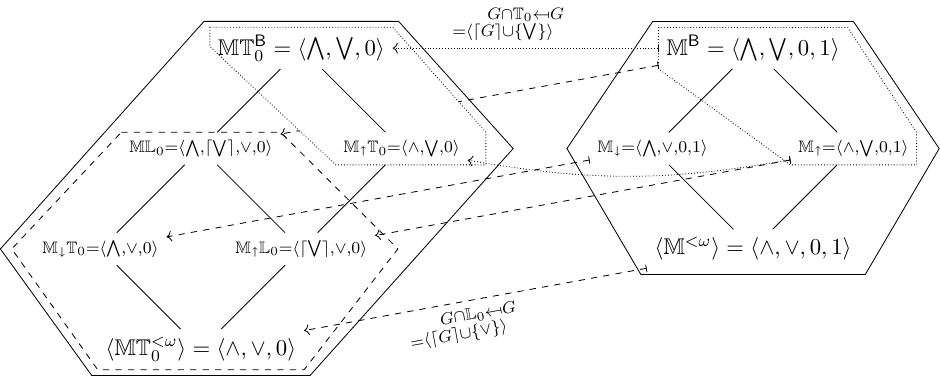}
\caption{Borel clones restricting to $\Mono\Cons01^{<\omega}$ and $\Mono^{<\omega}$, along with modularity isomorphisms from \ref{thm:post-down-mod} used to deduce the former from the latter.}
\label{fig:post-borel-M-MT0}
\end{figure}

\begin{lemma}
\label{thm:mono-meetdown-0}
$\Decr\Cons01^{<\omega_1} \subseteq \Limm011$.
\end{lemma}
\begin{proof}
This is easily seen directly, or via the dual of \cref{thm:scott-prod}, or the positive-primitive definition
\begin{equation*}
\lim_{n -> \infty} x_n = 0  \iff  \exists (y_n)_n\, \paren[\Big]{\bigwedge_n (x_n \le y_n) \wedge \paren[\Big]{\bigwedgedown_n y_n = 0}}.
\qedhere
\end{equation*}
\end{proof}

\begin{theorem}
\label{thm:post-borel-MT0}
There are precisely 6 Borel clones on $2$ restricting to $\Mono\Cons01^{<\omega}$ (see \cref{fig:post-borel-M-MT0}):
\begin{itemize}
\item  $\Mono\Cons01^\Borel = \Pol^\Borel \brace{\le, 0} = \ang{\bigwedge, \bigvee, 0}^\Borel$.
\item  $\Incr\Cons01^\Borel = \Pol^\Borel \brace{\bigveeup, 0} = \ang{\wedge, \bigvee, 0}^\Borel$.
\item  $\Mono\Limm011^\Borel = \Pol^\Borel \brace{\le, {\lim}{=}0} = \ang{\bigwedge, \ceil{\bigvee}, \vee, 0}^\Borel$.
\item  $\Incr\Limm011^\Borel = \Pol^\Borel \brace{\bigveeup, {\lim}{=}0} = \ang{\ceil{\bigvee}, \vee, 0}^\Borel$.
\item  $\Decr\Cons01^\Borel = \Pol^\Borel \brace{\bigwedgedown, 0} = \Pol^\Borel \brace{\bigwedgedown, {\lim}{=}0} = \ang{\bigwedge, \vee, 0}^\Borel$.
\item  $\ang{\Mono\Cons01^{<\omega}}^\Borel = \Pol^\Borel \brace{\le, \lim, 0} = \ang{\wedge, \vee, 0}^\Borel$.
\end{itemize}
\end{theorem}
\begin{proof}
By \cref{thm:post-wadge}\cref{thm:post-wadge:MT0}, each clone in $\Clo^\Borel_{\Mono\Cons01^{<\omega}}{2}$ either is contained in $\Mono\Limm011^\Borel$ or contains $\bigvee$.
As in the proof of \cref{thm:post-borel-T0T1}, the former (the last 4 of the 6 clones listed) are obtained by applying $G |-> G \cap \Limm011 = \ang{\ceil{G} \cup \{\vee\}}^\Borel$ to the clones in $\Clo^\Borel_\Mono{2}$ from \cref{thm:post-borel-M}, while the latter (the first 2 listed clones) are obtained by applying $G |-> G \cap \Cons01 = \ang{\ceil{G} \cup \{\bigvee\}}^\Borel$ to the clones in $\Clo^\Borel_\Mono{2}$ containing $\bigvee$ (see \cref{fig:post-borel-M-MT0}).
\end{proof}

\begin{theorem}
\label{thm:post-borel-MT0T1}
There are precisely 9 Borel clones on $2$ restricting to $\Mono\Cons01\Cons11^{<\omega}$ (see \cref{fig:post-borel-MT0T1}):
\begin{itemize}
\item  $\Mono\Cons01\Cons11^\Borel = \Pol^\Borel \brace{\le, 0, 1} = \ang{\bigwedge, \bigvee}^\Borel$.
\item  $\Mono\Cons01\Limm111^\Borel = \Pol^\Borel \brace{\le, 0, {\lim}{=}1} = \ang{\floor{\bigwedge}, \wedge, \bigvee}^\Borel$.
\item  $\Incr\Cons01\Cons11^\Borel = \Pol^\Borel \brace{\bigveeup, 0, 1} = \ang{\wedge, \bigvee}^\Borel$.
\item  $\Mono\Limm011\Cons11^\Borel = \Pol^\Borel \brace{\le, {\lim}{=}0, 1} = \ang{\bigwedge, \ceil{\bigvee}, \vee}^\Borel$.
\item  $\Decr\Cons01\Cons11^\Borel = \Pol^\Borel \brace{\bigwedgedown, 0, 1} = \ang{\bigwedge, \vee}^\Borel$.
\item  $\Mono\Limm011\Limm111^\Borel = \Pol^\Borel \brace{\le, {\lim}{=}0, {\lim}{=}1} = \ang{\floor{\bigwedge}, \ceil{\bigvee}}^\Borel$.
\item  $\Incr\Limm011\Cons11^\Borel = \Pol^\Borel \brace{\bigveeup, {\lim}{=}0, 1} = \ang{\ceil{\bigvee}, \vee}^\Borel$.
\item  $\Decr\Cons01\Limm111^\Borel = \Pol^\Borel \brace{\bigwedgedown, 0, {\lim}{=}1} = \ang{\floor{\bigwedge}, \wedge}^\Borel$.
\item  $\ang{\Mono\Cons01\Cons11^{<\omega}}^\Borel = \Pol^\Borel \brace{\le, \lim, 0, 1} = \ang{\wedge, \vee}^\Borel$.
\end{itemize}
\end{theorem}

\begin{figure}[tb]
\centering
\includegraphics{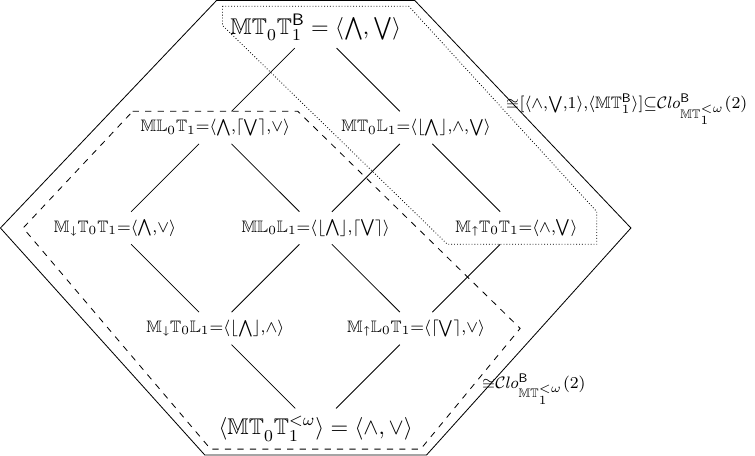}
\caption{Borel clones restricting to $\Mono\Cons01\Cons11^{<\omega}$, with the two indicated sublattices obtained by applying the modularity isomorphisms from \cref{thm:post-down-mod} to the Borel clones restricting to $\Mono\Cons11^{<\omega}$ (not shown).}
\label{fig:post-borel-MT0T1}
\end{figure}

\begin{proof}
By \cref{thm:post-wadge}\cref{thm:post-wadge:MT0}, every clone in $\Clo^\Borel_{\Mono\Cons01\Cons11^{<\omega}}{2}$ either is contained in $\Mono\Limm011\Cons11$ or contains $\bigvee$.
As in the proof of \cref{thm:post-borel-T0T1}, the former (the last 6 of the 9 clones listed) are obtained by applying $G |-> G \cap \Limm011 = \ang{\ceil{G} \cup \{\vee\}}^\Borel$ from \cref{thm:post-down-mod} to the clones in $\Clo^\Borel_{\Mono\Cons11^{<\omega}}{2}$, which are isomorphic via de~Morgan duality to $\Clo^\Borel_{\Mono\Cons01^{<\omega}}{2}$ from \cref{thm:post-borel-MT0}.
For example, the de~Morgan dual of the third clone listed in \cref{thm:post-borel-MT0} is
\begin{align*}
\Mono\Limm111^\Borel &= \ang{\floor{\bigwedge}, \wedge, \bigvee, 1}^\Borel \in \Clo^\Borel_{\Mono\Cons11^{<\omega}}{2}, \\
\shortintertext{which gets mapped to}
\Mono\Limm011\Limm111^\Borel
&= \ang{\ceil{\floor{\bigwedge}}, \ceil{\wedge}, \ceil{\bigvee}, \ceil{1}, \vee}^\Borel \\
&= \ang{\ceil{\floor{\bigwedge}}, \wedge^3, \ceil{\bigvee}, \pi_0, \vee}^\Borel \\
&= \ang{\floor{\bigwedge}, \ceil{\bigvee}}^\Borel
\end{align*}
since $\wedge^3 \in \ang{\wedge} \subseteq \ang{\ceil{\bigvee}}$ (\cref{thm:ceil-floor-gen}), $\vee \in \ang{\floor{\bigwedge}}$ similarly, and $\floor{\bigwedge} \in \Mono\Limm011\Limm111$.
The first 3 listed clones are similarly obtained by applying $G |-> G \cap \Cons01 = \ang{\ceil{G} \cup \{\bigvee\}}^\Borel$ to the 3 clones in $\Clo^\Borel_{\Mono\Cons11^{<\omega}}{2}$ containing $\bigvee$ (or alternatively, by applying de~Morgan duality within $\Clo^\Borel_{\Mono\Cons01\Cons11^{<\omega}}{2}$).
\end{proof}

This completes the classification of the Borel clones restricting to the top cube of Post's lattice, i.e., the interval $[\ang{\Mono\Cons01\Cons11^{<\omega}}^\Borel, \Op2^\Borel] \subseteq \Clo^\Borel{2}$, which is depicted altogether in \cref{fig:post-borel}.

\begin{remark}
\label{rmk:post-topcube-irred}
From this classification, we may read off that every Borel clone in $[\ang{\Mono\Cons01\Cons11^{<\omega}}^\Borel, \Op2^\Borel]$ is some intersection of the following, which form the meet-irreducible elements of $[\ang{\Mono\Cons01\Cons11^{<\omega}}^\Borel, \Op2^\Borel]$.
\begin{align*}
\ang{\Op2^{<\omega}} \gobble{= \Pol \brace{\lim}}, &&
\Cons01 \gobble{= \Pol \{0\}},
\Cons11 \gobble{= \Pol \{1\}}, &&
\Limm011 \gobble{= \Pol \brace{{\lim}{=}0}},
\Limm111 \gobble{= \Pol \brace{{\lim}{=}1}}, &&
\Mono \gobble{= \Pol \{\le\}}, &&
\Decr \gobble{= \Pol \brace{\bigwedgedown}},
\Incr \gobble{= \Pol \brace{\bigveeup}}.
\end{align*}
Dually, every Borel clone in $[\ang{\Mono\Cons01\Cons11^{<\omega}}^\Borel, \Op2^\Borel]$ is generated by some subset of the following functions, together with the generators $\wedge, \vee$ of $\Mono\Cons01\Cons11^{<\omega}$.
In other words, the clones generated by each of the following functions, together with $\wedge, \vee$, yield the join-irreducible elements of $[\ang{\Mono\Cons01\Cons11^{<\omega}}^\Borel, \Op2^\Borel]$.
\begin{align*}
0, 1, &&
\inline\bigwedge, \inline\bigvee, &&
\floor\bigwedge, \ceil\bigvee, &&
?: \; \text{(or $\ceil{->}$, or $\floor{-/>}$)}.
\end{align*}
\end{remark}

Using the classification of the Borel clones over $\Op2^{<\omega}, \Cons01^{<\omega}$, and the correspondence between self-dual clones and clones closed under $0, \beta$ given by \cref{thm:post-D-beta-mod}, we easily also obtain a classification of the Borel clones lying over 2 out of the 3 finitary self-dual, non-affine clones \cref{eq:post-D-beta-clones}:

\begin{corollary}
\label{thm:post-borel-D}
There are precisely 2 Borel clones restricting to $\Dual^{<\omega}$:
\begin{itemize}
\item  $\Dual^\Borel = \Pol^\Borel\{\neg\} = \ang{\beta(\bigvee), \neg}^\Borel = \ang{\beta(\bigwedge), \neg}^\Borel$.
\item  $\ang{\Dual^{<\omega}}^\Borel = \Pol^\Borel\{\lim,\neg\} = \ang{\exists^3_2, \neg}^\Borel$.
\end{itemize}
And there are precisely 3 Borel clones restricting to $\Dual\Cons01^{<\omega} = \Dual\Cons11^{<\omega}$:
\begin{itemize}
\item  $\Dual\Cons01^\Borel = \Dual\Cons11^\Borel = \Pol^\Borel\{\neg,0\} = \ang{\beta(\bigvee), +^3}^\Borel = \ang{\beta(\bigvee)}^\Borel$.
\item  $\Dual\Limm011^\Borel = \Dual\Limm111^\Borel = \Pol^\Borel\{\neg,{\lim}{=}0\} = \ang{\beta(\bigwedge), +^3}^\Borel$.
\item  $\ang{\Dual\Cons01^{<\omega}}^\Borel = \ang{\Dual\Cons11^{<\omega}}^\Borel = \Pol^\Borel\{\neg,\lim,0\} = \ang{\exists^3_2, +^3}^\Borel$.
\end{itemize}
(See \cref{fig:post-borel-D}.)
The functions $\beta(\bigvee), \beta(\bigwedge) : 2^\omega -> 2$ are given by
\begin{align*}
\beta(\inline\bigwedge)(x_0, x_1, \dotsc) &= \paren[\Big]{x_0 \wedge \bigvee_{i \ge 1} x_i} \vee \bigwedge_{i \ge 1} x_i, &
\beta(\inline\bigvee)(x_0, x_1, \dotsc) &= \paren[\Big]{\neg x_0 \wedge \bigvee_{i \ge 1} x_i} \vee \bigwedge_{i \ge 1} x_i.
\end{align*}
\end{corollary}
\begin{proof}
By applying \cref{thm:post-D-beta-mod} to the Borel clones restricting to $\Op2^{<\omega}$ (\cref{thm:post-borel-O2}) and $\Cons01^{<\omega}$ (\cref{thm:post-borel-T0}), which are all easily seen to be closed under $\beta$.
The generating sets above are obtained using \cref{thm:post-D-beta-gen}, by applying $\beta$ to generating sets for the respective earlier clones (namely $\{\bigvee, \neg\}, \{\bigwedge, \neg\}$, $\{\bigvee, +\}$, and $\{\bigwedge, +\}$ for the non-essentially finite clones).
\end{proof}

\begin{figure}
\centering
\includegraphics{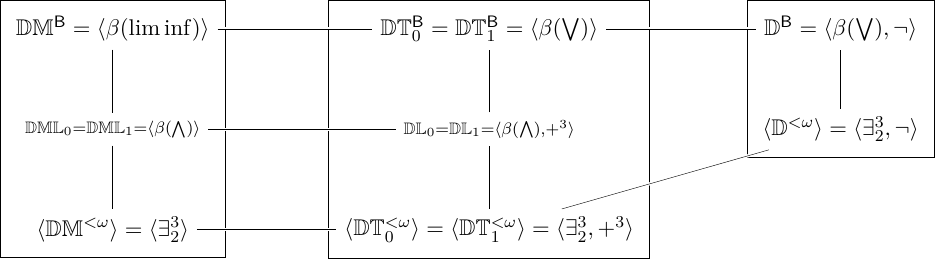}
\caption{Borel clones with finitary restrictions $\Dual\Mono^{<\omega}, \Dual\Cons01^{<\omega}, \Dual^{<\omega}$ respectively.}
\label{fig:post-borel-D}
\end{figure}

The classification of the Borel clones lying over $\Dual\Mono^{<\omega}$, which correspond via \ref{thm:post-D-beta-mod} to the Borel clones over $\Mono\Cons02^{<\omega}$, will be given in \cref{thm:post-borel-DM} after we have treated those latter clones.

\subsection{Bounded functions}
\label{sec:borel-T0inf}

We now turn our attention to the remaining regions in Post's lattice \ref{fig:post}: the ``side tubes'', between $\ang{\wedge}$ and $\Cons02$, or dually between $\ang{\vee}$ and $\Cons12$; we find it more convenient to focus on the former (left side of the diagram \ref{fig:post}).
Unlike above, here we are unable to give a full classification of the corresponding Borel clones.
Nonetheless, we can show some interesting partial structure, as well as give some indications that the remaining structure may be quite complicated (see \cref{fig:post-borel-T0k}).

In this subsection, we focus on the ``base'' of the ``side tube'', below $\Cons0{<\omega} = \Pol \brace{{\wedge^{<\omega}}{=}0}$, which consists of the indicator functions $f : 2^n -> 2$ of subsets $f^{-1}(1) \subseteq 2^n$ with the \emph{finite intersection property} (the meet of finitely many strings in $f^{-1}(1)$ is not $\vec{0}$).
Contained within these is $\Cons0\omega = \Pol \brace{{\bigwedge}{=}0}$, the indicator functions of subsets with the \emph{countable intersection property}, or equivalently functions $\le \pi_i$ for some $i$ (assuming countable arities; recall \cref{def:event}).
Clearly,
\begin{align}
\label{eq:post-T0inf-lim0}
\Cons0\omega^{<\omega_1} \subseteq \Limm011.
\end{align}
We also have the following interactions with the clones from the top cube (from \cref{rmk:post-topcube-irred}):

\begin{lemma}
\label{thm:post-T0inf-lim1}
$\Cons0{<\omega}\Limm111^{<\omega_1} \subseteq \Cons0\omega$.
\end{lemma}
\begin{proof}
We have the following positive-primitive definition of $({\bigwedge}{=}0) \subseteq 2^\omega$:
\begin{align*}
\bigwedge \vec{x} = 0  \iff  \exists \vec{y} \in 2^\omega\, \paren[\Big]{(\lim \vec{y} = 1) \wedge \bigwedge_{i < \omega} (x_0 \wedge \dotsb \wedge x_{i-1} \wedge y_i = 0)}.
&\qedhere
\end{align*}
\end{proof}

\begin{corollary}
\label{thm:post-T0inf-incr}
$\Incr\Cons0{<\omega}^{<\omega_1} \subseteq \Cons0\omega$.
\end{corollary}
\begin{proof}
By the dual of \cref{thm:mono-meetdown-0}, a Scott-continuous function $f \in \Incr$ is either the constant $0$ function, which is clearly in $\Cons0\omega$, or in $\Limm111$, hence in $\Cons0\omega$ by the above.
\end{proof}

\begin{remark}
It is not true that $\Decr\Cons0{<\omega}^{<\omega_1} \subseteq \Cons0\omega$; see \cref{ex:forall2}.
\end{remark}

We have already worked out the structure of the Borel clones in $[\ang{\Mono\Cons0{<\omega}\Cons11^{<\omega}}^\Borel, \Cons0\omega^\Borel]$, as part of the classification of the top cube in the preceding subsection:

\begin{corollary}
\label{thm:post-borel-T0omega}
There are 2, 3, 4, 6 Borel clones contained in $\Cons0\omega$ restricting to the finitary clones
$\Cons0{<\omega}^{<\omega},
\Cons0{<\omega}\Cons11^{<\omega},
\Mono\Cons0{<\omega}^{<\omega},
\Mono\Cons0{<\omega}\Cons11^{<\omega}$
respectively, namely $\ang{\ceil{-}}^\Borel$ applied to each of the Borel clones in
$\Clo^\Borel_{\Op2^{<\omega}}{2}$ (\cref{thm:post-borel-O2}),
$\Clo^\Borel_{\Cons11^{<\omega}}{2}$ (dual of \cref{thm:post-borel-T0}),
$\Clo^\Borel_{\Mono^{<\omega}}{2}$ (\cref{thm:post-borel-M}), and
$\Clo^\Borel_{\Mono\Cons11^{<\omega}}{2}$ (dual of \cref{thm:post-borel-MT0});
see \cref{fig:post-borel-T0inf}.
\end{corollary}
\begin{proof}
By \cref{thm:post-T0inf-mod}.
(We computed $\ang{\ceil{-}}^\Borel$ of these clones in proving \cref{thm:post-borel-T0,thm:post-borel-T0T1,thm:post-borel-MT0,thm:post-borel-MT0T1}; here we need only omit the last step of adding back $\vee$.)
\end{proof}

\begin{figure}[tb]
\centering
\includegraphics{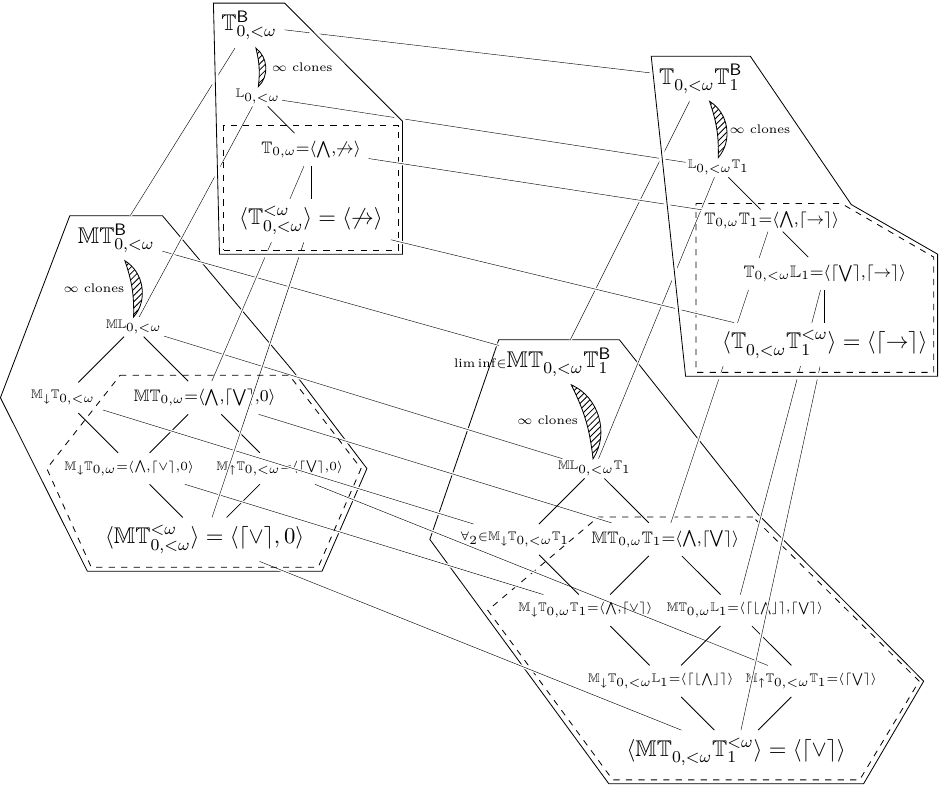}
\caption{All Borel clones restricting to $\Cons0{<\omega}^{<\omega}$, $\Cons0{<\omega}\Cons11^{<\omega}$, $\Mono\Cons0{<\omega}^{<\omega}$, or $\Mono\Cons0{<\omega}\Cons11^{<\omega}$ that are contained within $\Cons0\omega$ (dashed blocks), plus some examples of Borel clones restricting to these finitary clones that are not contained within $\Cons0\omega$ (see \cref{fig:post-borel-T0k,sec:borel-T0k} for structure of the infinite shaded regions).}
\label{fig:post-borel-T0inf}
\end{figure}

Recall that by \cref{thm:post-T0inf-mod}, every Borel clone above $\ang{\Mono\Cons0{<\omega}\Cons11^{<\omega}}^\Borel$ is sandwiched between exactly one of the above 15 clones in $[\ang{\Mono\Cons0{<\omega}\Cons11^{<\omega}}^\Borel, \Cons0\omega^\Borel]$ (the dashed blocks in \cref{fig:post-borel-T0inf}) and the corresponding clone above $\ang{\Mono\Cons11^{<\omega}}^\Borel = \ang{\wedge,\vee,1}^\Borel$ (\cref{fig:post-borel-O2-T0,fig:post-borel-T0T1,fig:post-borel-M-MT0}).
More generally, the same holds for ${<}\omega_1$-ary clones; the only difference is that there are more than just 15 clones below $\Cons0\omega^{<\omega_1}$, and we do not have a classification of all of them.
Nonetheless, we may deduce the following ``dichotomies'' for arbitrary ${<}\omega_1$-ary clones above $\ang{\Mono\Cons0{<\omega}\Cons11^{<\omega}}^{<\omega_1}$:

\begin{corollary}[of preceding subsection]
\label{thm:post-T0omega-dich}
Let $\Mono\Cons0{<\omega}\Cons11^{<\omega} \subseteq F \subseteq \Op2^{<\omega_1}$ be a ${<}\omega_1$-ary clone.
\begin{enumerate}[label=(\alph*)]
\item \label{thm:post-T0omega-dich:Meet}
Either $F \subseteq \ang{\Op2^{<\omega}}$, or $F \subseteq \Limm111$, or $F \subseteq \Incr$, or $\bigwedge \in F$.
\item \label{thm:post-T0omega-dich:cfMeet}
Either $F \subseteq \ang{\Op2^{<\omega}}$, or $F \subseteq \Incr$, or $\ceil{\floor{\bigwedge}} \in F$.
\item \label{thm:post-T0omega-dich:cJoin}
Either $F \subseteq \ang{\Op2^{<\omega}}$, or $F \subseteq \Decr$, or $\ceil{\bigvee} \in F$.
\end{enumerate}
\end{corollary}
Note that in each case, the last alternative is mutually exclusive with the others.
\begin{proof}
\Cref{thm:post-T0omega-dich:Meet} is a restatement of the dual of \cref{thm:post-wadge} (using the case there corresponding to the finitary restriction of $\ang{F \cup \{1\}}$).

\Cref{thm:post-T0omega-dich:cJoin} follows from applying \cref{thm:post-wadge} to $\ang{F \cup \{1\}}^{<\omega_1} \in [\ang{\Mono\Cons11^{<\omega}}^{<\omega_1}, \Op2^{<\omega_1}]$, which shows that if $F \not\subseteq \ang{\Op2^{<\omega}}, F \not\subseteq \Decr$, then ${\bigvee} \in \ang{F \cup \{1\}}^{<\omega_1}$, so by \cref{thm:post-down-mod}, $\ceil{\bigvee} \in \ceil{\ang{F \cup \{1\}}^{<\omega_1}} \subseteq F \cap \Cons0\omega$.

\Cref{thm:post-T0omega-dich:cfMeet}
Suppose $F \not\subseteq \ang{\Op2^{<\omega}}, \Incr$.
If $F \not\subseteq \Limm111$, then ${\bigwedge} \in F$ by \cref{thm:post-T0omega-dich:Meet}, and so clearly $\ceil{\floor{\bigwedge}} \in F$ (since $F \supseteq \Mono\Cons0{<\omega}\Cons11^{<\omega} \ni \wedge, \ceil{\vee}$).
Otherwise, we have $\ang{F \cup \{1\}}^{<\omega_1} \in [\ang{\Mono\Cons11^{<\omega}}^{<\omega_1}, \Limm111^{<\omega_1}]$.
By the dual of \cref{thm:post-down-mod} with $H = \ang{\Mono\Cons01\Mono11^{<\omega}}^{<\omega_1}$ and $\up H = \Limm111^{<\omega_1}$ as in \cref{eq:post-down-MT01T11-mod}, $\ang{F \cup \{1\}}^{<\omega_1}$ corresponds to $\ang{F \cup \{0,1\}}^{<\omega_1} \in [\ang{\Mono^{<\omega}}^{<\omega_1}, \Op2^{<\omega_1}]$, which is not contained in $\ang{\Op2^{<\omega}}$ or $\Incr$, hence contains $\bigwedge$ by the dual of \cref{thm:post-wadge}, whence $\floor{\bigwedge} \in \floor{\ang{F \cup \{0,1\}}^{<\omega_1}} \subseteq \ang{F \cup \{1\}}^{<\omega_1} \cap \Cons1\omega$.
Applying \cref{thm:post-down-mod} again yields $\ceil{\floor{\bigwedge}} \in \ceil{\ang{F \cup \{1\}}^{<\omega_1}} \subseteq F \cap \Cons0\omega$.
(This is essentially a more explicit version of the argument used in \cref{thm:post-borel-T0,thm:post-borel-MT0} to determine $[\ang{\Mono\Cons01^{<\omega}}^\Borel, \Limm011^\Borel]$.)
\end{proof}

Using the preceding dichotomies, we now simplify some of the machinery from \cref{sec:post-down}, on the global structure of the ``side tubes'' of Post's lattice, for Borel and/or ${<}\omega_1$-ary clones.

Recall \cref{def:down} of the downward-closure $\down G$ of a clone, as well as \cref{def:up} of the upward-closure $\up f$ of a function, which is the canonical candidate for a monotone function $\ge f$.

\begin{remark}
\label{rmk:post-up-analytic}
For a Borel function $f : 2^\omega -> 2$, $\up f$ need not be Borel.
Conceptually, this is because the disjunction in \cref{def:up} is over $2^\omega$, hence yields an analytic rather than Borel set.

For a concrete example, let $A \subseteq 2^\omega \times 2^\omega$ be a Borel set whose projection onto the first coordinate is not Borel (see e.g., \cite[14.2]{Kcdst}), and let $e : 2^\omega -> 2^\omega$ be a continuous embedding whose image is an antichain, e.g., $e(\vec{x}) := (x_0,\neg x_0,x_1,\neg x_1,\dotsc)$.
Then the upward-closure of $(e \times e)(A) \subseteq 2^\omega \times 2^\omega \cong 2^\omega$ is not Borel, since the first projection of $A$ is $\{\vec{x} \in 2^\omega \mid (e(\vec{x}),\vec{1}) \in (e \times e)(A)\}$.
\end{remark}

Thus for a Borel clone $\wedge \in F \subseteq \Op2^\Borel$, the canonical candidate for a monotone subclone $G \subseteq \Mono$ with $G \subseteq F \subseteq \down G$, namely $G = \up[F]$, need not be Borel (hence need not be contained in $F$).
The following shows that nonetheless, $\up[F]$ is always ``approximately'' contained in $F$:

\begin{lemma}
\label{thm:post-up-borel}
Let $\wedge \in F \subseteq \Op2^{<\omega_1}$ be a ${<}\omega_1$-ary clone such that $F \not\subseteq \Mono, \ang{\Op2^{<\omega}}$, let $f \in F^n$, and let $h : 2^n -> 2^n$ be Borel (meaning that each $h_i := \pi_i \circ h : 2^n -> 2$ is Borel) with $h \le \id$.
Then $f \circ h \in F$.
\end{lemma}
Note that the assumption $h \le \id$ implies that $f \circ h \le \up f$.
Thus, $h$ may be regarded as a Borel ``choice function'' witnessing the disjunction in the definition of $\up f$.
\begin{proof}
Since $F \not\subseteq \Pol^{<\omega_1} \{\le\}$, we have $F \cap \Op2^{<\omega} \not\subseteq \Mono^{<\omega}$ (by the right adjunction in \cref{thm:pancake}), and so (from Post's lattice \ref{fig:post}) $\ceil{->} \in \Cons01\Cons11^{<\omega} \subseteq F$; and if $F \not\subseteq \Cons11$, then $-/> \in \Cons01^{<\omega} \subseteq F$.
Also since $F \not\subseteq \Mono, \ang{\Op2^{<\omega}}$, from \cref{thm:post-T0omega-dich}, we have $\ceil{\floor{\bigwedge}} \in F$; and if $F \not\subseteq \Limm111$, then $\bigwedge \in F$.

First suppose $F \not\subseteq \Cons11$.
Since $h \le \id$, each $h_i \le \pi_i$, so $h_i \in \Cons0\omega^\Borel = \ang{\bigwedge, -/>}^\Borel \subseteq F$, and so $f \circ h \in F$.

Now suppose $F \subseteq \Cons11$.
Then either $F \subseteq \Limm111$ or not; in either case, we have $\bigwedge^{<\kappa} \in F$, and there is $J \subseteq n$ of size $\abs{J} < \kappa$ such that $f(\vec{x}) = 1$ whenever $\bigwedge_{j \in J} x_j = 1$, for $\kappa = \omega$ or $\omega_1$ respectively.
We may assume that $h(\vec{x}) = \vec{x}$ for all such $\vec{x}$, by replacing $h$ with $h(\vec{x}) \vee (\bigwedge_{j \in J} x_j \wedge \vec{x})$ which does not affect $f \circ h$.
From above, each $h_i \in \ang{\bigwedge, -/>}$, whence for any $j < n$, we have $\pi_j \vee h_i \in \ang{\floor{\bigwedge}, \floor{-/>}}$ (dual of \cref{thm:ceil-clone}), whence
\begin{align*}
h_i &= \bigwedge_{j \in J} (\pi_i \wedge (\pi_j \vee h_i))
\in \ang{\bigwedge^{<\kappa}, \ceil{\floor{\bigwedge}}, \ceil{\floor{-/>}}}^{<\omega_1}
= \ang{\bigwedge^{<\kappa}, \ceil{\floor{\bigwedge}}, \ceil{->}}^{<\omega_1}
\subseteq F.
\qedhere
\end{align*}
\end{proof}

Recall from \cref{thm:pol-inv-kcl} that the \emph{countable closure} (or \emph{${<}\omega_1$-closure}) of a clone $F \subseteq \Op2^{<\omega_1}$ consists of all functions which agree on any countably many inputs with a function in $F$, or equivalently, all functions which preserve every ${<}\omega_1$-ary relation preserved by every function in $F$.

\begin{corollary}
\label{thm:post-up-ctcl}
Let $\wedge \in F \subseteq \Op2^{<\omega_1}$ be a ${<}\omega_1$-ary clone.
Then for every $f \in F^n$, $\up f$ is in the countable closure of $F$.

Thus if $F$ is countably closed, i.e., defined as $\Pol^{<\omega_1}$ of a set of ${<}\omega_1$-ary relations, then $\up[F] = F \cap \Mono$, so $F \subseteq \down (F \cap \Mono)$, and so $F$ belongs to the domain of the embedding in \cref{thm:post-T0inf-down}.
\end{corollary}
\begin{proof}
If $f$ is monotone, then $\up f = f$.
If $f$ depends on finitely many variables, then $\up f$ depends on those same variables, and the result follows from inspecting finitary Post's lattice \ref{fig:post}.
Otherwise, in order to find $g \in F^n$ agreeing with $\up f$ on countably many $\vec{a}_i \in 2^n$, apply \cref{thm:post-up-borel} with the ``choice function'' $h : 2^n -> 2^n$ which is the identity at all $\vec{x}$ except the $\vec{a}_i \in \up f^{-1}(1)$, which it maps to any $\vec{a}_i \ge h(\vec{a}_i) \in f^{-1}(1)$.
The last statement follows from \cref{eq:up}.
\end{proof}

We also have the following, which concerns the ``left-versus-right'' direction ($\Cons11$ or not) of the ``side tubes'' of Post's lattice \ref{fig:post}, rather than the ``front-versus-back'' direction ($\Mono$ or not):

\begin{proposition}
\label{thm:post-T1-mod}
We have isomorphisms
\begin{align*}
[\ang{\ceil{\vee}}^{<\omega_1}, \Limm111^{<\omega_1}] &\cong \set{G \in [\ang{\ceil{\vee},0}^{<\omega_1}, \Op2^{<\omega_1}]}{G \subseteq \down (G \cap \Limm111)} \\
F &|-> \ang{F \cup \{0\}}^{<\omega_1} = \{f_0 \mid f \in F^{\ge2}\} \\
G \cap \Limm111 &<-| G,
\\[1ex]
[\ang{\bigwedge}^{<\omega_1}, \Cons11^{<\omega_1}] &\cong [\ang{\bigwedge,0}^{<\omega_1}, \Op2^{<\omega_1}] \\
F &|-> \ang{F \cup \{0\}}^{<\omega_1} \\
G \cap \Cons11 &<-| G.
\end{align*}
In each case, for every clone $F$ on the left, we have $\ang{F \cup \{0\}}^{<\omega_1} \subseteq \down F$.
\end{proposition}
\begin{proof}
To show that these maps compose to the identity on the left, in each case, let $\bigwedge^{<\kappa} \in F \subseteq \Cons11^{<\omega_1}$ be a clone such that each $f \in F$ is ${<}\kappa$-continuous at $\vec{1}$, for $\kappa = \omega, \omega_1$ respectively.
So for $f \in F^{1+n}$, there is $J \subseteq n$ of size $<\kappa$ such that $f(1, \vec{x}) = 1$ whenever $\bigwedge_{j \in J} x_j = 1$.
Thus
\begin{align*}
f_0(\vec{x}) = f(0, \vec{x}) \le f\paren[\Big]{\bigwedge_{j \in J} x_j, \vec{x}}.
\end{align*}
This shows that $\ang{F \cup \{0\}}^{<\omega_1} \subseteq \down F$.
If $f_0$ is also ${<}\kappa$-continuous at $\vec{1}$, then we may pick $J$ so that also $f(0, \vec{x}) = 1$ whenever $\bigwedge_{j \in J} x_j = 1$, whence the above $\le$ becomes $=$, showing that $f_0 \in F$.
Thus $\ang{F \cup \{0\}}^{<\omega_1} \cap \Limm111$, respectively $\ang{F \cup \{0\}}^{<\omega_1} \cap \Cons11$, is contained in $F$; the other inclusion is trivial.

Now let $\ceil{\vee}, 0 \in G \subseteq \Op2^{<\omega_1}$ be a clone such that $G \subseteq \down F$ for $F := G \cap \Limm111$; we must show $G \subseteq \ang{F \cup \{0\}}$.
Let $g \in G^n$, and let $g \le f \in F^n = G^n \cap \Limm111$.
Then $h : 2^{1+n} -> 2$ with
\begin{align*}
h(y,\vec{x}) := f(\vec{x}) \wedge (y \vee g(\vec{x})) = \ceil{\vee}(f(\vec{x}), y, g(\vec{x}))
\end{align*}
has $h_0 = g$ and $h_1 = f$, thus $h \in \Limm111$, and so $h \in G \cap \Limm111 = F$ with $g = h_0 \in \ang{F \cup \{0\}}$.

Finally, let $\bigwedge, 0 \in G \subseteq \Op2^{<\omega_1}$ be a clone, and let $F := G \cap \Cons11$; we must show $G \subseteq \ang{F \cup \{0\}}$.
If $G \subseteq \Mono$, then clearly the only monotone function not in $\Cons11$ is $0$, so $F = G \setminus \{0\}$ and $G = F \cup \{0\}$.
If $G \subseteq \ang{\Op2^{<\omega}}$, the result follows by inspection of Post's lattice \cref{fig:post}.
Otherwise, for every $g \in G \setminus \Cons11 \setminus \{0\}$, we have $f := g \vee \bigwedge \in G \cap \Cons11 = F$ by \cref{thm:post-up-borel} with $h : 2^n -> 2^n$ which is the identity at all $\vec{x}$ except for $\vec{1}$ which gets mapped to any element of $g^{-1}(1)$.
Since $G \not\subseteq \Mono$, from Post's lattice \ref{fig:post} we also have $\ceil{->} \in G \cap \Cons11 = F$, so $g = f -/> \bigwedge \in \ang{f, -/>, \bigwedge} \subseteq \ang{f, \ceil{->}, 0, \bigwedge} \subseteq \ang{F \cup \{0\}}$.
\end{proof}

Similarly to \cref{rmk:post-up-analytic}, note that for a Borel clone $G \subseteq \Op2^\Borel$, $\down G$ may contain non-Borel functions (e.g., if $1 \in G$, then $\down G$ contains all functions).
Following \cref{def:clone-borel}, we adopt

\begin{notation}
\label{def:borel-down}
For a Borel clone $G \subseteq \Op2^\Borel$, its \defn{Borel downward-closure} is
\begin{align*}
\down^\Borel G := \down G \cap \Op2^\Borel.
\end{align*}
\end{notation}

\begin{proposition}
\label{thm:borel-down-clone}
For every Borel clone $\wedge \in G \subseteq \Op2^\Borel$, $\down^\Borel G$ is a clone.
\end{proposition}
\begin{proof}
If $G \subseteq \Mono$, this follows from \cref{thm:down-clone}; if $G \subseteq \ang{\Op2^{<\omega}}$, this follows from inspecting finitary Post's lattice \ref{fig:post}.
Suppose neither of these holds.
Then from Post's lattice \ref{fig:post}, $\Mono\Cons0{<\omega}\Cons11^{<\omega} \subseteq G$,
whence by \cref{thm:post-T0omega-dich},
$\ang{G \cup \{0\}}^\Borel
= \ang{G \cup \{0, \bigwedge\}}^\Borel
= \ang{G \cup \Cons0\omega^\Borel}^\Borel$ (since $\bigwedge \in \ang{0, \ceil{\floor{\bigwedge}}}$), which by \cref{thm:ceil-down-clone} with $F = \Op2^\Borel$ is equal to $\ang{\down^\Borel G}^\Borel$.
But by \cref{thm:post-const-xsec}, every function in $\ang{G \cup \{0\}}^\Borel$ is $g_0$ for some $g \in G$, which by \cref{thm:post-T1-mod} is in $\down G$.
So $\ang{\down^\Borel G}^\Borel \subseteq \down^\Borel G$.
\end{proof}

\begin{corollary}[cf.\ \cref{thm:post-T0inf-down}]
\label{thm:post-borel-T0inf-down}
We have an order-embedding with right adjoint retraction
\begin{equation*}
\begin{tikzcd}[column sep=1em]
&& \scriptstyle(\ang{F \cup \{1\}}^\Borel, \down^\Borel F)
    \dar[phantom, "\scriptstyle\in"{sloped}]
\\[-1em]
\scriptstyle(F \cap \Cons0\omega, \down^\Borel F)
    \rar[phantom, "\scriptstyle\in"]
    \ar[urr, mapsto, bend left=5] &
{} [\ang{\wedge}^\Borel, \Cons0\omega^\Borel] \times [\Cons0\omega^\Borel, \Op2^\Borel] \rar[phantom, "\cong"]
    \dar[shift left=2, two heads, right adjoint'] &
|[label={[rotate=-90,anchor=west,inner sep=0pt,label={[name=GH]right:\scriptstyle(G,H)}]below:\scriptstyle\ni}]|
{} [\ang{\wedge, 1}^\Borel, \Op2^\Borel] \times [\Cons0\omega^\Borel, \Op2^\Borel]
\\[-1em]
\scriptstyle F \rar[phantom, "\scriptstyle\in"] \uar[mapsto] &
{} [\ang{\wedge}^\Borel, \Op2^\Borel]
    \uar[shift left=2, hook] &
|[xshift=-5em]|
\scriptstyle G \cap H
    \lar[phantom, "\scriptstyle\ni"]
    \ar[to=GH, mapsfrom]
\end{tikzcd}
\end{equation*}
\end{corollary}
\begin{proof}
By \cref{thm:post-lower}, $\ang{F \cup \{1\}}^\Borel \cap \down^\Borel F = \ang{F \cup \{1\}}^\Borel \cap \down F = F$.
\end{proof}

Thus, \cref{fig:post-mod-down} describes the entire left side (above $\ang{\wedge}^\Borel$) of the Borel version of Post's lattice.
We may also drop the monotonicity assumption from \cref{thm:post-down-emb,thm:post-down-mod}:

\begin{corollary}
\label{thm:post-borel-down-emb}
For any Borel clones $\wedge \in G \subseteq \Cons0\omega^\Borel$ and $G \subseteq H \subseteq \Cons0\omega^\Borel$, the modularity adjunction
\begin{align*}
[G, \ang{G \cup \{1\}}^\Borel] &\rightleftarrows [H, \ang{H \cup \{1\}}^\Borel]
\end{align*}
exhibits the left interval as a retract of the right.
\qed
\end{corollary}

\begin{corollary}
\label{thm:post-borel-down-mod}
For any Borel clone $\wedge \in H \subseteq \Op2^\Borel$, we have modularity isomorphisms
\begin{equation*}
\begin{alignedat}[b]{2}
[H \cap \Cons0\omega, \Cons0\omega^\Borel]
&\cong [H, \down H]
&&\cong [\ang{H \cup \{1\}}^\Borel, \Op2^\Borel]. %\\
%F &|-> \ang{F \cup H}^\Borel &&|-> \ang{F \cup H \cup \{1\}}^\Borel = \ang{F \cup \{1\}}^\Borel \\
%G \cap \Cons0\omega = \ang{\ceil{G}}^\Borel &<-| && \mathllap{G \cap \down H} <-| G.
\end{alignedat}
\qed
\end{equation*}
\end{corollary}

\Cref{thm:post-up-ctcl} continues to apply to ${<}\omega_1$-ary clones, rather than Borel clones; there does not seem to be a direct Borel analogue, since (as noted in \cref{rmk:post-up-analytic}) even for a Borel clone $F$, $\up[F]$ need not be Borel.
(Note that the last conclusion of \cref{thm:post-up-ctcl} is subsumed in the Borel setting by \cref{thm:borel-down-clone}.)
Nonetheless, we do have the following Borel variant of \cref{thm:post-up-ctcl}:

\begin{proposition}
\label{thm:post-borel-up}
Let $\wedge \in F \subseteq \Op2^\Borel$ be a Borel clone.
Suppose $F = \Pol^\Borel(\@M)$ for a countable set $\@M$ of Borel downward-closed relations $R \subseteq 2^k$, $k \le \omega$.
Then $\up[F] \subseteq \down(F \cap \Mono)$, and so $F \subseteq \down(F \cap \Mono)$.
\end{proposition}
\begin{proof}
For each $f \in F^n$, we have $\up f \in \Pol(\@M) \cap \Mono$ by \cref{thm:post-up-ctcl}.
Apply now the reflection theorem \cite[35.10]{Kcdst} to the analytic set $\up f^{-1}(1) \subseteq 2^n$, to get a Borel superset $g^{-1}(1) \supseteq \up f^{-1}(1)$ whose indicator function $g$ is in $\Pol^n(\@M)$.
To do so, we need to verify the hypothesis that the class of sets $g^{-1}(1) \subseteq 2^n$ whose indicator function $g$ is in $\Pol(\@M)$ is $\*\Pi^1_1$ on $\*\Sigma^1_1$,
which means (see again \cite[35.10]{Kcdst}), in terms of the indicator functions, that
for any family of functions $(g_\alpha : 2^n -> 2)_{\alpha \in 2^\omega}$ which is the indicator function in two variables of an analytic set $\{(\alpha,\vec{x}) \mid g_\alpha(\vec{x}) = 1\} \subseteq 2^\omega \times 2^n$, the set of indices $\alpha$ such that $g_\alpha \in \Pol(\@M)$ is $\*\Pi^1_1$.
Indeed,
\begin{align*}
g_\alpha \in \Pol(\@M)
\iff{}&
\forall R \in \@M^k\, \forall (x_{i,j})_{i<k,j<n} \in 2^{k \times n}\,
\paren[\Big]{\forall j < n\, \paren[\big]{(x_{i,j})_i \in R} \implies \paren[\big]{g_\alpha((x_{i,j})_j)}_i \in R};
\end{align*}
and by downward-closure of $R$,
\begin{align*}
\paren[\big]{g_\alpha((x_{i,j})_j)}_i \in R
\iff{}&
\forall \vec{y} \in 2^k\, \paren[\Big]{\vec{y} \le \paren[\big]{g_\alpha((x_{i,j})_j)}_i \implies \vec{y} \in R}
\\
\iff{}&
\forall \vec{y} \in 2^k\, \paren[\Big]{\forall i < k\, \paren[\big]{y_i = 1 \implies g_\alpha((x_{i,j})_j) = 1} \implies \vec{y} \in R}
\end{align*}
which are clearly $\*\Pi^1_1$.
So we have found Borel $\up f \le g \in \Pol(\@M)$.
By Dyck's monotone separation theorem \cite[28.12]{Kcdst}, there is a monotone Borel $\up f \le h \le g$.
Since each $R \in \@M$ is downward-closed, we still have $h \in \Pol(\@M)$, so $h \in F \cap \Mono$ with $\up f \le h$, as desired.
\end{proof}

\begin{example}
Each $\Cons0k^\Borel = \Pol^\Borel \{{\bigwedge^k}{=}0\}$ is defined by a countable-arity downward-closed Borel relation $\paren{{\bigwedge^k}{=}0} \subseteq 2^k$.
Thus, the above shows that every $\Cons0k$ Borel function is $\le$ a monotone such Borel function.
(See \cref{sec:borel-T0k} for similar but more involved examples.)
\end{example}

\begin{remark}
The above proof applies more generally if $\@M$ is a possibly uncountable family $\{R_i\}_{i \in I}$ of relations, indexed over some standard Borel space $I$, such that $\{(i,\vec{x}) \mid \vec{x} \in R_i\}$ is $\*\Pi^1_1$.
\end{remark}

We close this subsection with some basic remarks on the Borel clones below $\Cons0{<\omega}$ which are \emph{not} contained in $\Cons0\omega$; we will study these clones in more detail in \cref{sec:borel-T0k} (producing in particular the countably infinitely many examples in the crescent-shaped shaded regions in \cref{fig:post-borel-T0inf}).

\begin{example}
$\liminf : 2^\omega -> 2$ is in $\Mono\Cons0{<\omega}\Cons11^\Borel$, indeed in $\Meet\Cons01\Cons11^\Borel$ (it is the indicator function of the Fréchet filter of cofinite sets).
But it is not in $\Cons0\omega$, or even \cref{eq:post-T0inf-lim0} in $\Limm011$, since the strings of the form $0\dotsm 0 111\dotsm \in \liminf^{-1}(1)$ converge to $\vec{0}$.
\end{example}

\begin{example}
\label{ex:forall2}
The quantifier $\forall_2 : 2^\omega -> 2$, ``for all but at most one input bits'', namely
\begin{align*}
\forall_2(\vec{x}) = \bigwedge_{i < j < \omega} (x_i \vee x_j) = \bigvee_{i < \omega} \bigwedge_{j \ne i} x_j
\end{align*}
(recall \cref{def:post-funs}), is also in $\Mono\Cons0{<\omega}\Cons11^\Borel \setminus \Cons0\omega$.
Unlike $\liminf$, it is in $\ang{\bigwedge, \vee}^\Borel = \Decr\Cons01\Cons11^\Borel \subseteq \Limm011^\Borel$.
\end{example}

In fact, $\forall_2$ is a ``minimal'' example of a non-$\Cons0\omega$ function:

\begin{lemma}
\label{thm:post-forall2}
If $f : 2^\omega -> 2 \notin \Cons0\omega$,
then $\forall_2 \in \ang{\{f\} \cup \Decr\Cons0\omega\Cons11^\Borel} = \ang{f, \bigwedge, \ceil{\vee}}$.
\end{lemma}
\begin{proof}
Since $f \notin \Cons0\omega$, there are $\vec{x}_0, \vec{x}_1, \dotsc \in f^{-1}(1)$ with $x_{i,i} = \pi_i(\vec{x}_i) = 0$.
Define
\begin{align*}
g_i : 2^\omega -->{}& 2^\omega \\
\vec{y} |-->{}& (y_i \vee \vec{x}_i) \wedge \vec{y} \\
={}& \begin{cases}
\vec{x}_i \wedge \vec{y} &\text{if $y_i = 0$}, \\
\vec{y} &\text{if $y_i = 1$}.
\end{cases}
\end{align*}
Then each coordinate $g_{i,j} := \pi_j \circ g_i$ of each $g_i$ is either $\pi_j$ (if $x_{i,j} = 1$) or $\pi_i \wedge \pi_j$ (if $x_{i,j} = 0$), so $g_{i,j} \in \ang{\wedge}$.
Thus
\begin{equation*}
\begin{aligned}
\inline{\bigwedge_i g_i} : 2^\omega &--> 2^\omega \\
0111\dotsb &|--> \vec{x}_0 \\
1011\dotsb &|--> \vec{x}_1 \\
1101\dotsb &|--> \vec{x}_2 \\
&\vdotswithin{|-->}
\end{aligned}
\end{equation*}
has each coordinate in $\ang{\bigwedge}$, and so we have $f \circ {\bigwedge_i g_i} \in \ang{f, \bigwedge}$, mapping each string with exactly one $0$ in the $i$th coordinate to $f(\vec{x}_i) = 1$.
By \cref{thm:post-up-ctcl}, we then in fact have some $\forall_2 \le h \in \ang{f, \bigwedge}$.
Since $\forall_2 \in \Decr\Cons0{<\omega}\Cons11^\Borel \subseteq \Decr\Cons11^\Borel = \ang{\Decr\Cons0\omega\Cons11^\Borel \cup \{1\}}^{<\omega_1} \subseteq \ang{\{f\} \cup \Decr\Cons0\omega\Cons11^\Borel \cup \{1\}}^{<\omega_1}$, by \cref{thm:post-lower} it follows that $\forall_2 \in \ang{\{f\} \cup \Decr\Cons0\omega\Cons11^\Borel}^{<\omega_1}$.
\end{proof}

\begin{corollary}
\label{thm:post-T0inf-forall2}
Let $\Mono\Cons0{<\omega}\Cons11^{<\omega} \subseteq F \subseteq \Cons0{<\omega}^{<\omega_1}$ be a ${<}\omega_1$-ary clone.
Then $F \subseteq \Cons0\omega$ or $\forall_2 \in F$.
\end{corollary}
Note that as in \cref{thm:post-T0omega-dich}, Borelness of $F$ is not required.
\begin{proof}
If $F \not\subseteq \Cons0\omega$, then by \cref{thm:post-T0inf-lim1,thm:post-T0inf-incr}, $F \not\subseteq \Limm111, \Incr$, and clearly also $F \not\subseteq \ang{\Op2^{<\omega}}$ (since $F \not\subseteq \Cons0\omega$), whence by \cref{thm:post-T0omega-dich}, $\bigwedge \in F$; now apply the preceding lemma.
\end{proof}

\begin{remark}
\label{rmk:forall2}
We do not know whether $\Mono\Cons0{<\omega}\Cons11^\Borel = \ang{\forall_2}^\Borel$, but suspect this to be the case.
For instance, we have $\forall_3 \in \ang{\forall_2}$ (recall again \cref{def:post-funs}), since it is easily seen that
\begin{align*}
\forall_3(\vec{x}) \le \forall_2(
    \forall_2(x_{i_{0,0}},x_{i_{0,1}},\dotsc),
    \forall_2(x_{i_{1,0}},x_{i_{1,1}},\dotsc),
    \dotsc),
\end{align*}
where $(i_{j,k})_{j,k \in \omega}$ is any bijection $\omega \times \omega \cong \omega$; now apply \cref{thm:post-lower}.
More generally, $\ang{\forall_2} = \ang{\forall_k}$ for any $2 \le k < \omega$.

Note that if $\Mono\Cons0{<\omega}\Cons11^\Borel = \ang{\forall_2}^\Borel$, then by \cref{thm:post-borel-down-mod}, we would have a complete classification of all Borel clones ``just above'' $\Cons0\omega$ in \cref{fig:post-borel-T0inf}, i.e., between $\Decr\Cons0{<\omega}\Cons11$ and $\Limm0{<\omega}{}$ which is the downward-closure $\down \Decr\Cons0{<\omega}\Cons11$ (see \cref{rmk:post-borel-L0inf}).
\end{remark}

\subsection{Filters and ideals}

Next, we consider the Borel clones below $\Meet = \Pol\{\wedge\}$, which consists of indicator functions $f : 2^n -> 2$ of filters $f^{-1}(1) \subseteq 2^n$ or the empty set; the latter is excluded in $\Meet\Cons11 = \Pol\{\wedge, 1\}$, while the improper filter $2^n$ is excluded in $\Meet\Cons01$.
Similarly, $\Meet_\omega = \Pol\{\bigwedge\}$ consists of indicator functions of $\sigma$-filters.
But since $2^\omega$ is countably generated under meets, a $\sigma$-filter in it is just the principal filter above some element $\vec{x}$, namely $\{\vec{y} \in 2^\omega \mid \vec{x} \le \vec{y}\}$, whose indicator function is $\smash{\bigwedge_{\substack{i < \omega \\ x_i = 1}} \pi_i}$.
This yields
\begin{align*}
\yesnumber
\label{eq:meet-gen}
\Meet_\omega^{<\omega_1} &= \ang{\bigwedge, 0, 1}^{<\omega_1}, &
\Meet_\omega \Cons11^{<\omega_1} &= \ang{\bigwedge, 1}^{<\omega_1}, \\
\Meet_\omega \Cons01^{<\omega_1} &= \ang{\bigwedge, 0}^{<\omega_1}, &
\Meet_\omega \Cons01\Cons11^{<\omega_1} &= \ang{\bigwedge}^{<\omega_1}.
\end{align*}
We also clearly have
\begin{align}
\label{eq:meet-decr}
\Meet \Decr = \Pol\brace{\wedge, \bigwedgedown} &= \Pol\brace{\bigwedge} = \Meet_\omega, \\
\Meet_\omega \Cons01 = \Pol\brace{{\bigwedge},0} &\subseteq \Pol\brace{{\bigwedge}{=}0} = \Cons0\omega.
\end{align}

\begin{lemma}
\label{thm:post-meet-lim0}
$\Meet\Limm011^{<\omega_1} = \Meet\Cons0\omega^{<\omega_1}$.
\end{lemma}
\begin{proof}
$\supseteq$ by \cref{eq:post-T0inf-lim0}; conversely, if $f : 2^\omega -> 2$ is the indicator function of a filter $f^{-1}(1)$ but $f \not\le \pi_i$ for each $i < \omega$, then each of the strings $\vec{x}_i := 1\dotsm101\dotsm$ with a single $0$ in the $i$th bit would be in $f^{-1}(1)$, whence $\vec{x}_0, \vec{x}_0 \wedge \vec{x}_1, \dotsb \in f^{-1}(1)$ would converge to $\vec{0}$, whence $f \notin \Limm011$.
\end{proof}

\begin{remark}
\label{thm:post-meet-lim1}
$\Meet\Limm111^{<\omega_1} = \ang{\Meet\Cons11^{<\omega}}^{<\omega_1}$.
\end{remark}
\begin{proof}
Dually, an ideal $\subseteq 2^\omega$ with $\vec{0}$ in the interior is a clopen subgroup mod 2.
\end{proof}

From these observations, we may compute all Borel clones that are the intersection of $\Meet$ with a clone from the top cube (\cref{rmk:post-topcube-irred}); see \cref{fig:post-borel-Meet}.

\begin{figure}[tb]
\centering
\includegraphics{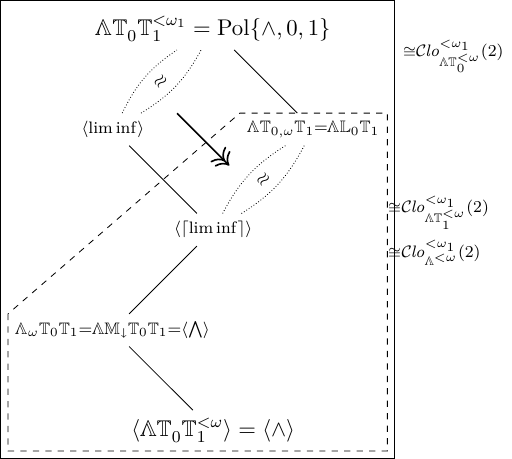}
\caption{All countable arity clones restricting to $\Meet\Cons01\Cons11^{<\omega}$.
Aside from the 6 labelled, all other clones must lie in one of the two `$\approx$' regions, inside of which no clones can be distinguished by countable arity relations.
The clones restricting to $\Meet\Cons01^{<\omega}, \Meet\Cons11^{<\omega}, \Meet^{<\omega}$ are isomorphic to the blocks shown.}
\label{fig:post-borel-Meet}
\end{figure}

We now show that these are ``approximately'' all of the Borel clones, indeed all of the ${<}\omega_1$-ary clones, with finitary restriction $\Meet^{<\omega}$ or one of its relatives.
To do so, we need ``dichotomies'' along the lines of \cref{thm:post-T0omega-dich,thm:post-forall2}.
Note however that the proof technique used there does not apply, since it depended on a prior understanding of $[\ang{\Mono\Cons0{<\omega}\Cons11}, \Cons0\omega]$, whereas we lack a similar understanding of $\Meet\Cons0\omega$.
Instead, we use the following, which is implicit in \cite{Kahane}:

\begin{lemma}[Kahane]
\label{thm:kahane}
If $f : 2^\omega -> 2 \in \Mono\Cons02 \setminus \Limm011$,
then $\liminf \in \ang{f, \wedge}$.
\end{lemma}
(Monotonicity of $f$ is automatic if $f \in \Meet$, but is in fact not required otherwise; see \cref{thm:fkt-gen}.)
\begin{proof}
As in \cref{thm:post-wadge}\cref{thm:post-wadge:MT0}, there are $\vec{x}_0 \ge \vec{x}_1 \ge \dotsb$ in $f^{-1}(1)$ converging to $\vec{0} \in f^{-1}(0)$.
By monotonicity, it follows that $f(\vec{y}) = 1$ whenever $\vec{y}$ has cofinally many $1$'s, i.e.,
\begin{align*}
\liminf \le f
&&\text{which implies by de~Morgan duality}&&
\delta(f) = \neg f \neg \le \limsup.
\end{align*}
But since $f \in \Cons02$, we have $f \wedge f\neg = 0$, i.e., $f \le \neg f \neg$; thus
\begin{align*}
f(\vec{y}) &= \lim \vec{y}
    \quad \text{whenever $\vec{y} \in 2^\omega$ converges}. \\
\shortintertext{Hence}
\liminf \vec{y}
&= \lim_{n -> \infty} \inf_{i \ge n} y_i
= \lim_{n -> \infty} \lim_{m -> \infty} (y_i \wedge y_{i+1} \wedge \dotsb \wedge y_m) \\
&= f\paren*{\begin{aligned}
    &f(y_0, y_0 \wedge y_1, y_0 \wedge y_1 \wedge y_2, \dotsc), \\
    &f(y_1, y_1 \wedge y_2, y_1 \wedge y_2 \wedge y_3, \dotsc), \\
    &f(y_2, y_2 \wedge y_3, y_2 \wedge y_3 \wedge y_4, \dotsc), \\
    &\dotsc
\end{aligned}}.
\qedhere
\end{align*}
\end{proof}

\begin{corollary}
\label{thm:kahane-ceil}
If $f : 2^\omega -> 2 \in \Meet \setminus \Decr$,
then $\ceil{\liminf} \in \ang{f, \wedge}$.
\end{corollary}
\begin{proof}
Since $f \notin \Decr$, it is clearly not a constant function; thus $f \in \Meet\Cons01\Cons11$.
If $f \notin \Limm011$, then \cref{thm:kahane} applies; so suppose $f \in \Meet\Limm011^{<\omega_1} = \Meet\Cons0\omega^{<\omega_1}$ (\cref{thm:post-meet-lim0}).
Let $\vec{x} := \bigwedge f^{-1}(1)$ and
\begin{equation*}
g(\vec{y}) := f(\vec{x} \vee \vec{y}) = f\paren[\big]{(x_i ? 1 : y_i)_{i < \omega}};
\end{equation*}
then $g \in \ang{f, 1} \subseteq \Meet$.
Since $f$ does not preserve $\bigwedgedown$, and $f^{-1}(1) \subseteq 2^\omega$ is a filter, there is a decreasing sequence in $f^{-1}(1)$ converging to $\vec{x} \in f^{-1}(0)$; the meet of each term in this sequence with $\neg \vec{x}$ is then a sequence in $g^{-1}(1)$ converging to $\vec{0} \in g^{-1}(0)$, and so $g \in \Meet\Cons01 \setminus \Limm011$.
By \cref{thm:kahane}, $\liminf \in \ang{g, \wedge}$, whence by \cref{thm:post-T0inf-mod}, $\ang{f, \wedge} = \ang{\ceil{\ang{f, \wedge, 1}}} \supseteq \ang{\ceil{\ang{g, \wedge}}} \ni \ceil{\liminf}$.
\end{proof}

\begin{corollary}
\label{thm:post-meet-cts}
If $f : 2^\omega -> 2 \in \Meet$ is discontinuous,
then $\bigwedge \in \ang{f, \wedge}$.
\end{corollary}
\begin{proof}
If $f \notin \Decr$, then by the preceding corollary, we get $\ceil{\liminf} \in \ang{f, \wedge}$, from which we get $\bigwedge \vec{x} = \ceil{\liminf}(x_0; x_1; x_1, x_2; x_1, x_2, x_3; x_1, x_2, x_3, x_4; \dotsc)$.
Otherwise, $f \in \Meet\Decr^{<\omega_1} = \Meet_\omega^{<\omega_1} = \ang{\bigwedge, 0, 1}^{\omega_1}$ (\ref{eq:meet-decr} and \ref{eq:meet-gen}); since $f$ is discontinuous, it must be a meet of infinitely many variables, from which we easily get $\bigwedge$ via a variable substitution.
\end{proof}

These results show that the ${<}\omega_1$-ary clones $\Clo^{<\omega_1}_{\Meet\Cons01\Cons11^{<\omega}}{2}$ restricting to $\Meet\Cons01\Cons11^{<\omega}$, say, are as depicted in \cref{fig:post-borel-Meet}, with all other potential clones in one of the intervals $[\liminf, \Meet\Cons01\Cons11]$ or $[\ceil\liminf, \Meet\Cons0\omega\Cons11]$.
Within these intervals, we have a mismatch between the minimal functions provided by \cref{thm:kahane,thm:kahane-ceil} and the invariant relations: for example, not every Borel filter $f \in \Meet\Cons01\Cons11^\Borel$ can be constructed from the Fréchet filter $\liminf$.
(See \cref{ex:summable}.)
We point out however that this is true ``approximately''.
Recall again from \cref{thm:pol-inv-kcl} that the \emph{countable closure} of a clone consists of all functions agreeing on any countably many input strings.

\begin{proposition}
\label{thm:post-ctcl-liminf}
For any finitary clone $\wedge \in F \subseteq \Op2^{<\omega}$ in Post's lattice, the infinitary clone $\-{\ang{F}}$ (of arbitrary arity) is the countable closure of $\ang{F \cup \{\liminf\}}$.
\end{proposition}
(Here, as usual, $\liminf$ means $\liminf^\omega : 2^\omega -> 2$.)
\begin{proof}
Since $\wedge \in F$, we have $\bigwedge \in \-{\ang{F}}$, and so $\liminf \in \-{\ang{F}}$; this shows $\supseteq$.
To show $\subseteq$: let $g \in \-{\ang{F}}^n$ and $\vec{x}_0, \vec{x}_1, \dotsc \in 2^n$; we must find $f \in \ang{F \cup \{\liminf\}}^n$ agreeing with $g$ on each $\vec{x}_i$.
Since $g \in \-{\ang{F}}$, we may find $f_0, f_1, \dotsc \in \ang{F}^n$ such that $f_i$ agrees with $g$ on $\vec{x}_0, \dotsc, \vec{x}_{i-1}$.
Let $f := \liminf_i f_i$.
\end{proof}

\begin{corollary}
\label{thm:post-T0T1Meet-ctcl}
$\Meet\Cons01\Cons11^\omega \subseteq 2^{2^\omega}$ is the countable closure of $\ang{\liminf}^\omega$.
\qed
\end{corollary}

\begin{corollary}
$\Meet\Cons01^\omega, \Meet\Cons11^\omega, \Meet^\omega$ are the countable closures of $\ang{\liminf, 0}^\omega, \ang{\liminf, 1}^\omega, \ang{\liminf, 0, 1}^\omega$.
\end{corollary}
\begin{proof}
A function $f \in \Mono$ is either constant or preserves both $0, 1$.
\end{proof}

\begin{corollary}
$\Meet\Cons0\omega\Cons11^\omega, \Meet\Cons0\omega^\omega$ are the countable closures of $\ang{\ceil{\liminf}}^\omega, \ang{\ceil{\liminf}, 0}^\omega$.
\end{corollary}
\begin{proof}
This may be seen using the modularity isomorphisms (similarly to \cref{thm:kahane-ceil}), or by slightly modifying the proof of \cref{thm:post-ctcl-liminf}: if $g \in \Mono\Cons0\omega$, say (up to a variable permutation) $g \le \pi_0$, then we may pick the $f_i$ in the proof of \cref{thm:post-ctcl-liminf} to agree with $g$ on $0111\dotsb$, hence also $f_i \le \pi_0$; then $f = \ceil\liminf(\pi_0, f_0, f_1, \dotsc)$.
\end{proof}

The following summarizes our positive classification results about $\Meet^{<\omega}$ and its variants:

\begin{theorem}
\label{thm:post-borel-Meet}
There are precisely 4 countably closed ${<}\omega_1$-ary clones on $2$ restricting to $\Meet\Cons01\Cons11^{<\omega}$ (see \cref{fig:post-borel-Meet}):
\begin{itemize}
\item  $\Meet\Cons01\Cons11^{<\omega_1} = \Pol^{<\omega_1}\brace{\wedge,0,1} =$ countable closure of $\ang{\liminf}^\Borel$.
\item  $\Meet\Cons0\omega\Cons11^{<\omega_1} = \Meet\Limm011\Cons11^{<\omega_1} = \Pol^{<\omega_1}\brace{\wedge,{\lim}{=}0,1} =$ countable closure of $\ang{\ceil{\liminf}}^\Borel$.
\item  $\Meet_\omega\Cons01\Cons11^{<\omega_1} = \Meet\Decr\Cons01\Cons11^{<\omega_1} = \Pol^{<\omega_1}\brace{\bigwedge,0,1} = \ang{\bigwedge}^\Borel$.
\item  $\ang{\Meet\Cons01\Cons11^{<\omega}}^{<\omega_1} = \Meet\Incr\Cons01\Cons11^{<\omega_1} = \Meet\Cons01\Limm111^{<\omega_1} = \Pol^{<\omega_1}\brace{\wedge,{\lim},0,1} = \ang{\wedge}^\Borel$.
\end{itemize}
Every other (non-countably-closed) ${<}\omega_1$-ary clone restricting to $\Meet\Cons01\Cons11^{<\omega}$ lies in one of the intervals in the first two bullet points, between the generated clone and its countable closure.
Moreover:
\begin{enumerate}[label=(\alph*)]
\item \label{thm:post-borel-Meet:T0}
The ${<}\omega_1$-ary clones restricting to $\Meet\Cons01^{<\omega}$ are isomorphic to these, via $F |-> \ang{F \cup \{0\}}^{<\omega_1}$.
\item \label{thm:post-borel-Meet:T1}
The ${<}\omega_1$-ary clones restricting to $\Meet\Cons11^{<\omega}$ are isomorphic to those in $[\ang{\Meet\Cons01\Cons11^{<\omega}}^\Borel, \Meet\Cons0\omega\Cons11^{<\omega_1}]$, via $F |-> \ang{F \cup \{1\}}^{<\omega_1}$, as are the ${<}\omega_1$-ary clones restricting to $\Meet^{<\omega}$, via $F |-> \ang{F \cup \{0,1\}}^{<\omega_1}$.
\item \label{thm:post-borel-Meet:ceil}
Every
$F \in [\ang{\ceil{\liminf}}^\Borel, \Meet\Cons0\omega\Cons11^{<\omega_1}]$
is $\Cons0\omega$ intersected with some
$G \in [\ang{\liminf}^\Borel, \Meet\Cons01\Cons11^{<\omega_1}]$.
(Similarly for the clones in
$[\ang{\ceil{\liminf},0}^\Borel, \Meet\Cons0\omega^{<\omega_1}]$
versus
$[\ang{\liminf,0}^\Borel, \Meet\Cons01^{<\omega_1}]$.)
\end{enumerate}
\end{theorem}

\begin{proof}
The equivalent defining invariant relations are by \cref{thm:post-meet-lim0,thm:post-meet-lim1}, the countable closure definitions are by the preceding corollaries, and that these options exhaust $\smash{\Clo^{<\omega_1}_{\Meet\Cons01\Cons11^{<\omega}}{2}}$ is by \cref{thm:kahane,thm:kahane-ceil,thm:post-meet-cts}.

\Cref{thm:post-borel-Meet:T0}
follows from \cref{thm:post-T1-mod}, which shows that the clones containing $\bigwedge$ in both fibers are isomorphic, and \cref{thm:post-meet-cts}, which shows that the only clones not containing $\bigwedge$ in both fibers are the essentially finitary ones.

\Cref{thm:post-borel-Meet:T1}
follows from \cref{thm:post-T0inf-mod} and \cref{thm:post-borel-Meet:T0},
as does
\Cref{thm:post-borel-Meet:ceil}
since $[\ang{\liminf}^\Borel, \Meet\Cons01\Cons11^{<\omega_1}]$ is sandwiched between the two intervals in \cref{thm:post-borel-Meet:T1}.
\end{proof}

We close this subsection by showing that indeed, $\ang{\liminf}^\Borel \subsetneqq \Meet\Cons01\Cons11^\Borel$:

\begin{example}
\label{ex:summable}
The indicator function $f : 2^\omega -> 2$ of the \defn{summable filter},
\begin{equation*}
f(\vec{x}) = 1  \coloniff  \sum_{\substack{i < \omega \\ x_i = 0}} \frac{1}{i+1} < \infty,
\end{equation*}
is not in $\ang{\liminf}$.
This follows from a much more general result of Kanovei--Reeken \cite{KanoveiReeken}.

Indeed, note, first, that every function $g \in \ang{\liminf}^\omega$, which is \emph{a priori} a composite of copies of $\liminf$, may be written as such a composite in ``normal form'' as follows: consider first the smallest subclass $L \subseteq \ang{\liminf}^\omega$ of functions containing all projections and closed under left composition $(h_0,h_1,\dotsc) |-> \liminf_i h_i$ with $\liminf$, \emph{subject to the constraint that the $h_i$ depend on disjoint sets of variables}; then $g$ can be obtained from some $h \in L$ via a variable substitution
\begin{align*}
g(x_0,x_1,\dotsc) = h(x_{s(0)},x_{s(1)},\dotsc) = h(s^*(\vec{x}))
\qquad \text{where} \qquad
s^*(\vec{x}) := (x_{s(0)},x_{s(1)},\dotsc)
\end{align*}
for some $s : \omega -> \omega$.
(This is easily seen via a minor variation of \cref{thm:clone-gen-left}.)
This means that $g^{-1}(1) \subseteq 2^\omega$ is the preimage of $h^{-1}(1) \subseteq 2^\omega$ under $s^* : 2^\omega -> 2^\omega$; such an $s$ is called a \defn{Rudin--Keisler reduction} of the filter $g^{-1}(1)$ to the filter $h^{-1}(1)$.
Now the construction of $h$ via left composition of $\liminf$ with functions $h_i$ with disjoint variable sets means in turn that $h^{-1}(1)$ is a \defn{Fubini product}, mod the Fréchet filter $\liminf^{-1}(1)$, of the filters $h_i^{-1}(1)$.
Thus, to say that $g \in \ang{\liminf}$ means that the filter $g^{-1}(1) \subseteq 2^\omega$ Rudin--Keisler reduces to a filter $h^{-1}(1)$ which is a transfinitely iterated Fubini product, mod the Fréchet filter, of principal filters $\pi_i^{-1}(1)$.
Equivalently, the dual ideal $\neg[g^{-1}(1)] \subseteq 2^\omega$ (i.e., subgroup mod $2$) Rudin--Keisler reduces to an ideal $\neg[h^{-1}(1)]$ which is a transfinitely iterated Fubini product, mod the Fréchet ideal of finite sets, of principal ideals (or trivial subgroups).
This implies in particular that the coset equivalence relation of the subgroup $\neg[g^{-1}(1)]$ is a continuous preimage of an iterated Fubini product mod Fréchet of equality relations.
But Kanovei--Reeken \cite{KanoveiReeken} showed that the coset equivalence relations of many Borel ideals $\subseteq 2^\omega$, including the summable ideal $\neg[f^{-1}(1)]$, do not admit a continuous (or even Borel) reduction to an iterated Fubini product mod Fréchet of equalities.
\end{example}

Thus the topmost $\approx$ interval in \cref{fig:post-borel-Meet} contains at least two distinct clones.
On the other hand, by \cref{thm:post-ctcl-liminf} there is no countable-arity relation $R \subseteq 2^\omega$ that can distinguish between such clones (i.e., that is preserved by $\liminf$ but not by the indicator function of the summable filter).
So in some sense, while \cref{thm:post-borel-Meet} does not completely classify all Borel clones with finitary restriction $\Meet\Cons01\Cons11^{<\omega}$, the remaining classification problem is ``strictly harder'' than all of the positive classifications we have obtained in this paper.

We expect there to be many other Borel clones between $\ang{\liminf}^\Borel$ and $\Meet\Cons01\Cons11^\Borel$, i.e., Borel filters (or ideals) which are indistinguishable by countable-arity relations but are not equivalent up to Rudin--Keisler reducibility and Fubini product.
Note that the proofs in \cite{KanoveiReeken} use deep tools from set-theoretic forcing and the theory of Borel equivalence relations (including ideas from Hjorth's turbulence theory \cite{Hjorth}), but show a much stronger Borel non-reducibility result; there may be an easier proof of Rudin--Keisler non-reducibility, which is all that is needed to yield distinct Borel clones.
See \cite{Hrusak}, \cite{Kanovei}, \cite{Solecki} for more information on Borel filters and ideals.

\subsection{The $k$-ary intersection property}
\label{sec:borel-T0k}

In this final section, we give some results on the structure of the left ``side tube'' $[\ang{\Mono\Cons0{<\omega}\Cons11^{<\omega}}^\Borel, \Cons02^\Borel]$.
We will produce a large yet richly structured family of distinct Borel clones lying over each $\Cons0k^{<\omega}$ and its variants, depicted in \cref{fig:post-borel-T0k} and in more detail in \cref{fig:post-borel-L0kt-T03}.

Recall from \cref{def:post-clones} the clones $\Limm011 = \Pol \{{\lim}{=}0\} \subseteq \Cons01$, of functions $f : 2^n -> 2$ ($n$ countable) such that $\vec{0} \notin \-{f^{-1}(1)}$.
Recall also the clones $\Cons0k = \Pol \{{\wedge^k}{=}0\}$ ($0 < k < \omega$), of functions $f$ such that $f^{-1}(1)$ has the \defn{$k$-ary intersection property}:
\begin{align*}
\vec{0} \notin \underbrace{f^{-1}(1) \wedge \dotsb \wedge f^{-1}(1)}_k =: f^{-1}(1)^{\wedge k}.
\end{align*}
Here by $A \wedge B$ for two sets $A, B \subseteq 2^n$, we mean $\{\vec{a} \wedge \vec{b} \mid \vec{a} \in A,\, \vec{b} \in B\}$.

\begin{definition}
\label{def:L0k}
More generally, for each $0 < k < \omega$ and $n \le \omega$, let
\begin{align*}
\Limm0kk^n :={}& \set[\big]{f : 2^n -> 2}{\text{the indicator function of } \-{f^{-1}(1)} \subseteq 2^n \text{ is in $\Cons0k$}} \\
={}& \set[\big]{f : 2^n -> 2}{\vec{0} \notin \-{f^{-1}(1)}^{\wedge k}}. \\
\intertext{It is easily seen (using compactness of $2^n$) that this is}
={}& \set[\big]{f : 2^n -> 2}{\nexists (\vec{x}_{r,q} \in f^{-1}(1))_{r<k,q<\omega} \text{ s.t.\ } \origdisplaystyle\lim_{q -> \infty} (\vec{x}_{0,q} \wedge \dotsb \wedge \vec{x}_{k-1,q}) = \vec{0}}.
\end{align*}
Thus, we define the partial operation
\begin{align*}
\lim \wedge^k : 2^\omega \cong 2^{k \times \omega} &--> 2 \\
\vec{x} = (x_i)_{i<\omega} = (x_{qk+r})_{r<k,q<\omega} &|--> \lim_{q -> \infty} (x_{qk} \wedge x_{qk+1} \wedge \dotsb \wedge x_{qk+k-1}),
\end{align*}
and put
\begin{equation*}
\Limm0kk := \Pol \{{\lim\wedge^k}{=}0\}.
\end{equation*}
This recovers the above definition of $\Limm0kk^n$ for $n$ countable, and shows that $\Limm0k{}^{<\omega_1}$ forms a clone.
\end{definition}

\begin{example}
\label{ex:L0k}
To illustrate, consider $\Limm033$.
Recalling \cref{def:pol-inv}, for a function $f$ to preserve the relation $({\lim\wedge^3}{=}0)$ means that for every $\omega \times \omega$ matrix, such as
\begin{equation*}
\begin{array}{ccccccccccccccc}
0 & 1 & 1 & 1 & 1 & 1 & 1 & 1 & 1 & \dotsm &->&f(011111111\dotsm) \\
1 & 0 & 1 & 1 & 1 & 1 & 1 & 1 & 1 & \dotsm &->&f(101111111\dotsm) \\
1 & 1 & 0 & 1 & 1 & 1 & 1 & 1 & 1 & \dotsm &->&f(110111111\dotsm) \\
\hline
0 & 1 & 1 & 0 & 1 & 1 & 1 & 1 & 1 & \dotsm &->&f(011011111\dotsm) \\
1 & 0 & 1 & 1 & 0 & 1 & 1 & 1 & 1 & \dotsm &->&f(101101111\dotsm) \\
1 & 1 & 0 & 1 & 1 & 0 & 1 & 1 & 1 & \dotsm &->&f(110110111\dotsm) \\
\hline
0 & 1 & 1 & 0 & 1 & 1 & 0 & 1 & 1 & \dotsm &->&f(011011011\dotsm) \\
1 & 0 & 1 & 1 & 0 & 1 & 1 & 0 & 1 & \dotsm &->&f(101101101\dotsm) \\
1 & 1 & 0 & 1 & 1 & 0 & 1 & 1 & 0 & \dotsm &->&f(110110110\dotsm) \\
\hline
\vdots & \vdots & \vdots & \vdots & \vdots & \vdots & \vdots & \vdots & \vdots & \ddots && \vdots
\end{array}
\end{equation*}
if in each original column, cofinally many blocks of $3$ entries have at least one $0$, then the same must hold for the resulting column obtained by applying $f$ to each row.
(See \cref{thm:fkt} for an application of the particular matrix shown above.)

Note that if this does \emph{not} hold, then after passing to a subsequence of blocks of $3$ rows, we may assume that the resulting column of $f(\dotsb)$'s is $\vec{1}$; then the rows will yield 3 sequences in $f^{-1}(1)$ whose termwise meet converges to $\vec{0}$.
By passing to further subsequences, we may assume each of these 3 sequences individually converges, and that the meet of the 3 limit strings is $\vec{0}$.
\end{example}

\begin{definition}
\label{def:L0kt}
More generally still, for $0 < k < \omega$, $0 \le t \le k$, and $n \le \omega$, let
\begin{align*}
\Limm0kt^n
:={}& \set[\big]{f : 2^n -> 2}
    {\vec{0} \notin f^{-1}(1)^{\wedge (k-t)} \wedge \-{f^{-1}(1)}^{\wedge t}
                = \smash{\underbrace{f^{-1}(1) \wedge \dotsb}_{k-t} \wedge \underbrace{\-{f^{-1}(1)} \wedge \dotsb}_{t}}}. \\
\intertext{It is easily seen that this is}
={}& \set*{f : 2^n -> 2}{\begin{aligned}
&\nexists (\vec{x}_r \in f^{-1}(1))_{r<k-t}, (\vec{x}_{r,q} \in f^{-1}(1))_{k-t\le r<k,q<\omega} \text{ s.t.\ } \\
&\qquad\origdisplaystyle\lim_{q -> \infty} (\vec{x}_0 \wedge \dotsb \wedge \vec{x}_{k-t-1} \wedge \vec{x}_{k-t,q} \wedge \dotsb \wedge \vec{x}_{k-1,q}) = \vec{0}
\end{aligned}}.
\end{align*}
Thus, we define the partial operation
\begin{align*}
\wedge^{k-t} \wedge \lim \wedge^t : 2^\omega \cong 2^{(k-t) + t \times \omega} &--> 2 \\
\vec{x} = (x_0,\dotsc,x_{k-t-1},x_{qt+r})_{k-t\le r<k,q<\omega} &|--> x_0 \wedge \dotsb \wedge x_{k-t-1} \wedge \lim_{i -> \infty} (x_{qt+k-t} \wedge \dotsb \wedge x_{qt+k-1}),
\end{align*}
and put
\begin{equation*}
\Limm0kt := \Pol \{{\wedge^{k-t}\wedge\lim\wedge^t}{=}0\}.
\end{equation*}
The matrix involved is as in \cref{ex:L0k}, except that there are $k-t$ distinguished rows at the top, then repeating blocks of $t$ rows; a column obeys ${\wedge^{k-t}\wedge\lim\wedge^t}{=}0$ iff either one of the first $k-t$ entries is $0$, or cofinally many of the repeating blocks contain at least one $0$.

When $t = 0$, $\wedge^{k-t} \wedge \lim \wedge^t$ becomes $\wedge^k$ (with domain $2^k$, rather than $2^\omega$); thus
\begin{equation*}
\Limm0k{{0}} = \Cons0k
    \quad (= \Cons01 \text{ if $k = 1$}).
\end{equation*}
On the other hand,
\begin{equation*}
\Limm0k{{k}} = \Limm0kk
    \quad (= \Limm011 \text{ if $k = 1$}).
\end{equation*}
We have the obvious relations
\begin{gather}
\label{eq:L0kt-ordering}
\Cons0\omega \subseteq \dotsb \subseteq \Limm0{k+2}k \subseteq \Limm0{k+1}k \subseteq \Limm0kk = \Limm0k{{k}} \subseteq \Limm0k{k-1} \subseteq \dotsb \subseteq \Limm0k1 \subseteq \Limm0k{{0}} = \Cons0k
\end{gather}
(see \cref{fig:post-borel-L0kt-grid}).
Based on these inclusions, it is natural to also define
\begin{align*}
\Limm0{{0}}{{{0}}} &:= \Limm0{{0}}{} := \Cons00 := \Op2, \\
\Limm0{<\omega}t &:= \bigcap_{k < \omega} \Limm0kt, \\
\Limm0{<\omega}{{<\omega}} &:= \Limm0{<\omega}{} := \bigcap_{k < \omega} \Limm0kk = \bigcap_{t \le k < \omega} \Limm0kt.
\end{align*}

Let also $\Limm1kt$ be the de~Morgan dual of $\Limm0kt$, consisting of all functions $f$ preserving the relation $({\vee^{k-t} \vee \lim \vee^t}{=}1)$ defined analogously, or (in countable arities) such that $f^{-1}(0)$ does not contain $k-t$ elements together with $t$ elements in its closure with join $\vec{1}$.
\end{definition}

The poset of all $\Limm0kt$ clones is depicted in \cref{fig:post-borel-L0kt-grid}.
\emph{However, the $\Limm0kt$'s do not form a sublattice of $\Clo^\Borel{2}$.}
For example, $\Limm031 \cap \Limm022 \supsetneqq \Limm032$; see \cref{fig:post-borel-L0kt-T03}.
In fact, there are \emph{no nontrivial meet relations} (that are not implied by the ordering) among the $\Limm0kt$'s, as we now show.

\begin{definition}
\label{def:fkt}
For $f_i : 2^{n_i} -> 2 \in \Cons11$, $i < k$, their \defn{orthogonal disjunction} is%
\footnote{Identifying $f : 2^n -> 2$ with the set $f^{-1}(1)$, a standard name from topology for $\sqcup$ would be the \emph{wedge sum} $\vee$; but that notation conflicts with ordinary disjunction.}
\begin{align*}
\smash{\bigsqcup_{i<k} f_i} : 2^{\bigsqcup_{i<k} n_i = \{(i,j) \mid i < k, j < n_i\}} -->{}& 2 \\
(x_{i,j})_{i,j} |-->{}& \bigvee_{i<k} \paren[\Big]{f_i((x_{i,j})_j) \wedge \bigwedge_{i' \ne i} \bigwedge_{j < n_{i'}} x_{i',j}} \\
={}& \forall^k_2 \paren[\Big]{\paren[\Big]{\bigwedge_{j<i} x_{i,j}}_{i<k}} \wedge \bigvee_{i<k} f_i((x_{i,j})_j).
\end{align*}
(The binary operation $\sqcup$ is associative up to a canonical permutation of the input bits.)

For example, when $f_i = 1 : 2 -> 2$ for all $i$, we get $\bigsqcup_{i<k} f_i = \forall^k_2 : 2^k -> 2$ (recall \cref{def:post-funs}).

More generally, when $k < \omega$, $f_0 = \dotsb = f_{k-t-1} = 1 : 2 -> 2$, and $f_{k-t} = \dotsb = f_{k-1} = \liminf : 2^\omega -> 2$, we may identify $2^{(k-t)+t\times\omega} \cong 2^\omega$ (assuming $t > 0$), so that $\bigsqcup_{i<k} f_i$ becomes
\begin{align*}
\forall^{k-t}_2 \sqcup \inline\liminf^{\sqcup t} : 2^\omega \cong 2^{(k-t)+t\times\omega} &--> 2 \\
\vec{x} = (x_0,\dotsc,x_{k-t-1},x_{qt+r})_{k-t\le r<k,q<\omega} &|--> \inline{\forall^k_2\paren{x_0, \dotsc, x_{k-t-1}, \bigwedge_q x_{qt+k-t}, \dotsc, \bigwedge_q x_{qt+k-1}} \wedge \liminf \vec{x}}. \\
\intertext{When $t = k$, this becomes}
\inline\liminf^{\sqcup t} : 2^\omega \cong 2^{k\times\omega} &--> 2 \\
\vec{x} = (x_{qk+r})_{r<k,q<\omega} &|--> \inline{\forall^k_2\paren{\bigwedge_q x_{qk}, \dotsc, \bigwedge_q x_{qk+k-1}} \wedge \liminf \vec{x}}. \\
\intertext{For example, when $t = k = 2$, we get}
\liminf \sqcup \liminf : 2^\omega \cong 2^{2\times\omega} &--> 2 \\
\vec{x} &|--> \paren[\Big]{\liminf_{q->\infty} x_{2q} \wedge \bigwedge_{q<\omega} x_{2q+1}} \vee \paren[\Big]{\bigwedge_{q<\omega} x_{2q} \wedge \liminf_{q->\infty} x_{2q+1}}.
\end{align*}
\end{definition}

\begin{lemma}
\label{thm:fkt}
For any $t \le k < \omega$, we have $\forall^{k-t}_2 \sqcup \liminf^{\sqcup t} \in \Mono\Cons11 \cap \bigcap_{\substack{t' \le k' < \omega \\ k' < k \text{ or } t' < t}} \Limm0{k'}{t'} \setminus \Limm0kt$.
\end{lemma}
\begin{proof}
To show that $\forall^{k-t}_2 \sqcup \liminf^{\sqcup t} \notin \Limm0kt$, consider a matrix like the one in \cref{ex:L0k}, which is the case $t = k = 3$; each column obeys $\lim \wedge^3 = 0$, while each row obeys $\liminf^{\sqcup 3} = 1$.
When $t < k$, we consider instead such a matrix split into blocks of $t$, and then prepend $(k-t)$ rows and columns which are all $1$'s except for $(k-t)$ $0$'s down the diagonal; then we again get that each column obeys $\wedge^{k-t} \wedge \lim \wedge^t = 0$ while each row obeys $\forall^{k-t}_2 \sqcup \liminf^{\sqcup t} = 1$.

It is clear that $\forall^{k-t}_2 \sqcup \liminf^{\sqcup t}$ is in $\Mono\Cons11$.
To show that it is in $\Limm0{k'}{t'}$, for $k' < k$: it is easily seen that $\-{(\forall^{k-t}_2 \sqcup \liminf^{\sqcup t})^{-1}(1)}$ has indicator function
\begin{align*}
\tag{$*$}
(\forall^{k-t}_2 \sqcup (1^\omega)^{\sqcup t})(\vec{x}) = \inline{\forall^k_2\paren{x_0, \dotsc, x_{k-t-1}, \bigwedge_q x_{qt+k-t}, \dotsc, \bigwedge_q x_{qt+k-1}}},
\end{align*}
which is in $\Cons0{k'}$, whence the former function is in $\Limm0{k'}{k'} \subseteq \Limm0{k'}{t'}$.

To show that it is in $\Limm0{k'}{t'}$, for $t' < t$: let $\vec{y}_0, \dotsc, \vec{y}_{k'-t'-1} \in (\forall^{k-t}_2 \sqcup \liminf^{\sqcup t})^{-1}(1)$ and $\vec{y}_{k'-t'}, \dotsc, \vec{y}_{k'-1} \in \-{(\forall^{k-t}_2 \sqcup \liminf^{\sqcup t})^{-1}(1)}$, i.e., when $\vec{x} = \vec{y}_{k'-t'}, \dotsc, \vec{y}_{k'-1}$, all but $1$ of the arguments of the $\forall^k_2$ on the right-hand side of ($*$) above evaluate to $1$.
Since $t' < t$, there must be some $k-t \le r < k$ such that $\bigwedge_q x_{qt+r} = 1$ whenever $\vec{x} = \vec{y}_{k'-t'}, \dotsc, \vec{y}_{k'-1}$.
Since also $\liminf_q x_{qt+r} = 1$ when $\vec{x} = \vec{y}_0, \dotsc, \vec{y}_{k'-t'-1}$, it follows that $\vec{y}_0 \wedge \dotsb \wedge \vec{y}_{k'-1}$ has infinitely many $1$'s.
This shows $\vec{0} \notin (\forall^{k-t}_2 \sqcup \liminf^{\sqcup t})^{-1}(1)^{\wedge (k'-t')} \wedge \-{(\forall^{k-t}_2 \sqcup \liminf^{\sqcup t})^{-1}(1)}^{\wedge t'}$, hence $\forall^{k-t}_2 \sqcup \liminf^{\sqcup t} \in \Limm0{k'}{t'}$.
\end{proof}

Recall that for a poset $P$, the \defn{free complete meet-semilattice} $\^P$ generated by $P$ is a complete lattice equipped with an order-embedding $P -> \^P$ whose image generates $P$ under meets, and such that every equation between meets of elements of $P$ which holds in $\^P$ is implied by the ordering, in the precise sense that any monotone map $P -> Q$ to another complete lattice $Q$ admits a unique meet-preserving extension $\^P -> Q$.
Concretely, $\^P$ may be constructed as the poset of upward-closed subsets $U \subseteq P$ under \emph{reverse} inclusion; we think of such $U$ as the formal meet of its elements.
If we exclude $U = \emptyset$, then we obtain instead the free poset with \emph{nonempty} meets $\bigwedge_{i \in I}$, $I \ne \emptyset$.

We claim that there are no nontrivial meet relations between the Borel clones $\Limm0kt^\Borel$, other than those implied by the order relations \cref{eq:L0kt-ordering} (depicted in \cref{fig:post-borel-L0kt-grid}).
Formally, this means

\begin{corollary}
\label{thm:L0kt-free}
The closure of the Borel clones $\Limm0kt^\Borel \in \Clo^\Borel{2}$, $t \le k < \omega$, under nonempty intersections yields the free complete nonempty-meet-semilattice generated by the poset depicted in \cref{fig:post-borel-L0kt-grid}, which is the poset of indices $\{(k,t) \mid t \le k < \omega\}$ under the coordinatewise reverse ordering $\ge$.
In other words, we have an order-embedding
\begin{align*}
(\{\emptyset \ne D \subseteq \{(k,t) \mid t \le k < \omega\} \mid D \text{ is $\le$-downward-closed}\}, \supseteq) &`--> (\Clo^\Borel{2}, \subseteq) \\
D &|--> \bigcap_{(k,t) \in D} \Limm0kt^\Borel.
\end{align*}
This remains an embedding upon intersecting with $\Mono$ and/or $\Cons11$.
\end{corollary}
\begin{proof}
If $D' \not\supseteq D$ are two such downward-closed sets of indices, then there is $(k,t) \in D \setminus D'$, whence for all $(k',t') \in D'$, we have $(k,t) \not\le (k',t')$; by \cref{thm:fkt}, we have $\forall^{k-t}_2 \sqcup \liminf^{\sqcup t} \in \Mono\Cons11 \cap \bigcap_{(k',t') \in D'} \Limm0{k'}{t'} \setminus \Limm0kt$.
\end{proof}

\begin{figure}[tb]
\centering
\includegraphics{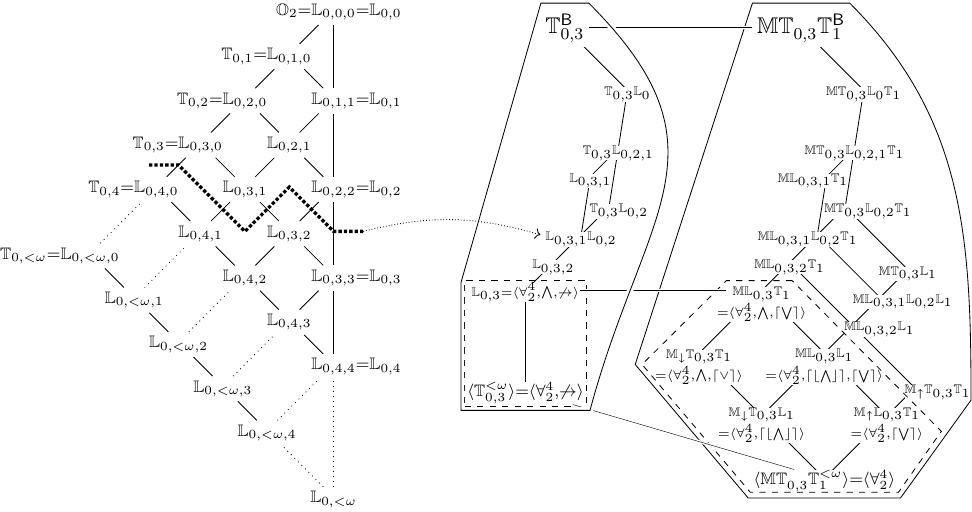}
\begin{minipage}{\textwidth}
\subcaptionbox{Inclusion ordering among the clones $\Limm0kt$, together with an upward-closed set (above thick dotted line), whose intersection $\Limm031\Limm022$ is not contained in any clones below.%
\label{fig:post-borel-L0kt-grid}}[.425\textwidth]{}
\hfill
\subcaptionbox{Some Borel clones with finitary restrictions $\Cons03^{<\omega}, \Mono\Cons03\Cons11^{<\omega}$ respectively, obtained by intersecting $\Limm0kt$'s with clones from top cube (\cref{rmk:post-topcube-irred}).
The dashed blocks are exhaustive below $\Limm033$ and are isomorphic to corresponding intervals below $\Cons0\omega$ (\cref{fig:post-borel-T0inf}).
(Similar blocks over $\Cons03\Cons11^{<\omega}, \Mono\Cons03^{<\omega}$ not shown.)%
\label{fig:post-borel-L0kt-T03}}[.525\textwidth]{}
\end{minipage}
\caption{Examples of Borel clones over left ``side tube'' of Post's lattice.}
\label{fig:post-borel-L0kt}
\end{figure}

The ``cobweb'' region of \cref{fig:post-borel-T0k}, above $\Limm0{<\omega}{}$, depicts all clones $F$ obtained by intersecting the $\Limm0kt$'s (except for $\Limm0{{0}}{{{0}}} = \Op2$), which are isomorphic to the poset of all downward-closed $\emptyset \ne D \subseteq \{(k,t) \mid t \le k < \omega\}$ under reverse inclusion.
\Cref{fig:post-borel-L0kt-grid} depicts a typical such $D$ (the region above the thick dotted line), which determines the intersection clone $F = \Limm031\Limm022$ (since these are the minimal clones above the dotted line).
We may categorize such $F$ as follows:
\begin{itemize}
\item
If $D$ contains a greatest $(k,0)$, then the intersection $F = \bigcap_{(k',t') \in D} \Limm0{k'}{t'}$ is contained in $\Limm0k{{0}} = \Cons0k$ but not $\Limm0{k+1}{{0}} = \Cons0{k+1}$, and contains $\Limm0k{{k}} = \Limm0kk \supseteq \ang{\Cons0k^{<\omega}}$ (since every $(k',t') \in D$ must be $\le (k,k)$).
Thus $F$ has finitary restriction $\Cons0k^{<\omega}$.

Note that there are precisely $2^k$ such $(k,0) \in D \notni (k+1,0)$.
Indeed, $D$ either contains or does not contain $(k,1)$; the latter $D$ are isomorphic to $\{D' \mid (k-1,0) \in D' \notni (k,0)\}$, via $D |-> D \setminus \{(k,0)\}$, as are the former $D$, via $D |-> \{(k'-1,t'-1) \mid (k',t') \in D\}$.
(In \cref{fig:post-borel-L0kt-grid}, there are $2^3$ possible thick dotted lines starting between $\Limm03{{0}}$ and $\Limm04{{0}}$ and moving right for $3$ steps, since at each step it may move down or up.)
Thus the poset of all $D$ such that $(k,0) \in D \notni (k+1,0)$ may be partitioned into two isomorphic copies of the poset of all $D'$ such that $(k-1,0) \in D' \notni (k,0)$.
By induction, there are $2^k$ such $D$.

For $k = 3$, the $2^3 = 8$ such clones are depicted in \cref{fig:post-borel-L0kt-T03} (above $\Limm033$).
Note that the ``lower and upper halves'', below $\Limm031$ or not, are each order-isomorphic to the $4$ intersection clones obtained for $k = 2$.
\end{itemize}

\begin{corollary}
\label{thm:post-borel-L0kt-2^k}
For each $k < \omega$, there are at least $2^k$ Borel clones between $\Limm0kk^\Borel$ and $\Cons0k^\Borel$ with finitary restriction $\Cons0k^{<\omega}$.
Similarly for $\Mono\Cons0k^{<\omega}, \Cons0k\Cons11^{<\omega}, \Mono\Cons0k\Cons11^{<\omega}$.
\qed
\end{corollary}

\begin{itemize}[resume*]
\item
On the other hand, if $D$ contains $(k,0)$ for all $k < \omega$, then the intersection $F$ is contained in $\bigcap_{k<\omega} \Cons0k = \Cons0{<\omega}$, and contains $\Limm0{<\omega}{} \supseteq \ang{\Cons0{<\omega}^{<\omega}}$.
Thus $F$ has finitary restriction $\Cons0{<\omega}^{<\omega}$.

We may further categorize such $F$, by considering whether or not there is a largest $t < \omega$ such that $(k,t) \in D$ for all $k < \omega$.
If not, then $F = \Limm0{<\omega}{}$.
If so, then $F$ is contained in $\Limm0{<\omega}t$ but not $\Limm0{<\omega}{t+1}$.
Then there is a least $k \ge t$ such that $(k+1,t+1) \notin D$, which means $F \subseteq \Limm0k{t+1}$ but $F \not\subseteq \Limm0{k+1}{t+1}$.
For fixed $t, k$, the poset of all such $D$ is order-isomorphic to the poset of $D'$ with $(k-t-1,0) \in D \notni (k-t,0)$.
(This means the thick dotted line in \cref{fig:post-borel-L0kt-grid} starts between $\Limm0{<\omega}t$ and $\Limm0{<\omega}{t+1}$, and moves up until just before it crosses $\Limm0k{t+1}$, then bends down.)
So the poset of all $D$ with $\{(k,0) \mid k \in \omega\} \subseteq D$, except for the maximum $D = \{(k,t) \mid t \le k < \omega\}$, may be partitioned into countably many mutually isomorphic subsets, each of which may be partitioned into copies of $\{D' \mid (k,0) \in D' \notni (k+1,0)\}$ for all $k$.
In particular, there are countably infinitely many such $D$.
\end{itemize}

\begin{corollary}
\label{thm:post-borel-L0inf-inf}
There are at least countably infinitely many distinct Borel clones above $\Limm0{<\omega}{}^\Borel$ with finitary restriction $\Cons0{<\omega}^{<\omega}$.
Similarly for $\Mono\Cons0{<\omega}^{<\omega}, \Cons0{<\omega}\Cons11^{<\omega}, \Mono\Cons0{<\omega}\Cons11^{<\omega}$.
\qed
\end{corollary}

\begin{remark}
It would be desirable if the recursive order-isomorphisms of the various kinds of index sets $D \subseteq \{(k,t) \mid t \le k < \omega\}$ mentioned in the above discussion could be enhanced to isomorphisms between the respective intervals of clones, perhaps by using certain ``operators'' on functions along the lines of \crefrange{sec:post-ceil-floor-xsec}{sec:post-beta}.
For instance, perhaps the (Borel) clones between $\Limm031, \Cons03^{<\omega}$ (see \cref{fig:post-borel-L0kt-T03}) could be shown to be isomorphic to those between $\Cons02, \Cons02^{<\omega}$.

This would be beneficial to the remaining classification problem for the Borel clones in the ``side tubes''.
We currently only know of the finite lower bound of $2^k$ (plus a little bit; see below).
For instance, in stark contrast to what \cref{fig:post-borel-T0k} suggests, we currently cannot even rule out the possibility of there being strictly more Borel clones over $\Cons02^{<\omega}$ than over $\Cons03^{<\omega}$.
\end{remark}

We now show a converse to \cref{thm:fkt}: the functions $\forall^{k-t}_2 \sqcup \liminf^{\sqcup t}$ are ``minimal'' $\notin \Limm0kt$.
This generalizes \cref{thm:kahane} (when $k = t = 1$), and is also analogous to \cref{thm:post-forall2}.

\begin{lemma}
\label{thm:fkt-gen}
If $f : 2^n -> 2 \in \Op2^{<\omega_1} \setminus \Limm0kt$, then $\forall^{k-t}_2 \sqcup \liminf^{\sqcup t} \in \ang{\{f\} \cup \Mono\Cons0\omega\Cons11^\Borel} = \ang{f, \bigwedge, \ceil{\bigvee}}$.
\end{lemma}
\begin{proof}
Since $f \notin \Limm0kt$, there are $\vec{x}_0, \dotsc, \vec{x}_{k-t-1}, (\vec{x}_{k-t,q})_{q<\omega}, \dotsc, (\vec{x}_{k-1,q})_{q<\omega}$ in $f^{-1}(1)$ such that
\begin{align*}
\lim_{q -> \infty} (\vec{x}_0 \wedge \dotsb \wedge \vec{x}_{k-t-1} \wedge \vec{x}_{k-t,q} \wedge \dotsb \wedge \vec{x}_{k-1,q}) = \vec{0}.
\end{align*}
We now modify these $\vec{x}_r, \vec{x}_{r,q}$'s to make them resemble points in $(\forall^{k-t}_2 \sqcup \liminf^{\sqcup t})^{-1}(1)$.

By passing to subsequences, we may assume that for each $k-t \le r < k$, the sequence $(\vec{x}_{r,q})_q$ converges to some $\vec{x}_r \in 2^n$.
Thus $\bigwedge_{0\le r<k} \vec{x}_r = \vec{0}$.

In fact, we may assume that $\vec{x}_{r,0} \ge \vec{x}_{r,1} \ge \vec{x}_{r,2} \ge \dotsb$ for each $k-t \le r < k$.
Indeed, this becomes true if we replace each $\vec{x}_{r,q}$ with $\bigvee_{p \ge q} \vec{x}_{r,p}$ (as in \cref{thm:post-wadge}\cref{thm:post-wadge:MT0}), and the above limit remains $\vec{0}$, provided that each $(\vec{x}_{r,q})_q$ converges.
Now these new points $\vec{x}_{r,q}$ may no longer be in $f^{-1}(1)$; but since they are $\ge$ the original points, and there are only countably many, by \cref{thm:post-up-ctcl} we may replace $f$ with another function in $\ang{\{f\} \cup \Mono\Cons0\omega\Cons11^\Borel}$ which is $1$ at these new points.

Finally, we may assume that the $\vec{x}_r$'s (= $\lim_q \vec{x}_{r,q}$ for $r \ge k-t$) have pairwise joins $\vec{1}$.
This is by replacing each $\vec{x}_r$ (or $\vec{x}_{r,q}$) with its join with $\bigvee_{s<r} \neg \vec{x}_s$.
Again, these new points are $\ge$ the original ones, so by \cref{thm:post-up-ctcl} we may replace $f$ to make it $1$ at these new points.
We still have $\bigwedge_{r<k} \vec{x}_r = \vec{0}$; thus for each $j < n$, there is a unique $r < k$ such that $x_{r,j} = 0$.

Now define the variable substitution
\begin{align*}
g : 2^{(k-t)+t\times\omega} &--> 2^n \\
(y_0,\dotsc,y_{k-t-1},y_{qt+r})_{k-t \le r < k, q < \omega} &|--> \paren*{\begin{cases}
y_r &\text{if $x_{r,j} = 0$ for $r < k-t$,} \\
y_{\min \{q < \omega \mid x_{r,q,j} = 0\} \cdot t + r} &\text{if $x_{r,j} = 0$ for $k-t \le r < k$}
\end{cases}}_{j<n}.
\end{align*}
Then $g(\vec{y}) \in f^{-1}(1)$ whenever $\vec{y} \in (\forall^{k-t}_2 \sqcup \liminf^{\sqcup t})^{-1}(1)$, so $\forall^{k-t}_2 \sqcup \liminf^{\sqcup t} \le f \circ g \in \ang{f}$.
Since clearly $\forall^{k-t}_2 \sqcup \liminf^{\sqcup t} \in \ang{\bigwedge, \bigvee} \subseteq \ang{\Mono\Cons0\omega\Cons11^\Borel \cup \{1\}}$, by \cref{thm:post-lower} it follows that $\forall^{k-t}_2 \sqcup \liminf^{\sqcup t} \in \ang{\{f\} \cup \Mono\Cons0\omega\Cons11^\Borel}$.
\end{proof}

\begin{corollary}
\label{thm:post-L0kt-dich}
For a ${<}\omega_1$-ary clone $F \supseteq \Mono\Cons0\omega\Cons11^\Borel$, either $F \subseteq \Limm0kt$ or $\forall^{k-t}_2 \sqcup \liminf^{\sqcup t} \in F$.
\qed
\end{corollary}

\begin{corollary}
\label{thm:post-L0kt-partition}
Every ${<}\omega_1$-ary clone $F \supseteq \Mono\Cons0\omega\Cons11^\Borel$ belongs to exactly one of the intervals
\begin{align*}
\textstyle
[\ang{\{\forall^{k-t}_2 \sqcup \liminf^{\sqcup t} \mid (k,t) \notin D\}}^\Borel, \bigcap_{(k,t) \in D} \Limm0kt^{<\omega_1}]
\end{align*}
for a downward-closed $\emptyset \ne D \subseteq \{(k,t) \mid t \le k < \omega\}$ (namely for $D = \{(k,t) \mid F \subseteq \Limm0kt\}$).
\end{corollary}
\begin{proof}
For said $D$, for every $(k,t) \notin D$, we have $\forall^{k-t}_2 \sqcup \liminf^{\sqcup t} \in F$ by the preceding corollary.
And these intervals are disjoint for $D \ne D'$: if without loss of generality $D \not\subseteq D'$, then for $(k,t) \in D$, $\forall^{k-t}_2 \sqcup \liminf^{\sqcup t}$ is in every clone in the latter interval but none in the former by \cref{thm:fkt}.
\end{proof}

\begin{remark}
\label{rmk:post-borel-L0kt-gen}
Each of the above intervals is nonempty, i.e., the lower bound is contained in the upper bound, by \cref{thm:fkt}; in fact there are at least 4 Borel clones in the interval, namely $\bigcap_{(k,t) \in D} \Limm0kt^\Borel$ and its intersections with $\Mono$ and/or $\Cons11$, or alternatively the clones generated by $\{\forall^{k-t}_2 \sqcup \liminf^{\sqcup t} \mid (k,t) \notin D\}$ together with $\Cons0\omega^\Borel, \Cons0\omega\Cons11^\Borel, \Mono\Cons0\omega^\Borel, \Mono\Cons0\omega\Cons11^\Borel$.

We do not know if, within each combination of $\Mono$ and/or $\Cons11$, the generators match the relations.
For example, we do not know if each inclusion
\begin{equation*}
\ang{\Cons0\omega^\Borel \cup \{\forall^{k-t}_2 \sqcup \liminf^{\sqcup t} \mid (k,t) \notin D\}}^\Borel \subseteq \bigcap_{(k,t) \in D} \Limm0kt^\Borel
\end{equation*}
is in fact an equality.
The only known cases are when $D \subseteq \{(0,0), (1,0), (1,1)\}$, which yield the clones over $\Cons01^{<\omega}, \Op2^{<\omega}$ from \cref{sec:borel-topcube}; $D = \{(k',t) \mid t \le k' \le k\}$ for fixed $k < \omega$, which yield the clones $\Limm0kk^\Borel$ whose generators are determined by \cref{ex:post-borel-L0k-gen} below; and $(2,0) \in D \notni (2,1)$, which yield a few sporadic clones over $(\Mono)\Cons02(\Cons11)^{<\omega}$ related to self-dual functions (see \cref{thm:post-borel-T02}).

In contrast to the analogous situation with the Borel clones over $\Meet^{<\omega}$ (\cref{thm:post-borel-Meet}), here we do not even know if such generation occurs up to countable closure, i.e., up to approximation at any given countably many input strings in $2^\omega$.
\end{remark}

The above dichotomies only apply to Borel clones $F$ containing $\Mono\Cons0\omega\Cons11^\Borel$, i.e., whose ``projection'' $F \cap \Cons0\omega$ to the ``base of the side tube'' as in \cref{thm:post-T0inf-mod,fig:post-mod-down} is one of the top 4 nodes given by \cref{thm:post-borel-T0omega} (\cref{fig:post-borel-T0inf}).
In order to extend these results to $F$ for which $F \cap \Cons0\omega$ is one of the other 11 clones in \cref{thm:post-borel-T0omega}, we need to consider the interactions between the $\Limm0kt$ and the clones from the top cube (\cref{rmk:post-topcube-irred}).

\begin{proposition}
\label{thm:post-borel-L0k-down}
For each $k < \omega$, we have $\Limm0kk^{<\omega_1} = \down \ang{\Mono\Cons0k\Cons11^{<\omega}}^\Borel$.
\end{proposition}
(Recall \cref{def:down} of the downward-closure $\down$, as well as \cref{def:borel-down} for the Borel downward-closure $\down^\Borel$.
Note that it follows that $\Limm0kk^\Borel = \down^\Borel \ang{\Mono\Cons0k\Cons11^{<\omega}}^\Borel$.)
\begin{proof}
Clearly from its definition, $\Limm0kk$ is downward-closed, which shows $\supseteq$.
To show $\subseteq$:
let $f : 2^n -> 2 \in \Limm0kk^{<\omega_1}$.
If $f = 0$, then $f \le \pi_0 \in \ang{\Mono\Cons0k\Cons11^{<\omega}}$; so suppose not.
Then $\-{f^{-1}(1)} \subseteq 2^n$ has the $k$-ary intersection property, i.e., $\-{f^{-1}(1)}^k \subseteq (2^n)^k$ is disjoint from $(\wedge^k : (2^n)^k -> 2^n)^{-1}(\vec{0})$.
Since $\-{f^{-1}(1)} \subseteq 2^n$ is a downward-directed intersection of clopen sets, by compactness of $(2^n)^k$, for some such clopen $f^{-1}(1)(1) \subseteq A \subseteq 2^n$, $A^k$ must also be disjoint from $(\wedge^k)^{-1}(\vec{0})$, i.e., $A$ also has the $k$-ary intersection property.
Then the upward-closure $\up A \subseteq 2^n$ is clearly a clopen upward-closed set still with the $k$-ary intersection property, hence its indicator function $g$ is in $\ang{\Mono\Cons0k\Cons11^{<\omega}}$ with $f \le g$.
\end{proof}

\begin{corollary}
\label{thm:post-borel-L0k}
For each $k < \omega$, there are 2, 3, 4, 6 Borel clones contained in $\Limm0kk$ restricting to the finitary clones
$\Cons0k^{<\omega},
\Cons0k\Cons11^{<\omega},
\Mono\Cons0k^{<\omega},
\Mono\Cons0k\Cons11^{<\omega}$
respectively, namely generated by the respective clones below $\Cons0\omega$ from \cref{thm:post-borel-T0omega} (\cref{fig:post-borel-T0inf}) together with the generator $\forall^{k+1}_2 \in \Mono\Cons0k\Cons11^{<\omega}$.
\end{corollary}
The simplest (2) and most complicated (6) of these are depicted in \cref{fig:post-borel-L0kt-T03} (dashed blocks), for $k = 3$; for the others, see \cref{fig:post-borel-T0k} (solid-shaded regions).
\begin{proof}
By \cref{thm:post-borel-down-mod} applied to $H = \ang{\Mono\Cons0k\Cons11^{<\omega}}^\Borel$.
\end{proof}

\begin{example}
\label{ex:post-borel-L0k-gen}
We get
\begin{align*}
\Mono\Limm0kk\Cons11^\Borel
&= \ang{\Mono\Cons0\omega\Cons11^\Borel \cup \{\forall^{k+1}_2\}}^\Borel
= \ang{\bigwedge, \ceil{\bigvee}, \forall^{k+1}_2}^\Borel, \\
\Limm0kk^\Borel
&= \ang{\Cons0\omega^\Borel \cup \{\forall^{k+1}_2\}}^\Borel
= \ang{\bigwedge, -/>, \forall^{k+1}_2}^\Borel
\end{align*}
(as shown in \cref{fig:post-borel-L0kt-T03} for $k = 3$).
Note that $\forall^{k+1}_2 = \forall^{k+1-0}_2 \sqcup \liminf^{\sqcup 0}$ (\cref{def:fkt}), which matches the candidate generating set for $\Limm0kk^\Borel$ from \cref{rmk:post-borel-L0kt-gen}.
\end{example}

\begin{remark}
\label{thm:post-T0k-decr}
For $k < \omega$, we have $\Decr\Cons0k^{<\omega_1} \subseteq \Limm0kk$, since for $f \in \Decr^{<\omega_1}$, $f^{-1}(1)$ is already closed (by the dual of \cref{thm:scott-prod}; cf.\ \cref{thm:mono-meetdown-0}).
\end{remark}

\begin{corollary}
\label{thm:post-borel-L0inf-down}
$\Limm0{<\omega}{}^{<\omega_1} = \down \Decr\Cons0{<\omega}\Cons11^\Borel$.
\end{corollary}
\begin{proof}
$\supseteq$ by intersecting the preceding remark over $k$; $\subseteq$ because for $f \in \Limm0{<\omega}{}^{<\omega_1}$, by \cref{thm:post-borel-L0k-down} we may write $f \le g_k \in \ang{\Mono\Cons0k\Cons11^{<\omega}}^\Borel$ for each $k$, whence $f \le \bigwedge_k g_k \in \Decr\Cons0{<\omega}\Cons11^\Borel$.
\end{proof}

\begin{remark}
\label{rmk:post-borel-L0inf}
It follows from \cref{thm:post-borel-down-mod} that
$[\Decr\Cons0{<\omega}\Cons11^\Borel, \Limm0{<\omega}{}^\Borel] \cong [\Decr\Cons0\omega\Cons11^\Borel, \Cons0\omega^\Borel]$,
which consists of 6 clones (see \cref{fig:post-borel-T0inf}).

If, as suggested in \cref{rmk:forall2}, we had $\Decr\Cons0{<\omega}\Cons11^\Borel = \ang{\forall_2}^\Borel$, then together with \cref{thm:post-borel-T0omega} this would give a complete classification of all Borel clones in $[\ang{\Mono\Cons0{<\omega}\Cons11^{<\omega}}^\Borel, \Limm0{<\omega}{}^\Borel]$.
\end{remark}

The following is analogous to \cref{thm:post-T0inf-lim1}:

\begin{lemma}
\label{thm:post-T0k-lim1}
For each $0 < k < \omega$, we have $\Cons0k\Limm111^{<\omega_1} \subseteq \Limm0{k-1}{}$.
\end{lemma}
\begin{proof}
We have the following positive-primitive definition of the relation $({\lim\wedge^{k-1}}{=}0)$ (\cref{def:L0k}), which we treat here as a subset of $(2^\omega)^{k-1}$ rather than $2^{(k-1) \times \omega} \cong 2^\omega$ for clarity:
\begin{align*}
\lim (\vec{x}_0 \wedge \dotsb \wedge \vec{x}_{k-2}) = 0
&\iff  \exists \vec{y} \in 2^\omega\, \paren[\Big]{(\lim \vec{y} = 1) \wedge \bigwedge_{q < \omega} (x_{0,q} \wedge \dotsb \wedge x_{k-2,q} \wedge y_q = 0)}.
\qedhere
\end{align*}
\end{proof}

Analogously to \cref{thm:post-T0inf-incr}, we may deduce from this that $\Incr\Cons0k^{<\omega_1} \subseteq \Limm0{k-1}{}$.
But in fact:

\begin{lemma}
\label{thm:post-T0k-incr}
For each $k < \omega$, we have $\Incr\Cons0k^{<\omega_1} \subseteq \Limm0k{k-1}$.
\end{lemma}
\begin{proof}
We have the following positive-primitive definition of the relation $({\wedge^1\wedge\lim\wedge^{k-1}}{=}0)$ (\cref{def:L0kt}), which we treat here as a subset of $2 \times (2^\omega)^{k-1}$:
\begin{align*}
x_0 \wedge \lim (\vec{x}_1 \wedge \dotsb \wedge \vec{x}_{k-1}) = 0
&\iff  \exists \vec{y} \in 2^\omega\, \paren[\Big]{\paren[\big]{\bigveeup \vec{y} = x_0} \wedge \bigwedge_{q < \omega} (x_{1,q} \wedge \dotsb \wedge x_{k-1,q} \wedge y_q = 0)}.
\qedhere
\end{align*}
\end{proof}

These lemmas explain why in the right $\Mono\Cons03\Cons11$ part of \cref{fig:post-borel-L0kt-T03}, below the primary ``cobweb'' of 8 clones $[\Mono\Limm033\Cons11, \Mono\Cons03\Cons11]$, there is a ``secondary row'' of only the clones below $\Limm022$ intersected with $\Limm111$, and why below only $\Limm032$, there is a ``tertiary'' intersection with $\Incr$.
(Over $\Cons0k\Cons11$, there would only be the ``secondary row'', whereas over $\Mono\Cons0k$, there would only be the ``tertiary'' intersection with $\Incr$; see \cref{fig:post-borel-T0k}.)

We now show that such ``secondary'' and ``tertiary'' Borel clones indeed correspond to the obvious ``primary'' clones in the ``cobweb'' suggested by \cref{fig:post-borel-L0kt-T03,fig:post-borel-T0k}.

\begin{proposition}
\label{thm:post-T0k-lim1-mod}
For each $0 < k < \omega$, we have modularity isomorphisms
\begin{gather*}
\begin{aligned}
[\ang{\Mono\Cons0k\Cons11^{<\omega}}^\Borel, \Cons0k\Limm111^{<\omega_1}]
&\cong [\ang{\Mono\Cons0k^{<\omega}}^\Borel, \Cons0k\Limm0{k-1}{}^{<\omega_1}] \\
F &|-> \ang{F \cup \{0\}}^{<\omega_1} \\
G \cap \Limm111 &<-| G,
\end{aligned}
\\[1ex]
\begin{alignedat}{2}
[\Decr\Cons0k\Limm111^\Borel, \Cons0k\Limm111^{<\omega_1}] &\cong
[\Decr\Cons0k\Cons11^\Borel, \Cons0k\Limm0{k-1}{}\Cons11^{<\omega_1}] &&\cong
[\Decr\Cons0k^\Borel, \Cons0k\Limm0{k-1}{}^{<\omega_1}] \\
F &|-> \ang{F \cup \{\bigwedge\}}^{<\omega_1} &&|-> \ang{F \cup \{0\}}^{<\omega_1}.
\end{alignedat}
\end{gather*}
\end{proposition}
\begin{proof}
The latter two isomorphisms above follow from the first isomorphism above together with the second isomorphism in \cref{thm:post-T1-mod}.

For the first isomorphism above, note first that the map is well-defined, since the upper bound of the left-hand interval is equal to $\Cons0k\Limm0{k-1}{}\Limm111^{<\omega_1}$ by \cref{thm:post-T0k-lim1}.
And it is an embedding by the first isomorphism in \cref{thm:post-T1-mod}, with image consisting of all clones $G$ in the right-hand interval such that $G \subseteq \down (G \cap \Limm111)$; thus it suffices to prove this for all $G$ in the right-hand interval.
For that, it suffices to prove that
\begin{equation*}
\forall g \in \Limm0{k-1}{}^{<\omega_1},\; g \in \down \paren[\big]{\ang{\{g\} \cup \Mono\Cons0k\Cons11^{<\omega}} \cap \Limm111}.
\end{equation*}
By \cref{thm:post-borel-L0k-down}, $\Limm0{k-1}{}^{<\omega_1} = \down \ang{\Mono\Cons0{k-1}\Cons11^{<\omega}}^\Borel = \down \ang{\forall^k_2}^\Borel$; thus it suffices to prove
\begin{align*}
\ang{\forall^k_2}^\Borel \subseteq \set[\big]{f \in \ang{\Mono\Cons0{k-1}\Cons11^{<\omega}}^\Borel}{\forall g \le f,\; g \in \down \paren[\big]{\ang{\{g\} \cup \Mono\Cons0k\Cons11^{<\omega}} \cap \Limm111}}.
\end{align*}
Clearly the right-hand side contains each $\pi_i$ ($g \le \pi_i \in \ang{\emptyset} \cap \Limm111$); thus it suffices (\cref{thm:clone-gen-left}) to prove that it is closed under left-composition with $\forall^k_2$.
Let $f_0, \dotsc, f_{k-1} : 2^n -> 2$ belong to the right-hand side, and let $g \le \forall^k_2(f_0, \dotsc, f_{k-1})$.
Then by the induction hypothesis, there are $h_i$ with
\begin{align*}
f_i \wedge g \le h_i
&\in \ang{\{f_i \wedge g\} \cup \Mono\Cons0k\Cons11^{<\omega}} \cap \Limm111 \\
&\subseteq \ang{\{g\} \cup \Mono\Cons0k\Cons11^{<\omega}} \cap \Limm111
\end{align*}
by \cref{thm:post-lower} since $f_i \wedge g \in \down \ang{\{g\} \cup \Mono\Cons0k\Cons11^{<\omega}} \cap \ang{\{g\} \cup \Mono\Cons0k\Cons11^{<\omega} \cup \{1\}}$.
Then
\begin{align*}
g
&\le \forall^{k+1}_2(g, h_0, \dotsc, h_{k-1})
\in \ang{\{g\} \cup \Mono\Cons0k\Cons11^{<\omega}} \cap \Limm111,
\end{align*}
since $g = g \wedge \forall^k_2(f_0, \dotsc, f_{k-1}) \le \forall^{k+1}_2(g, h_0, \dotsc, h_{k-1})$, and $h_0 \wedge \dotsb \wedge h_{k-1} \le \forall^{k+1}_2(g, h_0, \dotsc, h_{k-1})$ whence $\forall^{k+1}_2(g, h_0, \dotsc, h_{k-1})$ is continuous at $\vec{1}$.
\end{proof}

\begin{corollary}
\label{thm:post-T0k-lim1-down}
For any $t < k < \omega$, we have $\Limm0kt\Limm0{k-1}{}^{<\omega_1} = \down \Mono\Limm0kt\Limm0{k-1}{}\Limm111^{<\omega_1}$ and $\Limm0kt\Limm0{k-1}{}^\Borel = \down^\Borel \Mono\Limm0kt\Limm0{k-1}{}\Limm111^\Borel$.
\end{corollary}
\begin{proof}
$\supseteq$ is obvious; $\subseteq$ follows from the above and that $\ang{F \cup \{0\}}^{<\omega_1} \subseteq \down F$ for a ${<}\omega_1$-ary clone $F \subseteq \Mono$ (\cref{thm:down-T0inf}), as well as \cref{thm:post-borel-up} in the Borel case to get $\smash{\Limm0kt\Limm0{k-1}{}^\Borel}
\subseteq \smash{\down^\Borel \Mono\Limm0kt\Limm0{k-1}{}^\Borel}
\subseteq \smash{\down^\Borel \ang{\Mono\Limm0kt\Limm0{k-1}{}\Limm111^\Borel \cup \{0\}}^\Borel}
\subseteq \smash{\down^\Borel \Mono\Limm0kt\Limm0{k-1}{}\Limm111^\Borel}$.
\end{proof}

Thus for instance, in \cref{fig:post-borel-L0kt-T03}, the interval $[\Limm033, \Cons03\Limm022]$ is isomorphic to the two parallel intervals on the right.
Note that the proof of \cref{thm:post-T0k-lim1-mod} applies to \emph{all ${<}\omega_1$-ary clones in these intervals, not just the 4 shown}.
This includes non-Borel ${<}\omega_1$-ary clones, as well as any potential unknown Borel clones between the upper 3 shown and their proposed generators from \cref{rmk:post-borel-L0kt-gen}.

\begin{remark}
\label{rmk:post-T0k-lim1-mod}
From the proof of \cref{thm:post-T0k-lim1-mod}, we may read off a procedure for obtaining an upper bound $h \in \ang{\{g\} \cup \Mono\Cons0k\Cons11^{<\omega}} \cap \Limm111$ for $g \in \Cons0k\Limm0{k-1}{}$, given an upper bound $g \le f \in \ang{\forall^k_2}$.
Namely, we induct on the construction of $f$; whenever $f = \forall^k_2(f_0,\dotsc,f_{k-1})$, then we recursively find upper bounds $h_i$ for $f_i \wedge g$, and then take
\begin{align*}
h := \forall^{k+1}_2(g, h_0, \dotsc, h_{k-1})
= g \vee (h_0 \wedge \dotsb \wedge h_{k-1}).
\end{align*}
By \cref{thm:post-lower}, it follows that the preimage of $\ang{\{g\} \cup \Mono\Cons0k\Cons11^{<\omega}}$ under the second isomorphism in \cref{thm:post-T0k-lim1-mod} is $\ang{\{h\} \cup \Mono\Cons0k\Cons11^{<\omega}}$ (since $g$ is generated by the latter together with $0$).

We may apply this to $g = \forall^{k-t}_2 \sqcup \liminf^{\sqcup t} \in \Mono\Limm0k{t-1}\Limm0{k-1}{}\Cons11 \setminus \Limm0kt$ from \cref{thm:fkt}, for each $1 \le t \le k$.
An easy upper bound is $f = \forall^k_2(\pi_0,\dotsc,\pi_{k-1})$, which gives
\begin{align*}
h
&= (\forall^{k-t}_2 \sqcup \inline\liminf^{\sqcup t}) \vee (\pi_0 \wedge \dotsb \wedge \pi_{k-1}).
\end{align*}
For example, when $k = t = 2$ (cf.\ \cref{def:fkt}), we get
\begin{align*}
h(\vec{x}) &= \paren[\Big]{\liminf_{q->\infty} x_{2q} \wedge \bigwedge_{q<\omega} x_{2q+1}} \vee \paren[\Big]{\bigwedge_{q<\omega} x_{2q} \wedge \liminf_{q->\infty} x_{2q+1}} \vee (x_0 \wedge x_1).
\end{align*}
Using the above, we may convert \cref{thm:fkt-gen} into a ``dichotomy'' for such $h$:
\end{remark}

\begin{corollary}
\label{thm:fkt-lim1-gen}
For any $0 < t \le k < \omega$ and function
$f : 2^n -> 2 \in \Cons0k\Limm111^{<\omega_1} \setminus \Limm0kt$,
we have
$(\forall^{k-t}_2 \sqcup \inline\liminf^{\sqcup t}) \vee (\pi_0 \wedge \dotsb \wedge \pi_{k-1}) \in \ang{\{f\} \cup \Mono\Limm0kk\Limm111^\Borel} = \ang{f, \forall^{k+1}_2, \ceil{\floor{\bigwedge}}, \ceil{\bigvee}}$.
\end{corollary}
\begin{proof}
The clone $F := \ang{\{f\} \cup \Mono\Limm0kk\Limm111^\Borel}^{<\omega_1} \not\subseteq \Limm0kt\Limm111$ corresponds via the first isomorphism in \cref{thm:post-T0k-lim1-mod} to $\ang{F \cup \{0\}}^{<\omega_1} \not\subseteq \Limm0kt$, which hence contains $\ang{\{\forall^{k-t}_2 \sqcup \inline\liminf^{\sqcup t}\} \cup \Mono\Limm0kk^\Borel}^\Borel$ by \cref{thm:fkt-gen}, whence the original clone $F$ contains $h$ from the preceding remark.

(The generators $\forall^{k+1}_2, \ceil{\floor{\bigwedge}}, \ceil{\bigvee}$ for $\Mono\Limm0kk\Limm111^\Borel$ are by \cref{thm:post-borel-L0k}.)
\end{proof}

This completes the description of the ``secondary row'' $[\Mono\Limm033\Limm111, \Mono\Cons03\Limm111]$ in \cref{fig:post-borel-L0kt-T03}, again modulo the caveat (as in \cref{rmk:post-borel-L0kt-gen}) that we do not know if the candidate generators for $(\Mono)\Limm0kt\Limm111^\Borel$ match the relations once $t < k$.

It remains to explain the single ``tertiary'' node $\Incr\Cons03\Cons11 = \Incr\Limm032\Limm111$ in \cref{fig:post-borel-L0kt-T03}:

\begin{proposition}
\label{thm:post-T0k-incr-down}
For each $1 \le k < \omega$, we have $\Limm0k{k-1}^{<\omega_1} = \down \Incr\Cons0k\Cons11^\Borel$.
\end{proposition}
\begin{proof}
$\supseteq$ by \cref{thm:post-T0k-incr}.
Now let $f : 2^n -> 2 \in \Limm0k{k-1}^{<\omega_1}$; thus
\begin{align*}
\vec{0} &\notin f^{-1}(1) \wedge \-{f^{-1}(1)}^{\wedge (k-1)}. \\
\intertext{Using metrizability of $2^n$ for $n \le \omega$, write $\-{f^{-1}(1)} \subseteq 2^n$ as a countable decreasing intersection of open sets $U_0 \supseteq \-{U_1} \supseteq U_1 \supseteq \-{U_2} \supseteq U_2 \supseteq \dotsb$.
Then for each $\vec{x} \in f^{-1}(1)$, there is some $i$ such that}
\vec{0} &\notin \vec{x} \wedge \-{U_i}^{\wedge (k-1)},
\text{ i.e., }
\vec{x} \notin \down \neg [\-{U_i}^{\wedge (k-1)}].
\end{align*}
Let $U := \bigcup_i (U_i \setminus \down \neg [\-{U_i}^{\wedge (k-1)}])$.
Then $f^{-1}(1) \subseteq U$, and $\vec{0} \not\in U^{\wedge k}$, since for $\vec{x}_0, \dotsc, \vec{x}_{k-1} \in U$ where each $\vec{x}_j \in U_{i_j} \setminus \down \neg [\-{U_{i_j}}^{\wedge (k-1)}]$, with $i_0 \le i_1, \dotsc, i_{k-1}$, we have $\vec{x}_0 \wedge \dotsb \wedge \vec{x}_{k-1} \in \vec{x}_0 \wedge U_{i_1} \wedge \dotsb \wedge U_{i_{k-1}} \subseteq \vec{x}_0 \wedge U_{i_0}^{\wedge (k-1)} \notni \vec{0}$.
So the indicator function $g$ of the upward-closure of $U$ obeys $f \le g \in \Incr\Cons0k\Cons11$.
\end{proof}

\begin{remark}
The above argument is reminiscent of the $T_5$ separation axiom in metrizable spaces.
As such, it fails for uncountable arities $n > \omega$.
Indeed, if $2^n$ is not $T_5$ (see e.g., \cite[4.F]{Kelley}), then taking $A, B \subseteq 2^n$ disjoint from each other's closures but inseparable by open sets, and then embedding $2^n `-> 2^{2+2n}$ as an antichain as in \cref{rmk:post-up-analytic}, with meet $\ne \vec{0}$ and join $\ne \vec{1}$ (e.g., take only strings beginning with $01$), the indicator function $f$ of $A \cup \neg[B] \subseteq 2^{2+2n}$ will be in the $n$-ary analogue of $\Limm021$ (meaning $\vec{0} \not\in f^{-1}(1) \wedge \-{f^{-1}(1)}$) but not $\le$ any Scott-continuous $\Cons02$ function.
\end{remark}

It follows from the above (and the fact that $\Limm0k{k-1} \not\subseteq \Limm0kk$ by \cref{thm:fkt}) that the clones $\Incr\Cons0k(\Cons11)^\Borel = \Incr\Limm0k{k-1}(\Cons11)^\Borel$ are not contained in $\Limm0kk$, as shown in \cref{fig:post-borel-L0kt-T03}.

However, in contrast to the situation with $\Limm111$ (\cref{thm:post-T0k-lim1-mod}), here we do not know if \emph{every} subclone of $\Mono\Limm0k{k-1}^\Borel$ outside $\Limm0kk$ corresponds to a subclone of $\Incr\Limm0k{k-1}^\Borel$.
Nor do we have a candidate generating set for $\Incr\Limm0k{k-1}^\Borel$, or know if there is a single function generating a minimal subclone of $\Incr\Limm0k{k-1}$ outside $\Limm0kk$ (analogous to $\liminf^{\sqcup k} \vee (\pi_0 \wedge \dotsb \wedge \pi_{k-1})$ which is minimal in $\Mono\Limm0k{k-1} \setminus \Limm0kk$ from \cref{thm:fkt-lim1-gen}).
An example of such a function is given by

\begin{example}
The indicator function $g : 2^\omega -> 2$ of the set of all strings $\ge$ one of
\begin{align*}
&11\;00\;00\;00\,\dotsb, \\[1ex]
&10\;11\;00\;00\,\dotsb, \\
&01\;11\;00\;00\,\dotsb, \\[1ex]
&10\;10\;11\;00\,\dotsb, \\
&01\;01\;11\;00\,\dotsb, \\
&\vdotswithin{00\;00\;00\;00\,\dotsb}
\end{align*}
is in $\Incr\Limm021\Cons11^\Borel \setminus \Limm022$ (the strings $101010\dotsb, 010101\dotsb \in \-{g^{-1}(1)}$ are disjoint).
Note that this is an example of an upper bound for $\liminf \sqcup \liminf \in \Limm021$, as provided by \cref{thm:post-T0k-incr-down}.
\end{example}

The following summarizes all of the positive classification results we have obtained for Borel clones lying over the left ``side tube'' of Post's lattice \ref{fig:post}, from this subsection and \cref{sec:borel-T0inf}:

\begin{theorem}
\label{thm:post-borel-T0k}
For a Borel clone $\Mono\Cons0{<\omega}\Cons11^{<\omega} = \ang{\ceil{\vee}}^{<\omega} \subseteq F \subseteq \Op2^\Borel$, one of the following holds.
These cases can all occur, and are mutually exclusive except for the boundary cases indicated in \cref{thm:post-borel-T0k:T0omega}.
\let\displaystyle=\textstyle
\begin{enumerate}[label=(\alph*)]

\item \label{thm:post-borel-T0k:L0kt}
For some coordinatewise downward-closed $\emptyset \ne D \subseteq \{(k,t) \mid t \le k < \omega\}$, we have one of
\begin{alignat*}{2}
\Cons0\omega^\Borel \subseteq
\ang{\{\bigwedge, -/>\} \cup
\{\forall^{k-t}_2 \sqcup \liminf^{\sqcup t} \mid (k,t) \notin D\}}^\Borel
&\subseteq F \subseteq \bigcap_{(k,t) \in D} \Limm0kt,
\\
\Cons0\omega\Cons11^\Borel \subseteq
\ang{\{\bigwedge, \ceil{->}\} \cup
\{\forall^{k-t}_2 \sqcup \liminf^{\sqcup t} \mid (k,t) \notin D\}}^\Borel
&\subseteq F \subseteq \Cons11 \cap \bigcap_{(k,t) \in D} \Limm0kt,
\\
\Mono\Cons0\omega^\Borel \subseteq
\ang{\{\bigwedge, \ceil\bigvee, 0\} \cup
\{\forall^{k-t}_2 \sqcup \liminf^{\sqcup t} \mid (k,t) \notin D\}}^\Borel
&\subseteq F \subseteq \Mono \cap \bigcap_{(k,t) \in D} \Limm0kt, \text{ or}
\\
\Mono\Cons0\omega\Cons11^\Borel \subseteq
\ang{\{\bigwedge, \ceil\bigvee\} \cup
\{\forall^{k-t}_2 \sqcup \liminf^{\sqcup t} \mid (k,t) \notin D\}}^\Borel
&\subseteq F \subseteq \Mono\Cons11 \cap \bigcap_{(k,t) \in D} \Limm0kt
\end{alignat*}
(\cref{thm:post-L0kt-partition,thm:post-borel-T0omega}).
The first two of these intervals of $F$ are isomorphic, as are the latter two (\cref{thm:post-T1-mod}), which order-embed into the first two (\cref{thm:post-borel-down-emb}); the image includes all Borel clones defined as $\Pol$ of countably many Borel downward-closed relations $R \subseteq 2^\omega$ (\cref{thm:post-borel-up}).
Moreover, we have exactly one of the following subcases.
\begin{enumerate}[label=(\roman*)]
\item \label{thm:post-borel-T0k:L0kt:L0kt}
There is a greatest $k < \omega$ for which $(k,0) \in D$, in which case $\forall^{k+1}_2 \in F \subseteq \Cons0k$, so $F$ has finitary restriction $\Cons0k^{<\omega}, \Cons0k\Cons11^{<\omega}, \Mono\Cons0k^{<\omega}$ or $\Mono\Cons0k\Cons11^{<\omega}$ respectively.
There are $2^k$ of each of the four types of intervals above in this case (\cref{thm:post-borel-L0kt-2^k}).
\item \label{thm:post-borel-T0k:L0kt:L0inft}
We have $(k,0) \in D$ for all $k < \omega$ (so $F \subseteq \Limm0{<\omega}{{0}} = \Cons0{<\omega}$), and $\forall_2 = \forall^\omega_2 \in F$.
There are countably infinitely many of each of the four intervals above in this case (\cref{thm:post-borel-L0inf-inf}), the lowest one being $[\ang{\forall_2, \bigwedge, \ceil{\bigvee}}^\Borel, \Mono\Limm0{<\omega}{}\Limm111^\Borel]$ when $D = \{(k,t) \mid t \le k < \omega\}$.
\item \label{thm:post-borel-T0k:L0kt:T0omega}
We have $\forall_2 \not\in F$, in which case $F = \Cons0\omega^\Borel, \Cons0\omega\Cons11^\Borel, \Mono\Cons0\omega^\Borel$ or $\Mono\Cons0\omega\Cons11^\Borel$ (\cref{thm:post-T0inf-forall2}).
\end{enumerate}
(These are the ``primary cobwebs'' in \cref{fig:post-borel-L0kt-T03,fig:post-borel-T0k}.)

\item \label{thm:post-borel-T0k:L0ktL1}
For some $t < k < \omega$, $F$ has finitary restriction $\Cons0k\Cons11^{<\omega}$ or $\Mono\Cons0k\Cons11^{<\omega}$, and obeys respectively
\begin{alignat*}{2}
\qquad\quad\;
\mathllap{\Limm0kk\Limm111^\Borel \subseteq{}}
\ang{\forall^{k+1}_2, \ceil{\bigvee}, \ceil{->},
(\forall^{k-t-1}_2 \sqcup \liminf^{\sqcup (t+1)}) \vee (\pi_0 \wedge \dotsb \wedge \pi_{k-1})}^\Borel
&\subseteq F \subseteq \Limm0kt\Limm0{k-1}{}\Limm111 \text{ or}
\\
\mathllap{\Mono\Limm0kk\Limm111^\Borel \subseteq{}}
\ang{\forall^{k+1}_2, \ceil{\floor{\bigwedge}}, \ceil{\bigvee},
(\forall^{k-t-1}_2 \sqcup \liminf^{\sqcup (t+1)}) \vee (\pi_0 \wedge \dotsb \wedge \pi_{k-1})}^\Borel
&\subseteq F \subseteq \Mono\Limm0kt\Limm0{k-1}{}\Limm111
\end{alignat*}
(\cref{thm:fkt-lim1-gen}).
These intervals are isomorphic to the second and fourth intervals in \cref{thm:post-borel-T0k:L0kt}\cref{thm:post-borel-T0k:L0kt:L0kt}, for $(k,0), (k-1,k-1) \in D \notni (k+1,0), (k,k)$ (\cref{thm:post-T0k-lim1-mod}).
(These are the ``secondary rows'' below the ``cobwebs'' in \cref{fig:post-borel-L0kt-T03,fig:post-borel-T0k}.)

\item \label{thm:post-borel-T0k:L0ktIncr}
For some $0 < k < \omega$, $F$ has finitary restriction $\Mono\Cons0k^{<\omega}$ or $\Mono\Cons0k\Cons11^{<\omega}$, and obeys respectively
\begin{align*}
\Incr\Limm0kk^\Borel =
\ang{\forall^{k+1}_2, \ceil{\bigvee}, 0}^\Borel
&\subsetneqq F \subseteq \Incr\Cons0k^\Borel = \Incr\Limm0k{k-1}^\Borel \text{ or}
\\
\Incr\Limm0kk\Cons11^\Borel =
\ang{\forall^{k+1}_2, \ceil{\bigvee}}^\Borel
&\subsetneqq F \subseteq \Incr\Cons0k\Cons11^\Borel = \Incr\Limm0k{k-1}\Limm111^\Borel.
\end{align*}
These two sets of $F$ are isomorphic to each other (\cref{thm:post-T0k-lim1-mod}), and embed into the third and fourth intervals in \cref{thm:post-borel-T0k:L0kt}\cref{thm:post-borel-T0k:L0kt:L0kt} for $D = \{(k',t') \mid (k',t') \le (k,k-1)\}$ (\cref{thm:post-borel-down-emb}) with cofinal image (\cref{thm:post-T0k-incr-down}).
(These are the ``tertiary'' nodes in \cref{fig:post-borel-L0kt-T03,fig:post-borel-T0k}.)

\item \label{thm:post-borel-T0k:forall2}
$F$ has finitary restriction $\Mono\Cons0{<\omega}^{<\omega}$ or $\Mono\Cons0{<\omega}\Cons11^{<\omega}$, and obeys respectively
\begin{align*}
\ang{\forall_2, 0}^\Borel &\subseteq F \subseteq \Decr\Cons0{<\omega}^\Borel \text{ or} \\
\ang{\forall_2}^\Borel &\subseteq F \subseteq \Decr\Cons0{<\omega}\Cons11^\Borel.
\end{align*}
These two intervals are isomorphic (\cref{thm:post-T1-mod}), and they embed into the respective intervals below $\Limm0{<\omega}{}$ in \cref{thm:post-borel-T0k:L0kt}\cref{thm:post-borel-T0k:L0kt:L0inft} (\cref{thm:post-borel-down-emb}) with cofinal image (\cref{thm:post-borel-L0inf-down}).

\item \label{thm:post-borel-T0k:T0omega}
$F$ is equal to one of the precisely 2, 3, 4, 6 Borel clones contained in $\Limm0kk$ restricting to the finitary clones
$\Cons0k^{<\omega},
\Cons0k\Cons11^{<\omega},
\Mono\Cons0k^{<\omega},
\Mono\Cons0k\Cons11^{<\omega}$ respectively, for some $0 \le k < \omega$ (\cref{thm:post-borel-L0k}) or $k = \omega$ (\cref{thm:post-borel-T0omega}; by convention $\Limm0\omega\omega := \Cons0\omega$).
These are all isomorphic as $k$ varies, and are obtained by adjoining the generator $\forall^{k+1}_2$ to the respective clone below $\Cons0\omega$ from \cref{fig:post-borel-T0inf} for $k < \omega$, or alternately by intersecting $\Limm0kk$ with the respective clone above $\ang{\Mono\Cons11^{<\omega}}^\Borel$ in the top cube from \cref{sec:borel-topcube} (\cref{thm:post-T0inf-mod}).

For each $k$, and in each of the four cases ($\Mono$ and/or $\Cons11$), the top clone $(\Mono)\Limm0kk(\Cons11)^\Borel$ here coincides with one of the extreme cases from \cref{thm:post-borel-T0k:L0kt}\cref{thm:post-borel-T0k:L0kt:L0kt} or \cref{thm:post-borel-T0k:L0kt:T0omega}, where $(k,k) \in D \notni (k+1,0)$ (or $D = \{(k,t) \mid t \le k < \omega\}$, when $k = \omega$).
\qed
\end{enumerate}
\end{theorem}

Note that in each of the ``imprecise'' cases \cref{thm:post-borel-T0k:L0kt}--\cref{thm:post-borel-T0k:forall2} above, \emph{except for \cref{thm:post-borel-T0k:L0ktIncr}}, we have examples of ``minimal'' functions contained in all clones in the interval.
If these can be shown to generate the upper bound of the interval, then the entire interval would collapse.
We know this to happen only in the ``bottommost'' cases covered by \cref{thm:post-borel-T0k:T0omega}, in the top cube cases $k \le 1$ from \cref{sec:borel-topcube}, and in a few sporadic $k = 2$ cases; see below.
Also, in each of the ``topmost'' cases $(1,1) \not\in D$ in \cref{thm:post-borel-T0k:L0kt}, yielding an interval $\forall^{1-1}_2 \sqcup \liminf^{\sqcup 1} = \liminf \in F \subseteq \Cons0k$ (or $\subseteq \Cons0{<\omega}$ when $k = \omega$), we know by \cref{thm:post-ctcl-liminf} that the interval collapses \emph{up to taking countable closure}, i.e., every ${<}\omega_1$-ary function in $\Cons0k$ can be approximated at any countably many inputs by a Borel function built from the candidate generators.

We conclude by relating the partial classification of Borel clones over $\Mono\Cons02^{<\omega}$ given by the above, to the Borel self-dual clones restricting to $\Dual\Mono^{<\omega}$, the last of the 3 finitary self-dual, non-affine clones \cref{eq:post-D-beta-clones} which corresponds to $\Mono\Cons02^{<\omega}$ via the self-dualizing operator $\beta$ (\cref{thm:post-D-beta-mod}).

Recall that for the other 2 such finitary self-dual clones, $\Dual^{<\omega}$ and $\Dual\Cons01^{<\omega}$, the Borel clones were fully classified using $\beta$ in \cref{thm:post-borel-D}; recall from there the definition of $\beta(\bigwedge) \in \Dual\Limm011$.
Note that, in terms of \cref{def:fkt}, $\beta(\bigwedge)$ may be written as
\begin{align*}
\textstyle
\beta(\bigwedge)
= 1 \sqcup \bigvee
\ge 1 \sqcup \liminf.
\end{align*}
It follows by \cref{thm:post-lower} that $1 \sqcup \liminf \in \ang{\beta(\bigwedge), 0}$.
It is also easily seen that
\begin{align*}
\textstyle
\yesnumber
\label{eq:betaMeet-liminf}
\beta(\bigwedge)(\vec{x})
&= (1 \sqcup \liminf)(x_0, x_1, \exists^3_2(x_0, x_1, x_2), \exists^3_2(x_0, x_1, \exists^3_2(x_0, x_2, x_3)), \dotsc), \\
\exists^3_2(x, y, z)
&= \textstyle \beta(\bigwedge)(x, y, z, y, z, y, z, \dotsc).
\end{align*}
Thus
\begin{equation*}
\ang{\beta(\bigwedge), 0}^\Borel = \ang{1 \sqcup \liminf, \exists^3_2, 0}^\Borel
\end{equation*}
are two equivalent versions of the candidate generators for $\Mono\Cons02\Limm011^\Borel$ given by \cref{thm:post-L0kt-partition} (for the downward-closed set of indices $D = \{(2,0),(1,0),(1,1),(0,0)\}$).

\begin{proposition}
\label{thm:post-borel-MT02L0}
$\Mono\Cons02\Limm011^\Borel = \ang{\beta(\bigwedge), 0}^\Borel$.
\end{proposition}
\begin{proof}
Let $f \in \Mono\Cons02\Limm011^\Borel$; then $f \le \pi_0 \vee \dotsb \vee \pi_{m-1}$ for some $m < \omega$, and we must show $f \in \ang{\beta(\bigwedge), 0}$.
We induct on $m$.
If $m = 0$, then $f = 0$.
Now suppose the claim holds for $m$, and $f \le \pi_0 \vee \dotsb \vee \pi_m$.
Then (the cross-section) $f_0 \le \pi_0 \vee \dotsb \vee \pi_{m-1}$, so $f_0 \in \ang{\beta(\bigwedge), 0}^\Borel$, and so $\beta(f_0) \in \ang{\beta(\bigwedge), 0}$, since $\ang{\beta(\bigwedge), 0}$ is the clone containing $0$ corresponding via \cref{thm:post-D-beta-mod} to $\ang{\beta(\wedge)} \subseteq \Dual$, hence is closed under $\beta$.
But since $f \in \Mono\Cons02$, we have $f \le \beta(f_0)$, thus $f \in \ang{\beta(\bigwedge), 0}$ by \cref{thm:post-lower}.
\end{proof}

\begin{corollary}
\label{thm:post-borel-T02}
We have:
\begin{enumerate}[label=(\alph*)]
\item \label{thm:post-borel-T02:L0}
$\Cons02\Limm011^\Borel = \ang{\beta(\bigwedge), -/>}^\Borel$,
$\Cons02\Limm011\Cons11^\Borel = \ang{\beta(\bigwedge), \bigwedge, \ceil{->}}^\Borel$, and
$\Mono\Cons02\Limm011\Cons11^\Borel = \ang{\beta(\bigwedge), \bigwedge}^\Borel$.
\item \label{thm:post-borel-T02:L0L1}
$\Cons02\Limm011\Limm111^\Borel = \ang{\beta(\bigwedge), \ceil{->}}^\Borel$ and
$\Mono\Cons02\Limm111^\Borel = \ang{\beta(\bigwedge), \wedge}^\Borel$.
\item \label{thm:post-borel-T02:liminf}
$\Cons02^\Borel = \ang{\liminf, \exists^3_2, -/>}^\Borel$ and
$\Mono\Cons02^\Borel = \ang{\liminf, \exists^3_2, 0}^\Borel$.
\end{enumerate}
\end{corollary}
\begin{proof}
The first equality in \cref{thm:post-borel-T02:L0} follows from \cref{thm:post-borel-up,thm:post-lower,thm:post-borel-MT02L0};
the rest then follows from \cref{thm:post-T1-mod}.

\Cref{thm:post-borel-T02:L0L1} follows from \cref{thm:post-T0k-lim1-mod}.

\Cref{thm:post-borel-T02:liminf}:
Note first that from $\liminf, \exists^3_2$, we easily get $1 \sqcup \liminf$, via a formula similar to \cref{eq:betaMeet-liminf}, which then yields $\beta(\bigwedge)$; so the two generated clones in question contain $\beta(\bigwedge)$ and $-/>, 0$ respectively, hence contain all of $(\Mono)\Cons02\Limm011^\Borel$ by \cref{thm:post-borel-MT02L0} and \cref{thm:post-borel-T02:L0}.
Now for $f \in (\Mono)\Cons02^\Borel$, we have $f = \liminf_{m -> \infty} (f \wedge (\pi_0 \vee \dotsb \vee \pi_{m-1}))$ with each $f \wedge (\pi_0 \vee \dotsb \vee \pi_{m-1}) \in (\Mono)\Cons02\Limm011^\Borel$.
\end{proof}

Thus among the Borel clones with finitary restriction $\Cons02^{<\omega}$ and its variants, we have the curious situation (see \cref{fig:post-borel-T0k}) that sufficiently small (below $\Limm022$) or large (above $\Mono\Cons02\Limm011\Limm111$) such clones are fully classified (solid-shaded regions in \ref{fig:post-borel-T0k}), whereas the intermediate clones, $\Limm021$ and its variants, remain open (hatch-shaded regions in \ref{fig:post-borel-T0k}).
Nonetheless, this is enough to yield

\begin{corollary}
\label{thm:post-borel-DM}
There are precisely 3 Borel clones on $2$ restricting to $\Dual\Mono^{<\omega}$ (see \cref{fig:post-borel-D}):
\begin{itemize}
\item  $\Dual\Mono^\Borel = \Pol^\Borel \{\neg, \le\} = \ang{\beta(\liminf)}^\Borel$,
\item  $\Dual\Mono\Limm011^\Borel = \Dual\Mono\Limm111^\Borel = \Pol^\Borel \{\neg, \le, {\lim}{=}0\} = \ang{\beta(\bigwedge)}^\Borel$,
\item  $\ang{\Dual\Mono^{<\omega}}^\Borel = \Pol^\Borel \{\neg, \le, \lim\} = \ang{\exists^3_2}^\Borel$,
\end{itemize}
where
\begin{equation*}
\beta(\liminf)(x_0, x_1, \dotsc) = \paren[\big]{x_0 \wedge \limsup_{i -> \infty} x_i} \vee \liminf_{i -> \infty} x_i.
\end{equation*}
\end{corollary}
\begin{proof}
By \cref{thm:post-D-beta-mod} (and the fact that $\Dual\Mono^{<\omega} = \ang{\beta(\Mono\Cons02^{<\omega})}^{<\omega}$), such Borel clones are $\ang{\beta(G)}^\Borel$ for all Borel clones $G$ restricting to $\Mono\Cons02^{<\omega}$ which are closed under $\beta$.
If $G$ contains a discontinuous function $g$, then by \cref{thm:mono-lim}, either $g \notin \Incr$, or $g \notin \Decr$ in which case clearly from \cref{def:beta}, $\beta(g) \notin \Incr$; thus by \cref{thm:post-T0omega-dich}, $\bigwedge \in G$, so $\beta(\bigwedge) \in G$, so by \cref{thm:post-borel-MT02L0}, $\Mono\Cons02\Limm011^\Borel \subseteq G$, and so by \cref{thm:kahane,thm:post-borel-T02}\cref{thm:post-borel-T02:liminf}, either $G = \Mono\Cons02\Limm011^\Borel$ or $G = \Mono\Cons02^\Borel$.
So
\begin{align*}
\Mono\Cons02^\Borel &= \ang{\liminf, \exists^3_2, 0}^\Borel, &
\Mono\Cons02\Limm011^\Borel &= \ang{\beta(\bigwedge), 0}^\Borel, &
\ang{\Mono\Cons02^{<\omega}}^\Borel &= \ang{\exists^3_2, 0}^\Borel
\end{align*}
are the 3 Borel clones lying over $\Mono\Cons02^{<\omega}$ which are closed under $\beta$; these are easily seen to correspond via \cref{thm:post-D-beta-mod} to the above clones lying over $\Dual\Mono^{<\omega}$.
\end{proof}

% remove MR number from references
\def\MR#1{}
\bibliographystyle{amsalpha}
\bibliography{refs}

\newcommand{\etalchar}[1]{$^{#1}$}
\providecommand{\bysame}{\leavevmode\hbox to3em{\hrulefill}\thinspace}
\providecommand{\MR}{\relax\ifhmode\unskip\space\fi MR }
\providecommand{\MRhref}[2]{%
  \href{http://www.ams.org/mathscinet-getitem?mr=#1}{#2}
}
\providecommand{\href}[2]{#2}
\begin{thebibliography}{FMMT22}

\bibitem[BKKR69]{BKKR}
V.~G. Bodnarčuk, L.~A. Kalužnin, V.~N. Kotov, and B.~A. Romov, \emph{Galois theory for {P}ost algebras. {I}, {II}}, Kibernetika (Kiev) (1969), no.~3, 1--10; ibid. {\bf 1969}, no. 5, 1--9. \MR{300895}

\bibitem[Bod21]{Bodirsky}
Manuel Bodirsky, \emph{Complexity of infinite-domain constraint satisfaction}, Lecture Notes in Logic, vol.~52, Cambridge University Press, Cambridge; Association for Symbolic Logic, Ithaca, NY, 2021. \MR{4273453}

\bibitem[Bul17]{Bulatov}
Andrei~A. Bulatov, \emph{A dichotomy theorem for nonuniform {CSP}s}, 58th {A}nnual {IEEE} {S}ymposium on {F}oundations of {C}omputer {S}cience---{FOCS} 2017, IEEE Computer Soc., Los Alamitos, CA, 2017, pp.~319--330. \MR{3734240}

\bibitem[Dou88]{Dougherty}
Randall Dougherty, \emph{Monotone reducibility over the {C}antor space}, Trans. Amer. Math. Soc. \textbf{310} (1988), no.~2, 433--484. \MR{943302}

\bibitem[FMMT22]{FMMT}
Ralph~S. Freese, Ralph~N. McKenzie, George~F. McNulty, and Walter~F. Taylor, \emph{Algebras, lattices, varieties. {V}ol. {II}}, Mathematical Surveys and Monographs, vol. 268, American Mathematical Society, Providence, RI, 2022. \MR{4496007}

\bibitem[Gei68]{Geiger}
David Geiger, \emph{Closed systems of functions and predicates}, Pacific J. Math. \textbf{27} (1968), 95--100. \MR{234893}

\bibitem[GHK{\etalchar{+}}03]{GHKLMS}
G.~Gierz, K.~H. Hofmann, K.~Keimel, J.~D. Lawson, M.~Mislove, and D.~S. Scott, \emph{Continuous lattices and domains}, Encyclopedia of Mathematics and its Applications, vol.~93, Cambridge University Press, Cambridge, 2003. \MR{1975381}

\bibitem[Gr{\"a}03]{Gratzer}
George Gr{\"a}tzer, \emph{General lattice theory}, second ed., Birkh\"{a}user Verlag, Basel, 2003, With appendices by B. A. Davey, R. Freese, B. Ganter, M. Greferath, P. Jipsen, H. A. Priestley, H. Rose, E. T. Schmidt, S. E. Schmidt, F. Wehrung and R. Wille. \MR{2451139}

\bibitem[Hjo00]{Hjorth}
Greg Hjorth, \emph{Classification and orbit equivalence relations}, Math. Surveys Monogr., vol.~75, American Mathematical Society, Providence, RI, 2000.

\bibitem[Hru11]{Hrusak}
Michael Hru\v{s}\'{a}k, \emph{Combinatorics of filters and ideals}, Set theory and its applications, Contemp. Math., vol. 533, Amer. Math. Soc., Providence, RI, 2011, pp.~29--69. \MR{2777744}

\bibitem[Joh82]{Jstone}
Peter~T. Johnstone, \emph{Stone spaces}, Cambridge Studies in Advanced Mathematics, vol.~3, Cambridge University Press, Cambridge, 1982. \MR{698074}

\bibitem[Kah92]{Kahane}
Sylvain Kahane, \emph{Op\'{e}rations de {H}ausdorff it\'{e}r\'{e}es et r\'{e}unions croissantes de compacts}, Fund. Math. \textbf{141} (1992), no.~2, 169--194. \MR{1183330}

\bibitem[Kan08]{Kanovei}
Vladimir Kanovei, \emph{Borel equivalence relations}, University Lecture Series, vol.~44, American Mathematical Society, Providence, RI, 2008, Structure and classification. \MR{2441635}

\bibitem[Kec95]{Kcdst}
Alexander~S. Kechris, \emph{Classical descriptive set theory}, Graduate Texts in Mathematics, vol. 156, Springer-Verlag, New York, 1995. \MR{1321597}

\bibitem[Kel75]{Kelley}
John~L. Kelley, \emph{General topology}, Graduate Texts in Mathematics, vol. No. 27, Springer-Verlag, New York-Berlin, 1975, Reprint of the 1955 edition [Van Nostrand, Toronto, Ont.]. \MR{370454}

\bibitem[KR03]{KanoveiReeken}
V.~G. Kanove\u{\i} and M.~Reeken, \emph{Some new results on the {B}orel irreducibility of equivalence relations}, Izv. Ross. Akad. Nauk Ser. Mat. \textbf{67} (2003), no.~1, 59--82. \MR{1957916}

\bibitem[KTV23]{KatayTothVidnyanszky}
Tamás Kátay, László~Márton Tóth, and Zoltán Vidnyánszky, \emph{The {CSP} dichotomy, the axiom of choice, and cyclic polymorphisms}, preprint (2023), \url{https://arxiv.org/abs/2310.00514}.

\bibitem[Lau06]{Lau}
Dietlinde Lau, \emph{Function algebras on finite sets: A basic course on many-valued logic and clone theory}, Springer Monographs in Mathematics, Springer-Verlag, Berlin, 2006. \MR{2254622}

\bibitem[Lou83]{Louveau}
A.~Louveau, \emph{Some results in the {W}adge hierarchy of {B}orel sets}, Cabal seminar 79--81, Lecture Notes in Math., vol. 1019, Springer, Berlin, 1983, pp.~28--55. \MR{730585}

\bibitem[Mos09]{Mdst}
Yiannis~N. Moschovakis, \emph{Descriptive set theory}, second ed., Mathematical Surveys and Monographs, vol. 155, American Mathematical Society, Providence, RI, 2009. \MR{2526093}

\bibitem[Pos41]{Post}
Emil~L. Post, \emph{The {T}wo-{V}alued {I}terative {S}ystems of {M}athematical {L}ogic}, Annals of Mathematics Studies, vol. No. 5, Princeton University Press, Princeton, NJ, 1941. \MR{4195}

\bibitem[Ros09]{Racts}
Christian Rosendal, \emph{Automatic continuity of group homomorphisms}, Bull. Symbolic Logic \textbf{15} (2009), no.~2, 184--214. \MR{2535429}

\bibitem[Sol99]{Solecki}
Sławomir Solecki, \emph{Analytic ideals and their applications}, Ann. Pure Appl. Logic \textbf{99} (1999), no.~1-3, 51--72. \MR{1708146}

\bibitem[Sze86]{Szendrei}
\'{A}gnes Szendrei, \emph{Clones in universal algebra}, S\'{e}minaire de Math\'{e}matiques Sup\'{e}rieures [Seminar on Higher Mathematics], vol.~99, Presses de l'Universit\'{e} de Montr\'{e}al, Montreal, QC, 1986. \MR{859550}

\bibitem[Tho22]{Thornton}
Riley Thornton, \emph{An algebraic approach to {B}orel {CSP}s}, preprint (2022), \url{https://arxiv.org/abs/2203.16712}.

\bibitem[Wad83]{Wadge}
William~Wilfred Wadge, \emph{Reducibility and determinateness on the {B}aire space}, ProQuest LLC, Ann Arbor, MI, 1983, Thesis (Ph.D.)--University of California, Berkeley. \MR{2633374}

\bibitem[Zhu17]{Zhuk}
Dmitriy Zhuk, \emph{A proof of {CSP} dichotomy conjecture}, 58th {A}nnual {IEEE} {S}ymposium on {F}oundations of {C}omputer {S}cience---{FOCS} 2017, IEEE Computer Soc., Los Alamitos, CA, 2017, pp.~331--342. \MR{3734241}

\end{thebibliography}

\medskip
\noindent
Department of Mathematics\\
University of Michigan\\
Ann Arbor, MI, USA\\
\nolinkurl{ruiyuan@umich.edu}\\
\nolinkurl{zibai@umich.edu}

\end{document}